\documentclass[11pt, a4paper, leqno]{amsart}

\usepackage{amsmath,amsthm,amssymb} 
\usepackage{amsfonts} 
\usepackage{mathrsfs} 
\usepackage[T1]{fontenc} 
\usepackage[utf8]{inputenc} 

\usepackage[dvipsnames, svgnames]{xcolor} 
\usepackage{enumerate} 
\usepackage{breakcites}
\usepackage[colorlinks,citecolor=citec,hypertexnames=false,urlcolor=urlc,breaklinks=true, pagebackref, breaklinks=true]{hyperref} 
\usepackage{thmtools} 
\usepackage{accents} 

\usepackage{amsbsy} 
\usepackage{amscd} 
\usepackage{latexsym} 
\usepackage{txfonts} 
\usepackage{exscale} 
\usepackage{bbm} 

\let\ii\i 

\usepackage{geometry} 

\newif\ifsoldark
\newif\ifsollight
\newif\ifclassic
\newif\ifplain

\usepackage{cite} 

\renewcommand*\backref[1]{\ifx#1\relax \else (p. #1) \fi} 

\numberwithin{equation}{section}

\declaretheoremstyle[
headfont=\color{lmcolor}\normalfont\bfseries,
bodyfont=\normalfont\itshape,
]{colorlemma}
\theoremstyle{colorlemma}
\newtheorem{lemma}[equation]{Lemma}

\newtheorem{claim}[equation]{Claim}

\declaretheoremstyle[
headfont=\color{propcolor}\normalfont\bfseries,
bodyfont=\normalfont\itshape,
]{colorprop}
\theoremstyle{colorprop}
\newtheorem{proposition}[equation]{Proposition}
\newtheorem{corollary}[equation]{Corollary}

\declaretheoremstyle[
headfont=\color{thmcolor}\normalfont\bfseries,
bodyfont=\normalfont\itshape,
]{colorthm}
\theoremstyle{colorthm}
\newtheorem{theorem}[equation]{Theorem}

\declaretheoremstyle[
headfont=\color{defcolor}\normalfont\bfseries,
bodyfont=\normalfont,
]{colordef}
\theoremstyle{colordef}
\newtheorem{definition}[equation]{Definition}

\declaretheoremstyle[
headfont=\color{excolor}\normalfont\bfseries,
bodyfont=\normalfont,
]{colorexample}
\theoremstyle{colorexample}

\declaretheoremstyle[
headfont=\color{rmcolor}\normalfont\itshape,
bodyfont=\normalfont,
]{colorremark}
\theoremstyle{colorremark}
\newtheorem{remark}[equation]{Remark}
\newtheorem*{remark*}{Remark}
\newtheorem{remarks}[equation]{Remarks}

\makeatletter
\renewcommand{\eqref}[1]{\textup{\eqreftagform@{\ref{#1}}}}
\let\eqreftagform@\tagform@
\def\tagform@#1{%
	\maketag@@@{\color{eqcolor}(\ignorespaces#1\unskip\@@italiccorr)}%
}
\makeatother

\newcommand{\RR}{{\mathbb{R}}}
\newcommand{\eps}{\varepsilon}

\newcommand{\dint}{\int\!\!\!\!\!\int}
\newcommand{\divg}{\text{\normalfont{div}}}
\newcommand{\dist}{\operatorname{dist}}
\DeclareMathOperator{\divp}{\div_{\|}}
\newcommand{\CC}{{\mathbb{C}}}

\newcommand{\cl}[1]{\overline{#1}} 
\newcommand{\dv}{\operatorname{div}}
\newcommand{\n}[1]{\mathscr{#1}}
\newcommand{\re}{\mathbb{R}}
\newcommand{\rn}{{\mathbb{R}^n}}

\newcommand{\reu}{\mathbb{R}^{n+1}_+}
\newcommand{\ree}{\mathbb{R}^{n+1}}
\newcommand{\N}{\mathbb{N}}
\newcommand{\dd}{\mathbb{D}}
\newcommand{\bbm}[1]{\mathbbm{#1}}
\newcommand{\bb}[1]{\mathbb{#1}}
\newcommand{\C}{\mathcal{C}}

\newcommand{\om}{\Omega}

\newcommand{\F}{\mathcal{F}}
\newcommand{\cL}{\mathcal{L}}
\newcommand{\A}{\mathcal{A}}

\newcommand{\B}{\mathcal{B}}

\newcommand{\s}{\mathcal{S}}

\newcommand{\yot}{Y^{1,2}(\ree)}

\newcommand{\yotp}{Y^{1,2}(\reu)}
\newcommand{\yotn}{{Y^{1,2}(\rn)}}
\newcommand{\yota}{{Y^{1,2}(\Sigma_a^b)}}

\newcommand{\sltp}{S^2_+}
\newcommand{\sltn}{S^2_-}
\newcommand{\dltp}{D^2_+}
\newcommand{\dltn}{D^2_-}
\newcommand{\cD}{\mathcal{D}}

\newcommand{\lnrn}{{L^n(\rn)}}
\newcommand{\ltrn}{{L^2(\rn)}}

\newcommand{\nablap}{\nabla_{\|}}

\newcommand{\dno}{D_{n+1}}

\newcommand{\sab}{{\Sigma_a^b}}

\newcommand{\hm}{\omega}

\newcommand{\cee}{\mathbb{C}^{n+1}}

\newcommand{\RNum}[1]{\uppercase\expandafter{\romannumeral #1\relax}}
\newcommand{\tr}{\operatorname{Tr}}
\newcommand{\Tr}{\text{\normalfont Tr}}
\newcommand{\qs}{\m Q_s}

\newcommand{\tT}{\widetilde{\Theta}}

\DeclareMathOperator{\supp}{supp}

\DeclareMathOperator{\tro}{Tr_0}
\DeclareMathOperator{\trt}{Tr_{\mathnormal t}}

\def\div{\mathop{\operatorname{div}}\nolimits}

\def\Xint#1{\mathchoice
	{\XXint\displaystyle\textstyle{#1}}%
	{\XXint\textstyle\scriptstyle{#1}}%
	{\XXint\scriptstyle\scriptscriptstyle{#1}}%
	{\XXint\scriptscriptstyle\scriptscriptstyle{#1}}%
	\!\int}
\def\XXint#1#2#3{{\setbox0=\hbox{$#1{#2#3}{\int}$}
		\vcenter{\hbox{$#2#3$}}\kern-.5\wd0}}

\def\dashint{\Xint-}

\def\Yint#1{\mathchoice
	{\YYint\displaystyle\textstyle{#1}}%
	{\YYint\textstyle\scriptstyle{#1}}%
	{\YYint\scriptstyle\scriptscriptstyle{#1}}%
	{\YYint\scriptscriptstyle\scriptscriptstyle{#1}}%
	\!\dint}
\def\YYint#1#2#3{{\setbox0=\hbox{$#1{#2#3}{\iint}$}
		\vcenter{\hbox{$#2#3$}}\kern-.51\wd0}}
\def\longdash{\mkern-3.5mu{-}\mkern-3.5mu{-}}  

\def\fiint{\Yint\longdash}

\newcommand{\ep}{\varepsilon}
\newcommand{\ra}{\rightarrow}

\newcommand{\m}[1]{\mathcal{#1}}

\newcommand{\f}[1]{\mathfrak{#1}}
\newcommand{\vertiii}[1]{{\left\vert\kern-0.25ex\left\vert\kern-0.25ex\left\vert #1 
		\right\vert\kern-0.25ex\right\vert\kern-0.25ex\right\vert}}

\newcommand{\loc}{\operatorname{loc}}

\newcommand{\real}{\operatorname{\Re e\,}}

\newcommand{\Ya}{Y^{1,2}(\mathbb R^{n+1})}
\newcommand{\Y}{Y^{1,2}(\mathbb R^{n+1}_+)}
\newcommand{\Ym}{Y^{1,2}(\mathbb R^{n+1}_-)}
\newcommand{\Hf}{H^{\frac12}_0(\mathbb R^n)}
\newcommand{\Hfm}{H^{-\frac12}(\mathbb R^n)}

\newcommand*{\dt}[1]{%
	\accentset{\mbox{\Large\bfseries .}}{#1}}

\allowdisplaybreaks

\plaintrue 

	\ifplain
	
	\colorlet{citec}{blue}
	\colorlet{urlc}{blue}
	\colorlet{toc}{red}
	\colorlet{hyperc}{blue}
	\colorlet{bpcolor}{OliveGreen}
	\colorlet{smcolor}{purple}
	\colorlet{jllgcolor}{Red} 
	\colorlet{shcolor}{Orange} 
	\colorlet{sbcolor}{NavyBlue} 
	\colorlet{impcolor}{ProcessBlue}
	\colorlet{eqcolor}{blue}
	\colorlet{lmcolor}{black}
	\colorlet{propcolor}{black}
	\colorlet{thmcolor}{black}
	\colorlet{defcolor}{black}
	\colorlet{rmcolor}{black}
	\colorlet{excolor}{black}
	
	\fi

\begin{document}

\title[Critical Perturbation Theory, Part I]{Critical Perturbations for Second Order Elliptic Operators. Part I: Square function bounds for layer potentials}

	\author{S. Bortz}

\address{Simon Bortz \\
	Department of Mathematics, University of Alabama, Tuscaloosa, AL, 35487, USA}
\email{sbortz@ua.edu} 

\author{S. Hofmann}

\address{Steve Hofmann \\
	Department of Mathematics, University of Missouri, Columbia, MO 65211, USA}
\email{hofmanns@missouri.edu}

\author{J. L. Luna Garcia}

\address{Jos\'e Luis Luna Garcia \\
	Department of Mathematics and Statistics, McMaster University, Hamilton, Ontario, L8S 4K1, Canada}
\email{lunagaj@mcmaster.ca}

\author{S. Mayboroda}

\address{Svitlana Mayboroda \\
	School of Mathematics,	University of Minnesota, Minneapolis, MN 55455, USA} 
\email{svitlana@math.umn.edu}

\author{B. Poggi}

\address{Bruno Poggi \\
	Department of Mathematics,	Universitat Autònoma de Barcelona\\Bellaterra\\ 
	\phantom{h}
	Barcelona, 08193, Catalonia}
\email{bgpoggi.math@gmail.com}

\thanks{This material is based upon work supported by National Science Foundation under Grant No.\ DMS-1440140 while the  authors were in residence at the MSRI in Berkeley, California, during the Spring 2017 semester. S. Bortz and S. Mayboroda were partly supported by NSF INSPIRE Award DMS-1344235. S. Mayboroda was also supported in part by the NSF RAISE-TAQS grant DMS-1839077 and the Simons foundation grant 563916, SM. S. Hofmann was supported by NSF grant DMS-1664047. S. Bortz would like to thank Moritz Egert for some helpful conversations.}

\begin{abstract}  This is the first part of a series of two papers where we  study perturbations of divergence form second order elliptic operators $-\div A \nabla$ by   complex-valued  first and zeroth order terms, whose coefficients lie in critical spaces, via the method of layer potentials. In the present paper, we establish $L^2$ control of the square function via a vector-valued $Tb$ theorem and abstract layer potentials, and use these square function bounds to obtain uniform slice bounds for solutions.  For instance, an operator for which our results are new is the generalized magnetic Schr\"odinger operator $-(\nabla-i{\bf a})A(\nabla-i{\bf a})+V$ when the magnetic potential ${\bf a}$ and the electric potential $V$ are accordingly small in the norm of a scale-invariant Lebesgue space. 

\end{abstract}

\maketitle

\date{\today}


{
	\hypersetup{linkcolor=toc}
	\tableofcontents
}
\hypersetup{linkcolor=hyperc}

\section{Introduction} 
In this paper, the first in a two-part series, we lay the groundwork for the study of the $L^2$ Dirichlet, Neumann and Regularity problems for critical perturbations of second order divergence form equations by lower order terms. In particular, we produce the natural ($L^2$) square function estimates for (abstract) layer potential operators.  We consider differential operators of the form
\begin{equation}\label{opdef.eq}
\cL :=-\text{div }(A\nabla+B_1)+B_2\cdot\nabla+V
\end{equation}
defined on $\bb R^n\times\bb R = \{(x,t)\}$, $n \ge 3$, where $A=A(x)$ is an $(n+1)\times(n+1)$ matrix of $L^{\infty}$ complex coefficients, defined on $\bb R^n$ (independent of $t$) and satisfying a uniform ellipticity condition:
\begin{equation}\label{elliptic} 
\lambda|\xi|^2\leq~\real\langle A(x)\xi,\xi\rangle:=\real\sum\limits_{i,j=1}^{n+1}A_{ij}(x)\xi_j\overline{\xi_i},\qquad\Vert A\Vert_{L^{\infty}(\bb R^n)}\leq\Lambda,
\end{equation}
for some $\Lambda,\lambda>0$, and for all $\xi\in\bb C^{n+1}, x\in\bb R^n$. The first order complex coefficients $B_1=B_1(x),B_2=B_2(x)\in\big(L^n(\bb R^n)\big)^{n+1}$ (independent of $t$) and the complex potential $V=V(x)\in L^{\frac n2}(\bb R^n)$ (again independent of $t$) are such that
\begin{equation}\nonumber
\max\big\{\Vert B_1\Vert_{L^n(\bb R^n)},~\Vert B_2\Vert_{L^n(\bb R^n)},~\Vert V\Vert_{L^{\frac n2}(\bb R^n)}\big\}\leq \eps_0
\end{equation}
for some $\ep_0$ depending on dimension and the ellipticity of $A$ in order to ensure the accretivity of the form associated to the operator $\cL$ on the space
$$Y^{1,2}(\ree) : = \big\{ u \in L^{2^*}(\ree): \nabla u \in L^2(\ree)\big\}$$
equipped with the norm
$$\lVert u \rVert_{Y^{1,2}(\ree)} := \lVert u \rVert_{L^{2^*}(\ree)} + \lVert \nabla u \rVert_{L^2(\ree)},$$
where $p^* := \tfrac{(n+1)p}{n + 1 -p}$. We interpret solutions of $\cL u=0$ in the weak sense; that is, $u\in W^{1,2}_{\loc}(\ree)$ is a solution of $\cL u=0$ in $\om\subset \ree$ if for every $\varphi \in C_c^\infty(\om)$ it holds that
\begin{equation}\nonumber
\dint_{\ree} \left( (A\nabla u + B_1 u)\cdot\overline{\nabla \varphi} + B_2\cdot \nabla u \overline{\varphi} \right) =  0.
\end{equation}

Examples of operators of the type defined above include the Schr\"odinger operator $-\Delta+V$ with $t-$independent electric potential $V\in L^{\frac n2}(\bb R^n)$ having a small $L^{\frac n2}$ norm, and the  generalized magnetic Schr\"odinger operator $-(\nabla-i{\bf a})A(\nabla-i{\bf a})$, where $A$ is a $t-$independent complex matrix satisfying (\ref{elliptic}), and  the magnetic potential ${\bf a}\in L^n(\bb R^n)^{n+1}$ is $t-$independent and has small $L^n$ norm. We treat the case $n\geq3$ because the Sobolev spaces we encounter are of the form $\dt W^{1,2}(\bb R^n)\cap L^s$ for some $s\geq1$, and in this case, these spaces embed continuously into Lebesgue spaces. This is not the situation when $n=2$, in which case the Sobolev spaces considered embed continuously into $BMO$. If one were to treat the case $n = 2$, it would be natural to assume that $V=0$ and that $B_i$, $i = 1,2$ are divergence-free. Under these additional hypotheses, one can use a compensated compactness argument \cite{CLMS} to obtain the boundedness and invertibility of the form associated to $\cL$ (see \cite{GHN}).

However, there are several considerations in the case $n\geq3$ that set it qualitatively apart from $n=2$. For instance, when $n=2$, all solutions are locally H\"older continuous and this is certainly not the case when $n\geq3$. Indeed, let $u(x) = -\ln|x|$, $x \in \rn$ and build $V(x)$ or $B_1(x)$ so that either $ - \Delta u + Vu = 0$ or $-\Delta u + \div B_1 u = 0$ in the $n$-dimensional ball $B(0,1/2)$. By extending $u$ to be a function on $B(0,1/2) \times \re$ by $u(x,t) = u(x)$, we may see that the analogous equations in $n+1$ dimensions are satisfied by $u(x,t)$, and yet $u(x,t)$ fails to be locally bounded despite the fact that $V^2, B_1 \in L^n(\rn)$. Moreover, by considering $u(x,t)$ on a smaller ball and replacing $V$ or $B_1$ by $V_\epsilon = V \mathbbm{1}_{B(0,\epsilon)}$ or $(B_1)_\epsilon = B_1 \mathbbm{1}_{B(0,\epsilon)}$ respectively, we may ensure that $V^2_\epsilon$ or $(B_1)_\epsilon$ have arbitrarily small $L^n(\rn)$ norm, provided that we choose $\epsilon > 0$ small enough. Therefore, solutions in our perturbative regime fail to be locally bounded and hence fail (miserably) to be locally H\"older continuous. 

The lack of local H\"older continuity (or local boundedness) is one reason our results are not at all a straightforward adaptation of related works. For instance, in \cite{AAAHK} the authors are able to treat the fundamental solution as a Calder\'on-Zygmund-Littlewood-Paley kernel using pointwise estimates on the fundamental solution (and its $t$-derivatives) presented in \cite{HK}. Additionally, although establishing a Caccioppoli inequality (Proposition \ref{classCaccioppoli.prop}) is easy, constants are not necessarily null solutions to our operator and thus this Caccioppoli inequality does not yield the usual ``reverse'' Poincar\'e inequality for solutions. We remind the reader that if there are no lower order terms, the Caccioppoli inequality (becomes a ``reverse'' Poincar\'e inequality and) improves to an $L^p-L^2$ version; more precisely, we have that for each ball $B_r$ and some $p>2$, the estimate
\[
\Big(\dashint_{B_r} |\nabla u|^p \, dx \Big)^{1/p} \lesssim \frac{1}{r}\Big(\dashint_{B_{2r}} |u|^2 \, dx \Big)^{1/2}
\]
holds \cite{Meyers,Geh, Giaquinta}. We do not manage to obtain the above $L^p-L^2$ inequality, but rather a suitable $L^p-L^p$ version (Proposition \ref{Lpcaccop.prop}). The unavailability of these desirable estimates makes it far less clear whether constructing the fundamental solution will be useful for us, and so we do not attempt it. We still endeavor to use the method of layer potentials, whence we appeal to  (and adapt) the abstract constructions of Barton \cite{Bar}, which avoid the use of fundamental solutions entirely. Fundamental solutions have been constructed in other situations with lower order terms in \cite{DHM} and \cite{KimSak}, but they rely on sign conditions.

Our results in this series of papers concern the unique solvability of some classical $L^2$ boundary value problems in the upper half space $\ree_+:= \rn \times \re_+$. To state them, we ought to recall the definition of the ($L^2-$averaged) \emph{non-tangential maximal operator} $N$. Given $x_0\in\bb R^n$, define the cone $\gamma(x_0)=\{(x,t)\in\bb R^{n+1}_+:\,|x-x_0|<t\}$. Then, for $u:\bb R^{n+1}_+\ra\bb C$ we write
\[
Nu(x_0):=\sup_{(x,t)\in\gamma(x_0)}\Big(~\fiint_{|x-y|<t,\, |s-t|<\frac t2}|u(y,s)|^2\,dy\,ds\Big)^{\frac12}.
\]
Given $f:\bb R^n\ra\bb C$, we say that $u\ra f$ \emph{non-tangentially}, or $u\ra f ~n.t.$, if for almost every $x\in\bb R^n$, $\lim_{(y,t)\ra(x,0)}u(y,t)=f(x)$, where the limit runs over $(y,t)\in\gamma(x)$.

We consider the following boundary value problems: The Dirichlet problem
\[
\operatorname{(D2)} \begin{cases}
\cL u = 0\qquad\text{in } \ree_+,\\[1mm]
\lim_{t \to 0}u(\cdot ,t) = f\qquad\text{strongly in } L^2(\rn), \\[1mm]
Nu\in L^2(\bb R^n)\quad\text{and } u\ra f \qquad\text{non-tangentially}, \\[1mm]
\iint_{\ree_+}t |\nabla u(x,t)|^2 \, dx \, dt < \infty,\\[1mm]
\lim_{t \to \infty} u(\cdot, t) = 0\quad\text{ in the sense of distributions},
\end{cases}
\] 
the Neumann problem
\[\operatorname{(N2)} \begin{cases}
\cL u = 0\qquad\text{in } \ree_+, \\[1mm]
\tfrac{\partial u}{\partial \nu^{\cL}}: = -e_{n+1}(A\nabla u + B_1u)(\cdot,0) = g \in L^2(\rn),\footnotemark \\[1mm]
N(\nabla u)\in L^2(\bb R^n),\\[1mm]
\iint_{\ree_+} t | \partial_t \nabla u(x,t)|^2 \, dx \,dt < \infty,\\[1mm]
\lim_{t \to \infty} \nabla u(\cdot, t) = 0 \text{ in the sense of distributions},
\end{cases}
\]
\footnotetext{The boundary data is achieved in the distributional sense; see Section \ref{AbsLPthry.sec}. We elaborate on this in the upcoming paper.} and the Regularity problem 
\[\operatorname{(R2)} \begin{cases}
\cL u = 0 \qquad\text{in } \ree_+,\\[1mm]
u(\cdot, t) \to f\qquad\text{weakly in } Y^{1,2}(\rn),\\[1mm]
N(\nabla u)\in L^2(\bb R^n),\qquad\text{and }u\ra f\quad\text{non-tangentially},\\[1mm]
\iint_{\ree_+} t | \partial_t \nabla u(x,t)|^2 \, dx \,dt < \infty,\\[1mm]
\lim_{t \to \infty} \nabla u(\cdot, t) = 0 \text{ in the sense of distributions}.
\end{cases}\]
 
The idea is to follow a (by now) familiar process for proving $L^2$ existence and uniqueness for these boundary value problems. This process has three steps, which can be (very) roughly summarized as:
\begin{enumerate}
\item Show square function (and/or non-tangential maximal function) bounds for a linear operator defined, perhaps by continuous extension, on $L^2$, where the operator necessarily produces weak solutions to the elliptic equation (for us, this operator is either the single or double layer potential).
\item Show the boundedness and invertibility of the appropriate boundary trace of the operator. 
\item Show that any solution with square function (and/or non-tangential maximal function) bounds is, in fact, the solution produced by the linear operator with appropriate data.
\end{enumerate}
This paper is concerned establishing the square function bounds for abstract layer potential operators, that is, step (1) in the process above. We prove the following. 

\begin{theorem}[Square function bound for the single layer potential]\label{perturbfromlargem.thrm}  Suppose that $\cL_0 = - \div A\nabla$ is a divergence form elliptic operator with t-independent coefficients, and that the matrix $A$ is elliptic. Then, there exists $\eps_0 > 0$, depending on $n$, $\lambda$, and $\Lambda$, such that if $M \in \m M_{n+1}(\rn, \mathbb{C})$, $V \in L^{n/2}(\rn)$ and $B_i \in L^n(\rn)$ are (all) $t$-independent with
\begin{equation}\label{eq.control2}\nonumber
\|M\|_{L^\infty(\rn)} + \Vert B_1\Vert_{L^n(\bb R^n)} + \Vert B_2\Vert_{L^n(\bb R^n)} + \Vert V\Vert_{L^{\frac n2}(\bb R^n)}  < \eps_0,
\end{equation}
then   for each $m \in \N$, we have the estimate
$$\dint_{\ree_+} \big|t^m\partial_t^{m+1}\s_t^{\cL} f(x)\big|^2\,  \frac{dx \, dt}{t} \leq C \| f \|_{L^2(\rn)}^2,$$
where $C$ depends on $m$, $n$, $\lambda$, and $\Lambda$,   and
\[
\cL:= -\div([A + M]\nabla + B_1) + B_2\cdot \nabla  + V.
\]
Under the same hypothesis, the analogous bounds hold for $\cL$ replaced by $\cL^*$, and for $\bb R^{n+1}_+$ replaced by $\bb R^{n+1}_-$.
\end{theorem}

We point out that in the previous result, there is no restriction on the matrix $A$, other than that it be $t$-independent
and satisfy the complex ellipticity condition \eqref{elliptic}. In the homogeneous, purely second order case (i.e., the case that
$B_1, B_2,$ and $V$ are all zero), this result is due to Ros\'en \cite{Ro}; an alternative proof, with an extra hypothesis
of De Giorgi/Nash/Moser regularity,
appears in \cite{GH}.

We also obtain a uniform estimate on horizontal slices in terms of the square function.

\begin{theorem}[Uniform control of $\yotn$ norm on each horizontal
slice]\label{thm.suponslices} Suppose that $u\in\Y$ and $\cL u=0$ in $\bb R^{n+1}_+$ in the weak sense. Then for every $\tau>0$,
	\begin{gather}\label{eq.suponslices}
	\| \tr_\tau u\|_{L^{\frac{2n}{n-2}}(\rn)}+ \Vert\Tr_{\tau}\nabla u\Vert_{L^2(\bb R^n)}\lesssim \int_\tau^\infty \int_\rn t|\dno^2 u|^2\, dx dt \leq \||tD_{n+1}^2u\||,
	\end{gather}
	where  the traces exist in the sense of  Lemma \ref{ContSlices.lem}, and $C$ depends on $m$, $n$, $\lambda$, and $\Lambda$, provided that $\max\{\| B_1\|_n, \|B_2\|_n, \|V\|_{\frac n2}\}$ is sufficiently small depending on $m$, $n$, $\lambda$, and $\Lambda$. Under the same hypothesis, the analogous bounds hold for $\cL$ replaced by $\cL^*$, and for $\bb R^{n+1}_+$ replaced by $\bb R^{n+1}_-$.
\end{theorem}

Now we make the following important remark. Consider the modified Dirichlet, Neumann, and Regularity problems $\operatorname{(D2')}$, $\operatorname{(N2')}$ and $\operatorname{(R2')}$ where the third condition in each problem is deleted; that is,
\begin{multline}\label{eq.modified}
 \operatorname{(D2')}, \operatorname{(N2')} \text{ and } \operatorname{(R2')} \text{ are the problems with \emph{only} square function bounds,}\\ \text{and not   non-tangential maximal function estimates nor n.t. limits.}
\end{multline} 
In this case, our Theorems \ref{perturbfromlargem.thrm} and \ref{thm.suponslices}, when combined with well-known techniques in the literature, give the unique $L^2$ solvability of the modified Dirichlet, Neumann, and Regularity problems within a perturbative regime (see Theorem \ref{main2.thrm}).   Indeed, the boundedness and invertibility of the boundary trace operators\footnote{This invertibility gives the existence of solutions to the boundary value problems.} require  little more than the ``slice bounds'' produced here along with analytic perturbation theory; while the uniqueness of solutions with square function bounds is an exercise\footnote{Especially if one reads and is inspired by \cite{Auscher-Mourgoglou}.} in ``pushing'' a representation formula (Green's identity) to the boundary and exploiting the invertibility of the trace operators.   We  mention that uniqueness, that is, pushing a representation formula to the boundary, can be established under weak hypotheses that are implied by {\it either} square function {\it or} non-tangential maximal function estimates. Therefore, the significant contributions from the forthcoming paper are the non-tangential estimates which will allow us to obtain a stronger result than Theorem \ref{main2.thrm} below. For that reason, in this article we omit further details of the exercise that yields the $L^2$ solvability of the modified problems from the square function estimates and uniform slice estimates (but the full details will be given in the forthcoming paper).



 We summarize our observations in the following result. 
	
\begin{theorem}\label{main2.thrm}
Suppose that $\cL_0 = - \div A\nabla$ is a divergence form operator with complex, bounded, elliptic, $t-$independent coefficients. Suppose further that 
$\pm\tfrac{1}{2} I + K_{\cL} : L^2(\rn) \to L^2(\rn)$, $\pm\tfrac{1}{2} I + \widetilde K_{\cL} : L^2(\rn) \to L^2(\rn)$, and $(S_0)_{\cL} : L^2(\rn) \to Y^{1,2}(\rn)$ are all bounded and invertible for $\cL = \cL_0, \cL_0^*$, where $K_{\cL}, \widetilde K_{\cL}$ and $(S_0)_{\cL}$ are the ``boundary operators'' associated to $\cL$. 
Then there exists $\tilde \eps_0 > 0$ depending on dimension, ellipticity of $A$ and the constants in the operator norms of $\pm\tfrac{1}{2} I + K_{\cL}, \pm\tfrac{1}{2} I + \widetilde K_{\cL}$ and  $(S_0)_{\cL}$, $\cL = \cL_0, \cL_0^*$ and their inverses, such that if 
$M \in \m M_{n+1}(\rn, \mathbb{C})$, $V \in L^{n/2}(\rn)$ and $B_i \in L^n(\rn)$ are (all) $t$-independent with
\[\|M\|_{L^\infty(\rn)} + \Vert B_1\Vert_{L^n(\bb R^n)} + \Vert B_2\Vert_{L^n(\bb R^n)} + \Vert V\Vert_{L^{\frac n2}(\bb R^n)}  < \tilde \eps_0\]
then  the modified problems  $\operatorname{(D2')}$, $\operatorname{(N2')}$ and $\operatorname{(R2')}$ (see (\ref{eq.modified}))  are uniquely solvable for $\cL_1$, where
\[
\cL_1 := -\div([A + M]\nabla + B_1) + B_2\cdot \nabla  + V.
\]
Moreover, the solutions to $\operatorname{(D2')}$, $\operatorname{(N2')}$ and $\operatorname{(R2')}$ for $\cL_1$ can be represented by layer potentials, and have the natural square function bounds in terms of the data. 
\end{theorem}

This paper is organized as follows.  In Section \ref{PDEestimates.sec} we prove some elementary but essential PDE estimates, and  in Section \ref{AbsLPthry.sec}  we develop the notion of abstract layer potentials. Next, we show that for $\eps_0 > 0$ small enough, the single and double layer potentials have square function estimates (Theorems \ref{perturbfromlargem.thrm} and \ref{fullsqfnlargem.thrm}, and Lemma \ref{lm.l2slicesgrad}), which, in turn, give  us  `slice space' estimates for the single and double layer potentials (Theorems \ref{thm.suponslices2a} and \ref{thm.suponslicesd}). In passing we remark that  this analysis already allows us to establish the jump relations (as weak limits in $L^2(\rn)$)
\[ \m D_{\pm} f \to (\mp \tfrac{1}{2}I + K) f
\]
and
\[(\nabla S_t)|_{t = \pm s} f \to \mp \frac{1}{2}\frac{f(x)}{A_{n+1,n+1}} e_{n+1} + \m T f
\]
for $f$ in $L^2$, where $\m D$ and $\s$ are the double and single layer potentials.

Our results in this series may be best thought of as   extensions of the results in \cite{AAAHK} to lower-order terms as well as complex matrices (and not only those arising from perturbations of real symmetric coefficients or constant coefficients), albeit with the important distiction that we do not require DeGiorgi-Nash-Moser \cite{DeG, Na, Mo} estimates; this allows us to consider  any complex elliptic matrix for $A$.  
Let us mention a few applications of our theorems. For the magnetic Schr\"odinger operator $-(\nabla-i{\bf a})^2$ when ${\bf a}\in L^n(\bb R^n)^{n+1}$ is $t-$independent and has small $L^n(\bb R^n)$ norm, we obtain in this paper the first estimate for the square function  and solvability of the modified problems $\operatorname{(D2')}$, $\operatorname{(N2')}$, $\operatorname{(R2')}$  in the unbounded setting of the half-space. In fact, since our methods do not rely on an algebraic structure other than $t-$independence, we have similar novel conclusions for the generalized magnetic Schr\"odinger operators $-(\nabla-i{\bf a})A(\nabla-i{\bf a})$ where $A$ is a  real, symmetric, $t-$independent, elliptic, bounded matrix, and ${\bf a}$ is as above. 
	


We will now review some of the extensive history of boundary value problems for second order divergence form elliptic operators in Lipschitz domains. Unless otherwise stated, the results below will always be results for operators \emph{without} lower order terms. For Laplace  equation ($\cL = - \Delta$) in a Lipschitz domain, solvability of $\operatorname{(D2)}$ was obtained by Dahlberg \cite{D1}, and solvability of $\operatorname{(N2)}$ and $\operatorname{(R2)}$ was shown by Jerison and Kenig \cite{JK2}; these were also shown later by Verchota \cite{V} via the method of layer potentials, using the celebrated result of Coifman, McIntosh and Meyer \cite{CMcM}. For real, symmetric and $t$-independent coefficients, the solvability of $\operatorname{(D2)}$ was shown by Jerison and Kenig \cite{JK1}, and the solvability of $\operatorname{(N2)}$ and $\operatorname{(R2)}$ was shown by Kenig and Pipher \cite{KP1}. The solvability via the method of layer potentials in the case of real, symmetric and $t$-independent coefficients was shown in \cite{AAAHK} (and previously by Mitrea, Mitrea and Taylor \cite{Mit-Mit-Tay} with some  additional smoothness assumptions on the coefficients).

Aside from the results in \cite{AAAHK, Aus-Axel-Hof, Aus-Axel-McI}, which we describe further below, the known results for non-symmetric, $t$-independent matrices can be split into three categories: complex perturbations of constant coefficient matrices, `block' form matrices and real t-independent coefficients. In \cite{FJK}, Fabes, Jerison and Kenig showed solvability of $\operatorname{(D2)}$ for small complex $L^\infty$-perturbations of constant coefficient operators, the solvability of $\operatorname{(D2)}$, $\operatorname{(N2)}$ and $\operatorname{(R2)}$ in this setting was shown via layer potentials in \cite{AAAHK} (see \cite[Theorem 1.15]{AAAHK}). 

Solvability of $L^2$ boundary value problems in the case of \emph{all} `block' form matrices 
\begin{equation}\nonumber
A(x)=\left(\begin{array}{ccc|c}&&&0\\&B(x)&&\vdots\\&&&0\\ \hline 0&\cdots&0&1\end{array}\right),
\end{equation}
where $B(x)$ is an $n \times n$ matrix is, in the case of $\operatorname{(D2)}$, a consequence of the semigroup theory and, in the case of $\operatorname{(N2)}$ and $\operatorname{(R2)}$, a consequence of the solution to the Kato problem \cite{Katoproblem} on $\rn$. In particular, if we let $J: = - \div_x B(x) \nabla_x$, then one obtains the solvability of $\operatorname{(R2)}$ by solving the Kato problem for $J$, and one obtains the solvability of $\operatorname{(N2)}$ from solving the Kato problem for $adj(J)$. In fact, for $\operatorname{(N2)}$, one may equivalently show that the Riesz transforms associated to $J$, $\nabla J^{-1/2}$, are $L^2$ bounded, which can, in turn, be interpreted as a statement about the boundedness of the single layer potential from $L^2$ into $\dt W^{1,2}$. These results were obtained   in \cite{CMcM} ($n + 1 = 2$) and in  \cite{Katoproblem} ($n + 1 \ge 3$); see also \cite{HMc, AHLT, HLMc}.

In the case of real, $t$-independent coefficients, the results available are of the 
form $(\text{D}p)$ (for some $p <\infty$ sufficiently large), $(\text{N}p)$ and $(\text{R}p)$ (for some $p >1$, typically dual to the Dirichlet exponent), where $(\text{D}p)$, $(\text{N}p)$, and $(\text{R}p)$ are $L^p$ analogues of $(\text{D}2)$, $(\text{N}2)$, and $(\text{R}2)$ respectively. This is the best one can hope for by a counter-example in \cite{KKoPT}  (but see also \cite{Ax}), where the authors show that for  any fixed $p<\infty$, there exists a real (non-symmetric) coefficient matrix $A$, such that $(\text{D}p)$ fails to be solvable for the associated divergence form elliptic operator. In \cite{KKoPT}, the authors show that for all real $t$-independent coefficients with $n  + 1 = 2$, the problem  $(\text{D}p)$ is solvable  for some $p <\infty$.
In the same setting, Kenig and Rule \cite{KR} showed the solvability of $(\text{N}q)$ and $(\text{R}q)$ for $q$ the H\"older conjugate of the exponent $p$ from the aforementioned result \cite{KKoPT}. More recently, Barton \cite{bartonmem} perturbed these solvability results to deduce that $(\text{D}p)$, $(\text{N}q)$, and $(\text{R}q)$ remain solvable in the half-plane when the matrix consists of  almost-real coefficients; and the methods of \cite{KKoPT} were extended by Hofmann, Kenig, Mayboroda and Pipher \cite{HKMP1,HKMP2} to show the solvability of $(\text{D}p)$ for some $p <\infty$
for all real $t$-independent coefficients when $n  + 1 \ge 3$ and solvability of $(\text{R}q)$, again with $q$ dual to $p$.

As mentioned above, perhaps the closest results to the current exposition are \cite{AAAHK, Aus-Axel-Hof, Aus-Axel-McI}, where $L^2$ solvability of boundary value problems was explored for full complex matrices, either by the method of layer potentials \cite{AAAHK} or the ``first-order approach'' \cite{Aus-Axel-Hof, Aus-Axel-McI}  (which relies on the functional calculus of Dirac operators associated to divergence form elliptic operators). In \cite{AAAHK}, the authors show solvability of $(\text{D}2), (\text{N}2)$ and $(\text{R}2)$  via the method of layer potentials for $L^\infty$ perturbations of real, symmetric $t$-independent coefficients, and $L^\infty$ perturbations of constant coefficients. In \cite{Aus-Axel-Hof}, the authors show solvability of $(\text{D}2), (\text{N}2)$ and $(\text{R}2)$ in the same cases as \cite{AAAHK}, as well as perturbations from block form matrices. In \cite{Aus-Axel-McI}, the authors treat the previous cases of \cite{Aus-Axel-Hof} as well as perturbations of Hermitian coefficient matrices. 

We mention also the work of Gesztesy, Nichols, and the second author of this paper \cite{GHN}, where the authors studied the $n$-dimensional Kato problem related to our perturbations. The works \cite{AAAHK, Aus-Axel-Hof, Aus-Axel-McI, GHN} as 
well as \cite{Ro, GH, Hof-May-Mour} served as an indication that the present results should hold. The techniques 
from the solution to the (original) Kato problem are integral to our analysis. In particular, we adapt the methods 
from \cite{AAAHK, GH, Hof-May-Mour} to prove our square function estimates for the single layer potential (Theorems \ref{perturbfromlargem.thrm} and \ref{fullsqfnlargem.thrm}) via the generalized $Tb$ theory developed in the resolution of the Kato problem \cite{Katoproblem} and since refined in \cite{GH}. 

Let us remark on the assumption of $t$-independence. Given a second order divergence form elliptic operator (no lower order terms), define the transverse modulus of continuity $\hm(\tau): (0, \infty) \to [0,\infty]$ as 
\[\hm(\tau) : = \sup_{x \in \rn}\sup_{t \in (0, \tau)} |A(x,t) - A(x,0)|.\]
In \cite{CFK}, Caffarelli, Fabes and Kenig showed that given {\it any} function $\hm(\tau): (0, \infty) \to [0,\infty]$ such that $\int_0^1 [\hm(\tau)]^2\, d\tau/\tau = \infty$, there exists a real, symmetric elliptic matrix with  transverse modulus of continuity $\hm(\tau)$ such that the corresponding elliptic measure and $n$-dimensional Lebesgue measure (on $\rn \times \{0\}$) are mutually singular, and hence $(\text{D}2)$ (or even $(\text{D}p)$ for any $p$) fails to be solvable.  On the other hand, in \cite{FJK}, the authors show that $(\text{D}2)$ is uniquely solvable provided that $\int_0^1 [\hm(\tau)]^2\, d\tau/\tau < \infty$ and that $A(x,0)$ is sufficiently close to a constant matrix. Later, refinements of this condition were introduced and investigated; in these refinements one measures some discrepancy on Whitney boxes quantified by a Carleson measure condition; see, for instance, \cite{AA, FKP, D2, DPP, DPR, DP, Hof-May-Mour, KP2, KP3, fp}. In light of these constructions, it is natural to consider $t-$independent coefficients as an entry point to our investigations.

We end this review of the history of the work on the homogeneous (i.e., no lower order terms) operators by noting that the \emph{a priori} connections between the different problems $(\text{D}_p)$, $(\text{N}_{p'})$ and $(\text{R}_{p'})$ have also been of great interest. In some instances (say, $A$ is real, $t$-independent), one has that the solvability of $(\text{R}_{p})$ for $L$ implies solvability of $(\text{D}_{p'})$ for the adjoint operator $L^*$, and vice versa (where $p'$ is the H\"older conjugate to $p$) (see \cite{kbook}), but it was found in \cite{m} that such implications need not hold in the general setting of complex coefficients, even for $t$-independent matrices. We refer to \cite{m} for a more systematic review of these connections.

The literature in the setting with lower order terms present (that is, not all of $b_1,b_2, V$ are identically $0$) is much more sparse. In \cite{hoflew}, \emph{parabolic} operators with singular drift terms $b_2$ were studied, and their results would later be applied toward $(\text{D}_p)$ for elliptic operators with singular drift terms $b_2$ in \cite{KP3} and \cite{DPP}. When $A\equiv I$, $b_1\equiv b_2\equiv0$ and $V>0$ satisfies certain conditions, Shen proved the solvability of $(\text{N}_p)$ on Lipschitz domains in \cite{shenn}. His results were later extended in \cite{taor,taowang} to $(\text{R}_p)$ and under weaker assumptions on the potential $V$. It is a critical element of the proof that the leading term of $L\equiv-\Delta+V$ is the Laplacian, and the question of $(\text{N}_p)-$solvability for Schr\"odinger operators on rough domains in the case that $A\neq I$ remain open, even under generous assumptions on $V$.

More recently, the problems $(\text{D}_2)$ and $(\text{R}_2)$ for equations with lower order terms have been considered in \cite{sak} in bounded Lipschitz domains, under some continuity and sign assumptions on the coefficients. Solvability results for the variational Dirichlet problem of equations with lower order terms on unbounded domains have been obtained in \cite{mourg}.    Finally, we bring attention to \cite{mt}, where, through the development of a holomorphic functional calculus, the authors proved the $L^2$ well-posedness of the Dirichlet, Neumann, and regularity problems in the $t-$independent half-space setting for the Schr\"odinger operator $-\dv A\nabla+V$ with Hermitian $A$ and potential $V$ in the reverse H\"older class $RH^{\frac n2}$.

\section{Preliminaries}

As stated above, our standing assumption will be that $n \ge 3$, and the ambient space will always be $\ree = \{x,t: x \in \rn,\, t \in \mathbb{R}\}$. 
We employ the following standard notation:

\begin{list}{$\bullet$}{\leftmargin=0.4cm  \itemsep=0.2cm}
\item We will use lower-case $x, y, z$ to denote points in $\rn$ and lower-case $t,s, \tau$ to denote real numbers. By convention, $x=(x_1,\ldots,x_n)$, and $x_{n+1}=t$. We will use capital $X,Y,Z$ to denote points in $\ree$. The symbols $e_1,\ldots, e_{n+1}$ are reserved for the standard basis vectors in $\bb R^{n+1}$.

\item We will often be breaking up vectors into their parallel and perpendicular parts.
For an $(n + 1)$-dimensional vector $\vec{V} = (V_1, \dots, V_n, V_{n+1})$, we define its `horizontal' or `parallel' component as
$$V_\parallel : = (V_1, \dots, V_n),$$
and its `vertical' or `transverse' component as $V_\perp = V_{n+1}$. Similarly, we label the horizontal component of the $(n+1)$-dimensional gradient operator as
$$\nabla_\parallel := \nabla_x := (\partial_{x_1}, \dots, \partial_{x_n}),$$
and the `vertical' component as $D_{n+1}$ or $\nabla_{\perp}$.

\item Given the $(n+1)\times(n+1)$ complex-valued matrix $A$, for each $i,j=1,\ldots,n+1$, we denote by $A_{ij}$ the $ij-$th entry of $A$. We denote by  $\tilde A$ the $(n+1)\times n$ submatrix of $A$ consisting of the first $n$ columns of $A$. We define $\vec{A}_{i,\cdot}$ as the $(n+1)-$dimensional row vector made up of the $i-$th row of $A$; similarly we let $\vec{A}_{\cdot,j}$ be the $(n+1)-$dimensional column vector made up of the $j-$th row of $A$.

\item We set $\ree_+ := \rn \times (0, +\infty)$ and $\partial \ree_+ := \rn \times \{0\}$. We define $\ree_-$ similarly and often we write $\rn$ in place of $\partial \ree_+$ when confusion may not arise. For $t\in\bb R$, we denote $\bb R^{n+1}_t=\bb R^{n+1}_{+,t}:=\bb R^n\times(t,\infty)$, and $\bb R^{n+1}_{-,t}:=\bb R^n\times(-\infty,t)$.

\item The letter $Q$ will always denote a cube in $\rn$. By $\ell(Q)$ and $x_Q$ we denote the side length and center of $Q$, respectively. We write $Q(x,r)$ to denote the cube with center $x$ and sides of length $r$, parallel to the coordinate axes.

\item Given a (closed) $n$-dimensional cube $Q = Q(x,r)$, its concentric dilate by a factor of $\kappa >0$ will be denoted $\kappa Q := Q(x,\kappa r)$.   Similar dilations are defined for cubes in $\ree$ as well as (open) balls in $\rn$ and $\ree$.

\item For $a,b\in [-\infty,\infty]$, we set $\Sigma_a^b:=\big\{ X=(x,t)\in \ree: t\in (a,b)\big\}$.

\item Given a Borel set $E$ and Borel measure $\mu$, for any $\mu|_E$-measurable function $f$ we define the $\mu$-average of $f$ over $E$ as
\[\dashint_E f \, d\mu := \frac{1}{\mu(E)}\int_E f \, d\mu.\]

\item For a Borel set $E\subset \ree$, we let $\bbm{1}_E$ denote the usual
indicator function of $E$; that is, $\bbm{1}_E(x) = 1$ if $x\in E$, and $\bbm{1}_E(x)= 0$ if $x\notin E$.

\item For a Banach space $X$, we let $\B(X)$ denote the space of bounded linear operators on $X$. Simiarly, if $X$ and $Y$ are Banach spaces, we denote by $\B(X,Y)$ the space of bounded linear operators $X\ra Y$.

\end{list}

We will work with several function spaces; let us briefly describe them. For the rest of the paper, we assume that the reader is familiar with the basics of the theory of distributions and Fourier Transform   and the basics of the theory of Sobolev spaces (see \cite{L}). We delegate some of the basic definitions and results to these and other introductory texts. 

Let $\Omega$ be an open set in $\bb R^k$ for some $k\in\bb N$. For any $m\in\bb N$ and any $p\in[1,\infty)$, the space $L^p(\Omega)^m=L^p(\Omega,\bb C^m)$ consists of the complex-valued  $p-$th integrable $m-$dimensional vector functions over $\Omega$. We equip $L^p(\Omega,\bb C^m)$ with the norm
\[
\Vert\vec{f}\Vert_{L^p(\Omega,\bb C^m)}=\Big(\sum_{i=1}^m\int_{\Omega}|f_i|^p\Big)^{\frac1p},\qquad \vec f=(f_1,\ldots,f_m).
\]
For simplicity of notation, we often write $\Vert\vec f\Vert_p=\Vert\vec f\Vert_{L^p(\Omega)}=\Vert\vec f\Vert_{L^p(\Omega,\bb C^m)}$ when the domain $\Omega$ and the dimension of the vector function $\vec f$ are clear from the context (most often, when $\Omega$ is the ambient space, which for us means either $\Omega=\bb R^n$ or $\Omega=\bb R^{n+1}$).

The space $C_c^{\infty}(\Omega)$ consists of all compactly supported smooth complex-valued functions in $\Omega$. As usual, we denote $\n D=C_c^{\infty}(\bb R^{n+1})$, and we let $\n D'=\n D^*$ be the space of distributions on $\bb R^{n+1}$. The space $\n S$ consists of the Schwartz functions on $\bb R^{n+1}$, and $\n S'=\n S^*$ is the space of tempered distributions on $\bb R^{n+1}$.

For $p\in[1,\infty)$, we denote by $W^{1,p}(\Omega)$ the usual Sobolev space of functions in $L^p(\Omega)$ whose weak gradients exist and lie in $(L^p(\Omega))^{n+1}$. We endow this space with the norm
\[
\Vert u\Vert_{W^{1,p}(\Omega)}:=\Vert u\Vert_{L^p(\Omega)} +\Vert\nabla u\Vert_{L^p(\Omega)}.
\]
We define $W_0^{1,p}(\Omega)$ as the completion of $C_c^{\infty}(\Omega)$ in the above norm. We shall have occasion to discuss the homogeneous Sobolev spaces as well: by $\dt W^{1,p}(\Omega)$ we denote the space of functions in $L^1_{\loc}(\Omega)$ whose weak gradients exist and lie in $L^p(\Omega)$. We equip this space with the seminorm
\[
|u|_{\dt W^{1,p}(\Omega)}:=\Vert\nabla u\Vert_{L^p(\Omega)},
\]
and point out that  $\dt W^{1,p}(\Omega)$ coincides with the completion of the quotient space \\ $C^{\infty}(\Omega)/\bb C$ in the $|\cdot|_{\dt W^{1,p}(\Omega)}$ (quotient) norm. For $p\in(1,n+1)$ and $\om \subset \ree$ an open set, we define the space $Y^{1,p}(\om)$ as
$$Y^{1,p}(\om) : = \Big\{ u \in   L^{\frac{(n+1)p}{n+1-p}}(\om): \nabla u \in L^p(\om)\Big\}.$$
Write $p^* := \tfrac{(n+1)p}{n + 1 -p}$. We equip this space with the norm
$$\lVert u \rVert_{Y^{1,p}(\om)} := \lVert u \rVert_{L^{p^*}(\om)} + \lVert \nabla u \rVert_{L^p(\om)}.$$
We define $Y_0^{1,p}(\om)$ as the completion of $C_c^\infty(\om)$ in this norm. By virtue of the Sobolev embedding, when $p\in(1,n+1)$ we have that $Y_0^{1,p}(\Omega)$ coincides with the completion of $C_c^{\infty}(\om)$ in the $\dt W^{1,p}(\Omega)$ seminorm. Moreover, we have that $Y^{1,p}_0(\bb R^{n+1})=Y^{1,p}(\bb R^{n+1})$.

The $Y^{1,2}$ spaces exhibit the following useful property.

\begin{lemma}[Integrability up to a constant of a function with square integrable gradient on a half-space]\label{malziem.lem}
	Suppose that $u\in L^1_{\loc}(\sab)$ for some $a<b$, $a,b\in [-\infty,\infty]$, either $a=-\infty$ or $b=+\infty$, and that the distributional gradient satisfies $\nabla u\in L^2(\sab)$. Then there exists $c\in \CC$ such that $u-c\in \yota$.
\end{lemma}
The proof is very similar to that of Theorem 1.78 in \cite{MZ}, thus we omit it.

 In our paper, whenever we write $u(t)$ for $t\in\bb R$, we mean
	\begin{equation}\label{eq.ut}
		u(t)=u(\cdot,t);
	\end{equation}
	thus $u(t)$ is a measurable function on $\bb R^n$.  Let us present a fact regarding the regularity of functions in $Y^{1,2}(\ree)$ when seen as single-variable vector-valued maps. The proof is omitted as it is straightforward.

\begin{lemma}[Local H\"older continuity in the transversal direction]\label{ContSlices.lem}
	Suppose that $u\in Y^{1,2}(\Sigma_a^b)$ for some $a < b$. Then it holds that $u\in C_{\loc}^\alpha ((a,b); L^{2^*}(\rn))$
	for some exponent $\alpha>0$  (see \eqref{eq.ut}). Moreover, if $\partial_t \nabla u\in L^2(\Sigma_a^b)$, then we also have that $\nabla u\in C_{\loc}^{\beta} ((a,b); L^2(\rn))$ for some $\beta>0$.
\end{lemma}

\begin{remark}
Note that the functions above are representatives of $u(t)$ and $\nabla u(t)$, but that these retain the same properties as their smooth counterparts when acting on functions defined on the slice $\{ x_{n+1}=t\}$. More precisely, for any $\vec\varphi\in C_c^\infty(\rn;\CC^n)$ and any $t\in (a,b)$, we have the identity
\begin{equation*}
\int_\rn u(x,t) \divp\vec\varphi(x)\, dx = - \int_\rn \nablap u(x,t) \cdot\vec\varphi(x)\, dx.
\end{equation*}
The above identity is already true for a.e. $t\in (a,b)$, and is seen to be true for arbitrary $t\in(a,b)$ by the continuity of $u$ and $\nabla u$.
\end{remark}

Analogously, we introduce $Y^{1,2}(\bb R^n)$ as
$$Y^{1,2}(\bb R^n) : = \Big\{ u \in L^{\frac{2n}{n-2}}(\bb R^n): \nabla u \in L^2(\bb R^n)\Big\},$$
and equip it with the norm
$$\lVert u \rVert_{Y^{1,2}(\bb R^n)} := \lVert u \rVert_{L^{\frac{2n}{n-2}}(\bb R^n)} + \lVert \nabla u \rVert_{L^2(\bb R^n)}.$$
Note carefully that in our convention, $2^*=\frac{2(n+1)}{n-1}\neq\frac{2n}{n-2}$.

Some fractional Sobolev spaces will be useful for us when discussing trace operators. Let $\m F:L^2(\bb R^n)\ra L^2(\bb R^n)$ be the Fourier transform. Throughout this paper, we shall also denote $\hat u:=\m Fu$.   We write
\[
H^{\frac12}(\bb R^n)=\Big\{u\in L^2(\bb R^n)\,:\,\int_{\bb R^n}\big(1+|\xi|\big)|\hat u(\xi)|^2\,d\xi<+\infty\Big\}.
\]
The   space $\dt H^{\frac12}(\bb R^n)$ consists of those   tempered distributions $u\in\n S'$ whose Fourier transform $\hat u\in\n S'$ is a measurable function verifying that   $\int_{\bb R^n}|\xi||\hat u(\xi)|^2\,d\xi <+\infty$. Naturally, this space comes equipped with the seminorm $|u|_{\dt H^{\frac12}(\bb R^n)}=\int_{\bb R^n}|\xi||\hat u(\xi)|^2\,d\xi$. We define the space $\Hf=\dt H^{\frac12}_0(\bb R^n)$ as the completion of $C_c^{\infty}(\bb R^n)$ under the $\dt{H}^{\frac12}(\bb R^n)$ seminorm. We write $\Hfm:=(\Hf)^*$, and emphasize that we are departing from notation used elsewhere in the literature. Since $\Hf\supsetneq H^{\frac12}(\bb R^n)$, it follows that   $\Hfm$ is contained in  the dual space of $H^{\frac 12}(\bb R^n)$, which is the usual (inhomogeneous) fractional Sobolev space of order $-1/2$ that coincides with the space
\[
\Big\{u\in\n S'(\bb R^n)\,:\,\int_{\bb R^n}(1+|\xi|^2)^{-\frac12}|\hat u(\xi)|^2\,d\xi<+\infty\Big\}.
\]
For a survey on the properties of fractional Sobolev spaces, see \cite{dinpv}. We state without proof two easy results which are nevertheless useful.

\begin{proposition}[Sobolev embeddings of the fractional Sobolev spaces]\label{prop.fracemb}
Let $p_+:=\frac{2n}{n-1}$ and $p_-:=\frac{2n}{n+1}$. Then we have the continuous embeddings $
\Hf \hookrightarrow L^{p_+}(\rn)$, $L^{p_-}(\rn)\hookrightarrow \Hfm$.
\end{proposition}
\begin{sloppypar}
\begin{proposition}\label{mappropgrad.prop}
	The map $\nabla: \Hf\to \Hfm$ is bounded.
\end{proposition}
\end{sloppypar}

For fixed $t\in\bb R$ and any open set $\Omega\subset\bb R^{n+1}$ with nice enough (but possibly unbounded) boundary such that $\bb R^n\times\{\tau=t\}\subset\Omega$, we define the \emph{trace operator}
\begin{equation}\label{eq.traceop}
	 \Tr_t:C_c^{\infty}(\overline{\Omega})\ra C_c^{\infty}(\bb R^n),\qquad \Tr_tu=u(\cdot,t).
\end{equation}
The relevance of the fractional Sobolev spaces to our theory comes from the following   trace result; we cite a paper with the proof   for traces of functions in $\dt W^{1,2}(\bb R^{2})$, but the result is straightforwardly extended to our situation.

\begin{lemma}[Traces of $Y^{1,2}$ functions; \cite{strislab}]\label{lm.ytrace} Fix $t>0$. Let $\Omega$ be either $\bb R^{n+1}$, $\bb R^{n+1}_t$, or $\bb R^{n+1}_{-,t}$. Then, for each $s\in\bb R$ such that there exists $x\in\bb R^n$ with $(x,s)\in\Omega$, the  trace operator $\Tr_s$  (see (\ref{eq.traceop})) extends uniquely to a bounded linear operator $Y^{1,2}(\Omega)\ra H_0^{\frac12}(\bb R^n)$.
\end{lemma}

\begin{definition}[Local weak solutions]\label{weakform.def}   Let $\Omega\subseteq\bb R^{n+1}$ be an open set with Lipschitz (but possibly unbounded) boundary, and fix $f\in L^1_{\loc}(\om)$, $F\in L^1_{\loc}(\om,\cee)$, and $u\in W^{1,2}_{\loc}(\om)$. We say that $u$ solves the equation $\cL u=f-\div F$ in $\om$ \emph{in the weak sense} if, for every $\varphi\in C_c^\infty(\om)$, the identity
	\begin{equation}\label{solndef.eq}
	\dint_{\ree} \Big( (A\nabla u + B_1 u)\cdot\overline{\nabla \varphi} + B_2\cdot \nabla u \overline{\varphi}\Big) = \dint_{\ree} \left( f\overline{\varphi} + F\cdot \overline{\nabla \varphi}\right)
	\end{equation}
	holds.
\end{definition}

\begin{remark}\label{rm.extend}  Suppose that $\Omega$ is as in Lemma \ref{lm.ytrace}.  By a standard density argument, if $u\in Y^{1,2}(\Omega)$ solves $\cL u=f + \div F$ in $\om$ in the weak sense and either
	\begin{itemize}
		\item $F \in L^2(\om)$ and $f \in L^{(2n+2)/(n+3)}(\om)$, or
		\item $\om = D \times I$, where $D$ is a domain with nice enough (but possibly unbounded) boundary and $I$ is an interval, and 
		\begin{equation}\label{eq.extendcond}
			F \in L^2(\om),\quad f \in L^2(I;L^{(2n)/(n+2)}(D)) + L^{(2n+2)/(n+3)}(\om),
		\end{equation}
	\end{itemize}
	then \eqref{solndef.eq} holds for all $\varphi \in Y^{1,2}_0(\om)$. A similar observation to the second item can be made if $\om$ is a ball in $\ree$.
\end{remark}

For an infinite interval $I\subset \RR$ and a Banach space $X$, let $C^k_0(I;X)$ be the space of functions $f:I\to X$ such that all their first $k$ derivatives $f^{(l)}:I \to X$, $0\leq l\leq k$, exist, are continuous on $I$, and satisfy that  $\lim_{t\to\infty} \| f^{(l)}(t)\|_X =0$ for all $ 0\leq l\leq k$. When $k=0$, we will omit the superscript and simply write $C^0=C$.

\begin{definition}[Slice Spaces]\label{def.slicespace}For $n\geq 3$, we define 
\begin{equation*}
\dltp:= \left\{ v\in C_0\big((0,\infty); L^2(\rn)\big): \| u\|_{\dltp} <\infty\right\},
\end{equation*}
with norm given by $\| v\|_{\dltp}:= \sup_{t>0} \| v(t)\|_\ltrn$ (see (\ref{eq.ut})). We also define
\begin{equation*}
\begin{split}
\sltp & :=\left\{ u\in C^2_0\big((0,\infty); \yotn\big) : u'(t) \in C_0\big((0,\infty); \ltrn\big), \, \| u\|_{\sltp}<\infty\right\},
\end{split}
\end{equation*}
with norm given by 
\begin{multline*}
\| u\|_{\sltp}    := \sup_{t>0} \| u(t)\|_{\yotn}  +\sup_{t>0} \| u'(t)\|_\ltrn\\
  \quad + \sup_{t>0} \| tu'(t)\|_{\yotn} + \sup_{t>0}\| t^2 u''(t)\|_{\yotn}.
\end{multline*}
In particular, both $\dltp$ and $\sltp$ are Banach spaces. Similarly, with obvious modifications, we can define the slice spaces $\sltn$ and $\dltn$ in the negative half line $(-\infty, 0)$.
\end{definition}
We also state, without proof, the following criterion for the existence of weak derivatives in $L^2(I;X)$. See \cite{CH} for further results and definitions.

\begin{theorem}[Vector-valued weak derivatives; \cite{CH} Theorem 1.4.40]\label{diffboch.thm}
Suppose that $X$ is a reflexive Banach space and let $I\subset \RR$ be a (not necessarily bounded) interval. Let $f\in L^2(I;X)$. Then $f\in W^{1,2}(I;X)$ if and only if there exists $\varphi\in L^2(I;\RR)$ such that for any $t,s\in I$, the estimate
\begin{equation*}
\| f(t)-f(s)\|_X \leq \Big| \int_{s}^{t} \varphi(r)\, dr\Big|
\end{equation*}  
holds. Moreover, for a.e. $t\in I$,  the difference quotients 
\begin{equation*}
\Delta^hf(t):= \dfrac{f(t + h)-f(t)}{h},\qquad h\in\bb R, |h|\ll1,
\end{equation*}
converge weakly in $X$ to $f'(t)$   as $h\ra0$.
\end{theorem} 
\begin{remark} We will see that if $u\in W^{1,2}_{\loc}(\reu)\cap \sltp$ and $\cL u=0$ in $\reu$, then  by Caccioppoli's inequality (on slices) we have that
	\begin{multline*}
	\| u\|_{S^2_+}   \approx \sup_{t>0} \| u(t)\|_{\yotn} + \sup_{t>0} \| u'(t)\|_{\ltrn}\\
	  \approx\sup_{t>0} \| \nablap\trt u\|_{\ltrn} + \sup_{t>0} \|\trt( \dno u)\|_\ltrn. 
	\end{multline*} 
\end{remark}

We now state a Trace Theorem in cubes. We set
\begin{equation*}
I_R^\pm:= (-R,R)^n\times (0,\pm R), \quad I_R:=(-R,R)^{n+1}, \quad \Delta_R:= (-R,R)^n\times\{0\}.
\end{equation*}

\begin{proposition}[Trace operator on a cube]\label{troncubes.prop}  Let $H^{\frac12}(\Delta_R)$ be the space consisting of pointwise restrictions of functions in $H^{\frac12}(\bb R^n)$ to $\Delta_R$.  There exists a bounded linear operator $\tro^\pm: W^{1,2}(I_R^\pm)\to H^{\frac12}(\Delta_R)$   (called \emph{the trace operator associated to $I_R^{\pm}$}) with the following properties.
\begin{enumerate}[(i)]
	\item For each $u\in C^\infty(\overline{I_R^\pm})$, $\tro^\pm u(\cdot) = u(\cdot, 0)$.
	\item For each $\Phi\in C_c^{\infty}(I_R)$, the identity
\begin{equation}\nonumber
\int_{\Delta_R} (\tro^\pm u) \overline{\phi}= \mp\dint_{I_R^\pm} (u \overline{\dno \Phi} + \dno u\overline{\Phi}) 
\end{equation}
holds, where $\phi(\cdot)= \Phi(\cdot, 0)$.
\end{enumerate} 
In particular, the traces are consistent in the sense that for every $R'<R$, the restriction  to $I_{R'}^\pm$ of the trace operator associated to $I_R^\pm$, agrees with the trace in $I_{R'}^\pm$. 
\end{proposition}
\begin{proof}
The result follows from the usual Trace Theorem on Lipchitz domains (see, for instance, \cite{L} Theorem 15.23 and the results which follow this theorem) and the fact that $I_R^+$ is an extension domain for $W^{1,2}$ (see \cite{L} Theorem 12.15).\end{proof}

We now remark that the zeroth-order term $V$ in our differential equation can be absorbed into the first order terms. 

\begin{lemma}[Zeroth order term absorbed by first order terms]
Let $\cL$ be as in \eqref{opdef.eq} with
\[\max\big\{\Vert B_1\Vert_n,~\Vert B_2\Vert_n,~\Vert V\Vert_{\frac n2}\big\}\leq \eps_0.\]
Then 
\[\cL = -\div(A\nabla+\tilde B_1)+\tilde B_2\cdot\nabla,\]
where 
\[\max\big\{\Vert \tilde B_1\Vert_n,~\Vert \tilde B_2\Vert_n,\big\}\leq C_n\eps_0.\]
\end{lemma}
\begin{proof}
We write 
$$V(x) = -\div_x \nabla_\parallel I_2 V(x) = c_n\div_x \vec{R} I_1 V(x),$$
where $I_\alpha$ is the $\alpha$-order Riesz potential
\[
(I_{\alpha}f)(x)=\frac1{c_{\alpha}}\int_{\bb R^n}\frac{f(y)}{|x-y|^{\alpha}}\,dy,
\]
and $\vec{R}$ is the Riesz transform on $\rn$. For definitions and properties, see \cite{steinsid}. To conclude the lemma, we note that $I_1: L^{n/2}(\rn) \to L^n(\rn)$ and $\vec{R}$ is a bounded operator $L^n(\rn) \to [L^n(\rn)]^n$.\end{proof}

Observe that it suffices that $V\in \dt L^n_{-1} = \{ V\in \n D': I_1 V \in L^n\}$, with small norm. Thus, our results hold under this slightly more general assumption on $V$.

Accordingly, from now on we drop the term $V$ from our operator. We obtain invertibility of the operator $\cL$ on the Hilbert space $Y^{1,2}(\ree)$ when the size of the lower order terms is small enough.

\begin{definition}[Sesquilinear form and associated operator] Define the sesquilinear form $B_{\cL}:C_c^{\infty}(\bb R^{n+1})\times C_c^{\infty}(\bb R^{n+1})\ra\bb C$ via
\begin{equation*} 
B_{\cL}[u,v]:=\dint_{\bb R^{n+1}}\Big[A\nabla u\cdot\overline{\nabla v}+uB_1\cdot\overline{\nabla v}+\overline{v}B_2\cdot\nabla u\Big],\qquad u,v\in C_c^{\infty}(\bb R^{n+1}).
\end{equation*}
Define the operator $\cL:\n D\ra\n D'$ via the identity
\[
\langle\cL u,v\rangle=B_{\cL}[u,v],\qquad   u, v\in C_c^{\infty}(\bb R^{n+1}).
\]
It is clear that $\cL$ is linear.		
\end{definition} 

In fact, the form $B_{\cL}$ extends to a bounded, coercive form on  $Y^{1,2}(\bb R^{n+1})\times Y^{1,2}(\bb R^{n+1})$, and the operator $\cL$ extends to  an isomorphism  $Y^{1,2}(\bb R^{n+1})\ra(Y^{1,2}(\bb R^{n+1}))^*$. This is precisely the content of the following result.

\begin{proposition}[Extension of operator to $Y^{1,2}$]\label{laxmilgram.prop} The form $B_{\cL}$ extends to a bounded form on  $Y^{1,2}(\ree)$; that is,
\begin{equation*}
|B_{\cL}[u,v]|\lesssim \| \nabla u\|_2 \| \nabla v\|_2,\quad\text{for all } u,v\in C_c^{\infty}(\bb R^{n+1}),
\end{equation*}
with the implicit constant  depending on $n, \lambda, \Lambda,$ and  $\max\big\{\| B_1\|_n,\| B_2\|_n\big\}$. Hence $\cL$ extends to a bounded operator $Y^{1,2}(\bb R^{n+1})\ra(Y^{1,2}(\bb R^{n+1}))^*$.

Moreover, there exists a constant $\varepsilon_0 =\varepsilon_0(n,\lambda, \Lambda) > 0$ such that if\\ $\max\big\{\| B_1\|_n,\| B_2\|_n\big\} < \eps_0$, then $B_{\cL}$ is also coercive in $Y^{1,2}(\ree)$ with lower bound $\lambda/2$; that is,
\begin{equation*}
\frac{\lambda}{2}\| \nabla u\|_2^2 \lesssim \f Re\, B_{\cL}[u,u], \quad \text{for all } u\in C_c^{\infty}(\bb R^{n+1}).
\end{equation*} 
In particular, if $\max\big\{\| B_1\|_n,\| B_2\|_n\big\} <\varepsilon_0$, then by the Lax-Milgram Theorem the operator $\cL^{-1}: (Y^{1,2}(\ree))^*\to Y^{1,2}(\ree)$ exists as a bounded linear operator. 
\end{proposition}
\noindent\emph{Proof.} The proof is straightforward, thus omitted.\hfill{$\square$}

\begin{remark} We will always assume that $\max\{\| B_1\|_n,\| B_2\|_n\} <\varepsilon_0$, as above. The value of $\eps_0$ may be made smaller, but it will always depend only on $n , \lambda$ and $\Lambda$, and we will explicitly state when we impose further smallness.
\end{remark}

\begin{definition}[Dual operator]
Associated to $\cL$ we also have the dual operator, denoted $\cL^*:Y^{1,2}(\bb R^{n+1})\ra(Y^{1,2}(\bb R^{n+1}))^*$, defined by the relation
\begin{equation*}
\langle \cL u, v\rangle = \langle u, \cL^* v\rangle.
\end{equation*}
It is a matter of algebra to check that 
\begin{equation*}
\cL^* v = -\div( A^* \nabla v + \overline{B}_2v) + \overline{B}_1\cdot \nabla v
\end{equation*}
holds in the weak sense. 
\end{definition}

In particular, $\cL^*$ is an operator of the same type as $\cL$ and if \\ $\max\{\| B_1\|_n,\| B_2\|_n\} <\varepsilon_0$ so that $\cL^{-1}$ is defined, then  $(\cL^*)^{-1}$ is well defined, bounded, and satisfies $(\cL^*)^{-1}=(\cL^{-1})^*$.

\subsection{Generalized Littlewood-Paley Theory}

In this subsection, we review some of the known results from the generalized Littlewood-Paley theory. Here, the generalization is that one replaces the classical smoothness assumption by a so-called \emph{quasi-orthogonality} condition, and one replaces the classical pointwise decay condition by  off-diagonal decay in an $L^2$ sense. 

First, we introduce the \emph{square function norm} $\|| \cdot \||$. We define
\[\|| F \||_{\pm} := \Big(\dint_{\ree_\pm} |F(x,t)|^2 \, \frac{dx \, dt}{t} \Big)^{1/2}\quad, \quad \|| F \||_{all} := \Big(\dint_{\ree} |F(x,t)|^2 \, \frac{dx \, dt}{t} \Big)^{1/2}. \]
For a family of   linear  operators  on $L^2(\bb R^n)$, $\{\theta_t\}_{t>0} $, we define
\[ \||\theta_t  \||_{+,op} := \sup_{\|f \|_2 = 1} \||\theta_t  f\||_{+}, \]
and similarly define $\||\theta_t  \||_{-,op}$ and $\||\theta_t  \||_{all,op}$. We will often drop the sign in the subscript when in context it is understood that we work in the upper half space. 

Recall that a Borel measure $\mu$ on $\bb R^{n+1}_+$ is called {\it Carleson} if there exists a constant $C$ such that $\mu(R_Q) \leq C |Q|$ for all cubes $Q \subset \rn$, where $R_Q = Q \times (0, \ell(Q))$ is the \emph{Carleson box above $Q$}. Given a measurable function $\Upsilon$ on $\ree_+$, we define
$$\| \Upsilon \|_{\mathcal{C}} := \sup_{Q} \frac{1}{|Q|} \int_0^{\ell(Q)}\int_{Q} |\Upsilon(x,t)|^2 \, \frac{dx \, dt}{t},$$
where the supremum is taken over all cubes $Q\subset\bb R^n$.  In other words, $\| \Upsilon \|_{\mathcal{C}} < \infty$ if and only if $|\Upsilon(x,t)|^2 \, \tfrac{dx \, dt}{t}$ is a Carleson measure; in this case, we say that $\Upsilon\in\mathcal C$.   There is a deep connection between Carleson measures and square function estimates, as seen in the $T1$ theorem for square functions of Christ and Journ\'e \cite{CJ}. In this article, we   use  a generalized version of their result \cite[Theorem  4.3]{GH}.

We record several results from \cite{AAAHK}, which will be crucial in establishing square function estimates for solutions. 

\begin{definition}[Good off-diagonal decay]\label{goodODdecay.def}
We say that a family of   linear  operators  on $L^2(\bb R^n)$, $\{\theta_t\}_{t>0}$, has \emph{good off-diagonal decay} if there exist $M \ge 0$ and $C>0$ such that  for all $f\in L^2(\bb R^n)$, the estimate 
$$\|\theta_t (f {\bbm 1}_{2^{k+1}Q\setminus 2^kQ})\|_{L^2(Q)}^2 \lesssim_M 2^{-nk}\Big( \frac{t}{2^k \ell(Q)}\Big)^{2M + 2} \|f \|_{L^2(2^{k+1}Q \setminus 2^kQ)}^2$$
holds for every cube $Q \subset \rn$, every $k \ge 2$ and all $0 < t \leq C\ell(Q)$. Here, the implicit constants may depend only on dimension, $M$, and on the family of operators.
\end{definition}

If $b\in L^{\infty}(\bb R^n)$, then for any cube $Q$ in $\bb R^n$ and any $t\in(0,C\ell(Q))$, it can be shown via the good off-diagonal decay that $\theta_t(b\bbm 1_{\bb R^n\backslash Q})\in L^2(Q)$. This  allows us to define $\theta_tb:=\theta(b\bbm 1_Q)+\theta_t(b\bbm 1_{\bb R^n\backslash Q})\in L^2(Q)$ for any $t>0$ and $Q$ with $\ell(Q)\geq t/C$ (the independence of $\theta_tb$ over $Q$ is given by the linearity). Thus, for $b\in L^{\infty}(\bb R^n)$,   $\theta_tb\in L^2_{\loc}(\bb R^n)$ for each $t>0$. We omit further details.

\begin{lemma}[Consequences of off-diagonal decay; {\cite{FS}, \cite[Lemma 3.2]{AAAHK}}]\label{FSAAAHKL3.2.lem} 
Suppose that $\{\theta_t\}_{t>0}$ is a family of   linear operators  on $L^2(\bb R^n)$ with good off-diagonal decay   which verifies that $\||\theta_t\||_{op}\leq C$. Then, for every  $b\in  L^{\infty}(\bb R^n)$ (see the above remarks), the family $\{\theta_t\}_{t>0}$ satisfies the estimate
$$\| \theta_t b \|_{\C} \lesssim ( 1 + \|| \theta_t\||_{op}^2)\|b\|_\infty^2.$$
Moreover, if $\| \theta_t \|_{L^2 \to L^2} \lesssim1$ and $\theta_t 1 = 0$ for all $t>0$, then for every $b\in BMO(\bb R^n)$,  
$$\| \theta_t b \|_{\C} \lesssim ( 1 + \|| \theta_t\||_{op}^2)\|b\|_{BMO}^2.$$
\end{lemma}

\begin{lemma}[{\cite[Lemma 3.11]{AAAHK}}]\label{AAAHKL3.11.lem} 
Suppose that $\{\theta_t\}_{t>0} $ is a family of linear  operators on $L^2(\bb R^n)$ with good off-diagonal decay and which satisfies $\|\theta_t\|_{L^2 \to L^2} \lesssim1$ for all $t>0$. For each $t>0$, let $\A_t$ denote a self-adjoint averaging operator on $L^2(\bb R^n)$, given as $\A_tf=\int_{\bb R^n}f(y)\varphi_t(\cdot,y)\,dy$, whose kernel satisfies 
\[
0 \le \varphi_t(x,y) \lesssim t^{-n}{\bbm1}_{|x - y| \le Ct},\quad\text{and}\quad \int_{\rn} \varphi_t(x,y) \, dy = 1.
\]
Then for each $t>0$ and any $b\in L^{\infty}(\bb R^n)$, the function $\theta_t b$ is well defined as an element of $L^2_{\loc}(\bb R^n)$, and we have that
$$\sup_{t>0} \|(\theta_t b)\A_t f \|_{L^2(\bb R^n)} \lesssim \| b \|_\infty \|f \|_2.$$
\end{lemma}

\begin{lemma}[{ \cite[Lemma 3.5]{AAAHK}}]\label{AAAHKL3.5.lem}
Suppose that $\{R_t \}_{t>0}$ is a family of operators  on $L^2(\bb R^n)$ with good off-diagonal decay, and suppose further that $\|R_t\|_{L^2 \to L^2} \lesssim1$ and $R_t 1 = 0$ for all $t>0$ (note that by Lemma \ref{AAAHKL3.11.lem}, $R_t 1$ is defined as an element of $L^2_{\loc}(\bb R^n)$). Then for each $h\in\dt W^{1,2}(\bb R^n)$, we have that
$$\int_{\rn} |R_t h|^2  \lesssim t^2 \int_{\rn} |\nabla_x h|^2.$$ 
If, in addition, $\|R_t\dv_x \|_{L^2 \to L^2} \lesssim\frac1{t}$, then we also have for each $f\in L^2(\bb R^n)$ that
$$\int_{\ree_+} |R_t f(x) |^2 \, \frac{dx \, dt}{t} \lesssim \|f \|_2^2.$$
\end{lemma}

The following definition is important in establishing quasi-orthogonality estimates
(compare to the notion of an $\epsilon$-family in \cite{CJ}).

\begin{definition}[CLP Family]\label{CLP.def}
We say that a family of convolution operators  on $L^2(\bb R^n)$, $\{\m Q_s\}_{s > 0}$, is a \emph{CLP family} (``Calder\'on-Littlewood-Paley'' family), if there exist $\sigma > 0$ and $\psi \in L^1(\bb R^n)$ satisfying
\[
|\psi(x)| \lesssim (1 + |x|)^{-n-\sigma},\qquad\text{and}\qquad |\hat\psi(\xi)|\lesssim \min(|\xi|^\sigma, |\xi|^{-\sigma}),
\]
such that the following four statements hold.
\begin{enumerate}[i)]
\item The representation $\m Q_sf = s^{-n}\psi(\cdot/s) \ast f$ holds for each $f\in L^2(\bb R^n)$.
\item For each $f\in L^2(\bb R^n)$, we have control of the following $L^2$ norms uniformly in $s$:
\[
\sup_{s > 0}\left(\|\m Q_s f\|_2  +  \|s\nabla\m Q_s f\|_2\right) \lesssim \|f\|_2.
\]
\item For each $f\in L^2(\bb R^n)$, we have the square function estimate
\[
\int_0^\infty \int_{\rn} |\m Q_sf(x)|^2 \, \frac{dx \, ds}{s} \lesssim \|f\|^2_2.
\]
\item   Let $I:L^2(\bb R^n)\ra L^2(\bb R^n)$ be the identity operator. The equation
\[
\int_{0}^\infty \m Q_s^2 \, \frac{ds}{s} = I
\]
holds in the sense that   the  Bochner  integrals $\int_{\delta}^R\m Q_s^2\frac{ds}s$ converge to $I$ in the strong operator topology on $\B(L^2(\bb R^n))$ as $\delta\ra0$ and $R\ra\infty$. 
\end{enumerate}
\end{definition}

\begin{proposition}[Qualitative mappings]\label{CLPW12.prop}
Let $f\in \yotn$ and $\{\qs\}_{s>0}$ be either

\begin{enumerate}[a)]
\item A standard Littlewood-Paley family as in Definition \ref{CLP.def}, with kernel $\psi$, with the additional condition that there exists $\sigma>1$ such that $|\hat\psi(\xi)|\lesssim \min( |\xi|^\sigma, |\xi|^{-\sigma})$.
\item $\qs= I-P_s$, where $P_s$ is a nice approximate identity.
\end{enumerate} 
Then for all $s>0$, we have that $\qs f\in W^{1,2}(\rn)$.
\end{proposition}
\begin{proof}
In either case, via Plancherel's Theorem, it will suffice to estimate the $L^2$ norm of $\widehat{\m Q_sf}$. In case a), by basic properties of the Fourier Transform, we see that 
\begin{equation*}
\int_\rn |\widehat{\m Q_sf} (\xi)|^2 \, d\xi =\int_\rn |\hat{\psi}(s\xi)|^2 |\hat{f}(\xi)|^2\, d\xi \lesssim \int_\rn \min( |s\xi|^{\sigma-1}, |s\xi|^{-\sigma-1})^2 |\xi|^2|\hat{f}(\xi)|^2\, d\xi,
\end{equation*}
whence the desired conclusion follows in this case. For case b), we similarly compute, using Plancherel's Theorem and the Fundamental Theorem of Calculus, that if $\varphi$ is the radial kernel of the nice approximate identity $P_s$,
\begin{multline*}
\int_\rn |\widehat{\m Q_sf}(\xi)|^2\, d\xi   = \int_\rn |1-\hat{\varphi}(s|\xi|)|^2|\hat{f}(\xi)|^2\, d\xi
  = \int_\rn |\hat{f}(\xi)|^2 \Big| \int_0^{s|\xi|} \hat{\varphi}'(\tau)\, d\tau \Big|^2\, d\xi\\
  \leq \int_\rn s^2|\xi|^2|\hat{f}(\xi)|^2 \dashint_0^{s|\xi|} |\hat{\varphi}'(\tau)|^2\, d\tau \, d\xi
  \leq s^2  \| \hat{\varphi}'\|_{L^\infty(\rn)} \int_\rn |\xi|^2|\hat{f}(\xi)|^2\, d\xi. 
\end{multline*}\end{proof}

\section{Elliptic theory estimates}\label{PDEestimates.sec}

In this section, we establish several estimates for the operators under consideration, which are `standard' in the elliptic theory. We begin with Caccioppoli-type estimates.

\subsection{Caccioppoli-type inequalities}

Let us first show

\begin{proposition}[Caccioppoli inequality, \cite{DHM}]\label{classCaccioppoli.prop} Let $\Omega\subset\bb R^{n+1}$ be an open set. 
Suppose that $u\in W^{1,2}_{\loc}(\om)$, $f\in L^2_{\loc}(\Omega)$, $\vec F\in L^2_{\loc}(\Omega)^{n+1}$, and that $\cL u=f-\div \vec F$ in $\om$ in the weak sense. Then, for every ball $B\subset 2B\subset \om$, the estimate
\begin{equation*}
\dint_B |\nabla u|^2 \lesssim \dint_{2B} \Big( \dfrac{1}{r(B)^2}|u|^2+ |\vec F|^2 + r(B)^2|f|^2\Big), 
\end{equation*}
holds, with the implicit constant  depending only on $n, \lambda, \Lambda$.
\end{proposition}
The above estimate is a particular case of a Caccioppoli inequality obtained in a very general setting of elliptic systems in \cite{DHM}. Since our techniques will be exploited in several calculations later, we present here a self-contained proof. 
\begin{proof} 
	
	 Consider $\eta\in C_c^\infty(2B)$ such that $0\leq \eta \leq 1$, $\eta\equiv 1$ in $B$ and $|\nabla \eta|\lesssim r(B)^{-1}$. Note that $u\eta^2$ is a valid testing function in (\ref{solndef.eq}), and therefore we obtain that
\begin{multline*}
\dint_{\ree} \lambda|\nabla u|^2 \eta^2 \leq \dint_{\ree}A\nabla u\cdot \overline{\nabla u}\eta^2\\
=\dint_{\ree} \Big( -2(A\nabla u\cdot\nabla\eta) \eta \bar{u}  + B_1 u\cdot \overline{\nabla (u\eta^2)} - B_2\cdot \nabla u \overline{u\eta^2} \Big) \\
\quad + \dint_{\ree}\Big(\vec F\cdot \overline{\nabla (u\eta^2)} +f\overline{u\eta^2}\Big)\\
=: I +II +III + IV+V.
\end{multline*}
To handle the term $I$, we use Cauchy's inequality with $\ep>0$ and the boundedness of $A$ to obtain that
\begin{equation*}
|I|\leq 2\Lambda \dint_{\ree} |\nabla u|\eta |\nabla \eta||u|\leq \Lambda \varepsilon \dint_{\ree} |\nabla u|^2\eta^2 + \dfrac{\Lambda}{\varepsilon} \dint_{\ree} |u|^2|\nabla \eta|^2.
\end{equation*}
with $\varepsilon$ small enough (depending only on $\lambda, \Lambda$) that we can hide the first term. The second term is seen to be of a desired form after using the bound on $|\nabla\eta|$.

To handle the term $III$, we use the H\"older and Sobolev inequalities in $\rn$ coupled with the $t-$independence of $B_2$, as follows:
\begin{multline*}
|III| \leq \int_{-\infty}^\infty\int_{\rn} |B_2|(|\nabla u|\eta)|u|\eta\, dxdt\\
 \leq \int_{-\infty}^\infty \| B_2\|_{\lnrn} \|\eta\nabla u\|_{L^2(\rn)} \| u\eta\|_{L^{\frac{2n}{n-2}}(\rn)}\, dt\\
 \lesssim \| B_2\|_{\lnrn} \int_{-\infty}^\infty\|\eta \nabla u \|_\ltrn \|\nabla_{\|} (u\eta)\|_\ltrn \, dt\\
 \leq \| B_2\|_\lnrn \int_{-\infty}^\infty \Big( \|\eta \nabla u  \|_\ltrn^2+ \|\eta \nabla u \|_\ltrn\|u\nabla \eta\|_\ltrn\Big)\, dt.
\end{multline*} 
Using the Cauchy inequality on the second term, we arrive at the estimate
\begin{equation*}
|III|\lesssim \| B_2\|_n \dint_{\ree} \left( |\nabla u|^2 \eta^2+ |u|^2 |\nabla \eta|^2\right). 
\end{equation*}
If we choose $\|B_2\|_n<\varepsilon_0$ (see Proposition \ref{laxmilgram.prop}) with $\varepsilon_0$ small enough (depending only on $n,\lambda, \Lambda$), we can hide the first term, while the second term is of a desired form.

To handle the term $II$, notice that the product rule allows us to write the estimate
\begin{equation*}
|II|\leq \dint_{\ree} \left( |B_1||u||\nabla u|\eta^2+ 2|B_1||u|^2\eta|\nabla \eta|\right)=:II_1+II_2.
\end{equation*}
The first term is handled similarly as $III$. As for $II_2$, we appeal again to the H\"older and Sobolev inequalities, together with the $t$-independence of $B_1$, to see that
\begin{multline*}
|II|  \lesssim \int_{-\infty}^\infty \| B_1\|_\lnrn \| u\nabla \eta\|_\ltrn \| u\eta\|_{L^{2n/n-2}(\rn)}\, dt\\
  \lesssim \| B_1\|_\lnrn \int_{-\infty}^\infty \| u\nabla \eta\|_\ltrn \| \nabla_{\|} (u\eta)\|_\ltrn\, dt,
\end{multline*}
and this last expression may be handled in the same way as in $II$.

For the term $IV$, we use the product rule to obtain that
\begin{equation*}
\begin{split}
|IV| & \leq \dint_{\ree} \big( |\vec F|\eta |\nabla u|\eta + 2|\vec F|\eta |u||\nabla \eta|\big)=: IV_1+IV_2.  
\end{split}
\end{equation*}
The first term may be estimated with the Cauchy's inequality with $\varepsilon$:
\begin{equation*}
IV_1 \leq \dint_{2B} \Big( \dfrac{1}{\varepsilon}|\vec F|^2 + \varepsilon|\nabla u|^2\eta^2\Big),
\end{equation*}
and we can hide the second term. The term $IV_2$, since after using the Cauchy inequality, both terms are of a desired form:
\begin{equation*}
IV_2 \leq \dint_{2B} \big( |\vec F|^2 + |u|^2|\nabla \eta|^2\big).
\end{equation*}

Combining these estimates gives
\begin{equation*}
\dint_B |\nabla u |^2  \leq \dint_{\ree} |\nabla u|^2 \eta^2  \lesssim \dfrac{1}{r(B)^2}\dint_{2B} \left( |u|^2+ |\vec F|^2\right)  +|V|. 
\end{equation*}

To handle the term $V$, we use the Cauchy inequality to obtain that
\begin{equation*}
|V| \leq \dint_{\ree} |f||u|\eta^2  \leq \dint_{2B} \Big(r(B)^2|f|^2 + \dfrac{1}{r(B)^2} |u|^2\Big) .
\end{equation*}
This completes the proof.\end{proof}

\begin{remark}[$Y^{1,p}$ form a complex interpolation scale]
In the case of purely second order operators (that is, $B_1=B_2=0$), we may exploit the fact that constants are always null-solutions. Applying the Poincar\'e inequality, we obtain a weak reverse H\"older inequality for $\nabla u$, which in particular implies $L^p$ integrability for the gradient, for some $p>2$. We do not obtain the analogous estimate here, but rather a suitable substitute. More precisely, we shall muster an $L^p$ version of the Caccioppoli inequality. In order to prove this result, we remark that the spaces $Y^{1,p}(\ree)$ and their dual spaces, $(Y^{1,p})^*$,  form a complex interpolation scale, with
\[
[Y^{1,p_1}, Y^{1,p_2}]_{\theta} = Y^{1,p_\theta}, \quad \frac1{p_\theta} = \frac{1 -\theta}{p_1} + \frac{\theta}{p_2},
\]
for $\theta \in (0,1)$ and $1 < p_1 < p_2 < n$. We may show this fact by gathering the following two ingredients. First, the homogeneous spaces $\dt{W}^{1,p}$ form a complex interpolation scale (see \cite{trieb}). Next, one uses that the map that sends an element in $\dt{W}^{1,p}$ to its unique representative in $Y^{1,p}$ is a `retract' (see \cite[Lemma 7.11]{KMM} and the discussion preceding it). Thus, we employ \cite[Lemma 7.11]{KMM} and conclude that the spaces $Y^{1,p}$ form a complex interpolation scale. The fact that $(Y^{1,p})^*$ form a complex interpolation scale is a general consequence of the interpolation scale for $Y^{1,p}$; see, for instance, \cite[Theorem 4.5.1]{Ber-Lof}.
\end{remark}

The $L^p$ Caccioppoli inequality will also make use of the well-known lemma of {\v{S}}ne{\u{\ii}}berg \cite{Snei}. The (explicitly) quantitative version stated here appears in \\ \cite{ABES-KMS}. 

\begin{theorem}[{\v{S}}ne{\u{\ii}}berg's Lemma {\cite[Theorem A.1]{ABES-KMS},\cite{Snei}}]
\label{thm:Sneiberg}
Let \\$\cl{X} = (X_0, X_1)$ and $\cl{Y} = (Y_0, Y_1)$ be interpolation couples of  Banach spaces, and $T \in \B(\cl{X}, \cl{Y})$. Suppose that for some $\theta^* \in (0,1)$ and some $\kappa>0$, the lower bound $ \|T x\|_{Y_{\theta^*}} \geq \kappa \|x\|_{X_{\theta^*}}$ holds for all $x\in X_{\theta^*}$. Then the following statements are true.

\begin{enumerate}[i)]
 \item Given $0 < \eps < 1/4$, the lower bound $ \|T x\|_{Y_\theta} \geq \eps \kappa \|x\|_{X_\theta}$ holds for all $x\in X_{\theta}$, provided that $|\theta - \theta^*| \leq \frac{\kappa(1-4\eps)\min \{\theta^*, 1- \theta^*\}}{3\kappa + 6M}$, where $M = \max_{j=0,1} \|T\|_{X_j \to Y_j}$.
 \item If $T: X_{\theta^*} \to Y_{\theta^*}$ is invertible, then the same is true for $T: X_{\theta} \to Y_{\theta}$ if $\theta$ is as in (i). The inverse mappings agree on $X_{\theta} \cap X_{\theta^*}$ and their norms are bounded by $\frac1{\eps \kappa}$.
\end{enumerate}
\end{theorem}

Using the above result, we can easily obtain

\begin{lemma}[Invertibility of $\cL$ in a window around $2$]\label{Snei.lem} Let $p \in (1,n)$ be such that $p' < n$, where $p'$ is the H\"older conjugate of $p$. The operator $\cL$ extends to a bounded operator $Y^{1,p}(\ree) \to (Y^{1,p'}(\ree))^*$. Moreover, the operator is invertible if $|p -2|$ is small enough depending on $n$, $\lambda$, and $\Lambda$.
\end{lemma}

\begin{remark}\label{plowerrestrict.rmk}
Here and throughout, we assume that the range of $p$ near $2$ in Lemma \ref{Snei.lem} is such that $p_*=\frac{(n+1)p}{n+1+p}< 2$.
\end{remark}

The following lemma details the modification to the operator output upon multiplying a solution by a cut-off function.
\begin{lemma}\label{cutoffremainder.lem} Let $\om \subset \ree$ be an open set. Suppose that $u\in W^{1,2}_{\loc}(\Omega)$ satisfies $\cL u = 0$ in $\om$ in the weak sense. Then for any $\chi \in  C_c^\infty(\om, \re)$, we have that
\begin{equation}\label{cutoffsolveseq.eq}
\cL(\chi u) =  \div\vec{F} + f
\end{equation}
in $\ree$ in the weak sense, where $\vec F = A (\nabla \chi) u$, and $f = -A \nabla u \cdot \nabla \chi - B_1u \nabla \chi + B_2u \nabla \chi$.
\end{lemma}
\begin{proof}
We apply the operator $\cL$ to $u\chi$ and test against $\varphi \in C_c^\infty(\ree)$ with the goal in mind of extracting a term of the form $\langle \cL u, \varphi \chi \rangle = 0$. Observe that
\begin{align*}
\int_{\bb R^{n+1}} A \nabla (u \chi) \cdot \overline{\nabla \varphi} & = \int_{\bb R^{n+1}} A \nabla u \cdot \overline{\nabla (\chi \varphi)} + \int_{\bb R^{n+1}} uA \nabla\chi\cdot  \overline{\nabla \varphi} - \int_{\bb R^{n+1}} [A \nabla u \cdot \nabla \chi] \overline{\varphi},
\\  \int_{\bb R^{n+1}} (B_1 u\chi)\cdot \overline{\nabla\varphi} & = + \int_{\bb R^{n+1}} B_1 u \overline{\nabla(\chi \varphi)} - \int_{\bb R^{n+1}} [B_1 u \nabla \chi] \overline{ \varphi},
\\ \int_{\bb R^{n+1}} B_2\nabla (u\chi) \overline{\varphi} &= \int_{\bb R^{n+1}} B_2 \nabla u \overline{\chi \varphi} + \int_{\bb R^{n+1}}[B_2 u \nabla \chi] \overline{\varphi},
\end{align*}
where we use that $\chi$ is real-valued. Collecting the first terms in each inequality and noting that $\varphi\chi \in C_c^\infty(\om)$, we realize that the contribution of these terms is $\langle \cL u, \varphi \chi \rangle = 0$. Then we have that $\langle \cL (\chi u), \varphi \rangle  =  \langle \div\vec{F} + f, \varphi \rangle$, as desired.\end{proof}

We are now ready to combine the past few results and obtain the local high integrability of the gradient. 

\begin{lemma}[Local high integrability of the gradient of a solution]\label{yougotSneiberg'd.lem}
Let $\om$ be an open set. Suppose that $u\in W_{\loc}^{1,2}(\Omega)$ solves $\cL u = 0$ in $\om$ in the weak sense. Then $u \in W^{1,p}_{\loc}(\om)$, where $p$ is close to $2$ and depends only on $n$, $\lambda$, $\Lambda$, and $\ep_0$. Moreover, for any $\chi \in C_c^\infty(\om, \re)$ we have the estimate
\begin{equation*}
\lVert \chi u \rVert_{Y^{1,p}(\bb R^{n+1})} \le \lVert \cL^{-1}(\div \vec{F} +  f)\rVert_{Y^{1,p}(\bb R^{n+1})} \lesssim \lVert \vec F \rVert_p + \lVert f \rVert_{p_*}, 
\end{equation*}
where $\vec F$ and $f$ are as in Lemma \ref{cutoffremainder.lem}.
\end{lemma}
\begin{proof}
Let $\vec{F}$ and $f$ be as in the previous lemma. One may verify, using the Sobolev embedding and the fact that $\chi$ is smooth and compactly supported, that $\vec{F} \in L^1(\bb R^{n+1}) \cap L^{2^*}(\bb R^{n+1})$ and that $f \in L^1(\bb R^{n+1}) \cap L^2(\bb R^{n+1})$. Choosing $p > 2$ with $|p - 2|$ sufficiently small, we may apply Lemma \ref{Snei.lem} to show that the operator $\cL$ extends to a bounded and invertible operator $Y^{1,p}(\ree) \to (Y^{1,p'}(\ree))^*$. Hence $\cL^{-1}$ is bounded. Applying $\cL^{-1}$ to each side of \eqref{cutoffsolveseq.eq}, we obtain that
$$\lVert \chi u \rVert_{Y^{1,p}} \le \lVert \cL^{-1}(\div \vec{F} +  f)\rVert_{Y^{1,p}} \lesssim \lVert \vec F \rVert_{p} + \lVert f \rVert_{p_*}.$$
Here, we note that $L^{p_*}$ embeds continuously into $(Y^{1,p'})^*$, and $\div \vec F \in (Y^{1,p'})^*$ since $\vec F \in L^p$. This observation uses the identity $[(p')^*]' = p_*$ and the continuous embedding $Y^{1,p'}(\ree)\hookrightarrow L^{(p')^*}(\ree)$.\end{proof}

Finally, we provide a more precise version of the above Lemma, namely the $L^p$-Caccioppoli inequality.

\begin{proposition}[$L^p$-Caccioppoli inequality]\label{Lpcaccop.prop}
Let $\om \subset \ree$ be an open set and let $u \in W_{\loc}^{1,2}(\om)$ solve $\cL u = 0$ in $\om$ in the weak sense. Suppose that $B$ is a ball such that $\kappa B \subset \om$ for some $\kappa > 1$. Then, for every $p>0$ such that $|p -2|$ is small enough  that the conditions of Lemma \ref{yougotSneiberg'd.lem} are satisfied, the estimate
\begin{equation}\label{starp.eq}
\lVert \nabla u \rVert_{L^p(B)} \lesssim \frac{1}{r(B)} \lVert u \rVert_{L^p(\kappa B)}
\end{equation}
holds, where the implicit constants depend on $\kappa$, $p$, $n$, $\lambda$, $\Lambda$, and $\ep_0$. 
\end{proposition}

\begin{proof}
Set $r := r(B)$ and let $\chi = \eta^2$ with $\eta \in C_c^\infty(\frac{1 + \kappa}{2} B, \re)$, $0 \le \eta \le 1$, $|\nabla \eta| \lesssim \tfrac{1}{r}$. Note that $\chi$ has the same properties as $\eta$. The estimate \eqref{starp.eq} will follow immediately from the estimate
\begin{equation}\label{dubstarp.eq}
\lVert u \chi \rVert_{Y^{1,p}(\ree)} \lesssim \frac{1}{r} \lVert u \rVert_{L^p(\kappa B)},
\end{equation}
since $\lVert \nabla u \rVert_{L^p(B)} \lesssim \lVert (\nabla u) \chi \rVert_{p}$ and (the reverse triangle inequality yields)
$$\lVert (\nabla u) \chi \rVert_{p} - \lVert (\nabla \chi) u \rVert_{p} \lesssim \lVert \nabla (u \chi) \rVert_{p} \le \lVert u \chi \rVert_{Y^{1,p}(\ree)}.$$
We immediately note that we have already established \eqref{dubstarp.eq} in the case $p = 2$; this is the classical Caccioppoli inequality. Applying Lemma \ref{yougotSneiberg'd.lem}, we have that
\begin{equation}\label{termsforcacc.eq}
\lVert \chi u \rVert_{Y^{1,p}(\ree)} \lesssim \lVert \vec F \rVert_{p} + \lVert f \rVert_{p_*},
\end{equation}
where $\vec F$ and $f$ are as in Lemma \ref{cutoffremainder.lem}. The bound 
\begin{equation}\label{vecFbound.eq}
\lVert \vec{F} \rVert_p= \lVert A \nabla \chi u \rVert_p \lesssim \frac{1}{r} \lVert u \rVert_{L^p(\kappa B)}
\end{equation}
is trivial from the properties of $A$ and $\chi$ and desirable from the standpoint of \eqref{dubstarp.eq}.
It remains to find appropriate bounds for the terms appearing in the expression for $f$. To this end, we have by Minkowski's inequality that
\begin{equation*}\label{fmink.eq}
\lVert f \rVert_{p_*} \le \lVert A \nabla u \cdot \nabla \chi \rVert_{p_*} + \lVert  B_1u \nabla \chi \rVert_{p_*}+ \lVert B_2u \nabla \chi \rVert_{p_*} = I + II + III.
\end{equation*}
Before continuing, we remark that the relation $\frac{n +1}{p_*} = \frac{n+1}{(n+1)p}[(n + 1) + p] = \frac{n+1}{p} +1$

holds. Using the $L^2$ Caccioppoli inequality, Jensen's inequality and the fact that $p > 2$, we have that
\begin{multline}\label{f1bound.eq}
I =  \lVert A \nabla u \cdot \nabla \chi \rVert_{p_*} 
  \lesssim r^\frac{n+1}{p}\Big(\dashint_{\frac{1 + \kappa}{2}B} |\nabla u |^{2} \Big)^{\frac{1}{2}}
    \lesssim \frac{1}{r} r^\frac{n+1}{p}\Big(\dashint_{\kappa B} |u |^{2} \Big)^{\frac{1}{2}}
\\   \lesssim \frac{1}{r} r^\frac{n+1}{p}\Big(\dashint_{\kappa B} |u|^{p} \Big)^{\frac{1}{p}}
  \lesssim \frac{1}{r} \Big(\int_{\kappa B} |u|^{p} \Big)^{\frac{1}{p}}.
\end{multline}

Next we bound $II$ and $III$.  The Sobolev embedding on $\rn$ and the Caccioppoli inequality \footnote{More precisely, we use \eqref{dubstarp.eq} with $p = 2$.} yield for $i = 1,2$ the estimate
\begin{multline}\label{f23bound.eq}
\lVert B_i u (\nabla \chi) \rVert_{p_*} \lesssim \frac{1}{r} \lVert B_i (u\eta) \rVert_{p_*}   \lesssim \frac{1}{r} r^\frac{n+1}{p_*} \Big(\dashint_{\frac{1 + \kappa}{2}B} |B_i(u\eta)|^{p_*} \Big)^\frac{1}{p_*}
\\[1mm] \lesssim \frac{1}{r} r^\frac{n+1}{p_*}r^{-\frac{n+1}{2}} \Big(\int_{\frac{1 + \kappa}{2}B} |B_i(u\eta)|^{2} \Big)^\frac{1}{2}
   \lesssim \frac{1}{r} r^\frac{n+1}{p_*}r^{-\frac{n+1}{2}} \Big(\int_{\ree} |\nabla(u\eta)|^{2} \Big)^\frac{1}{2}
\\[1mm] \lesssim \frac{1}{r} r^\frac{n+1}{p} \Big(\dashint_{\kappa B} |u|^{2} \Big)^\frac{1}{2}
  \lesssim \frac{1}{r} \Big(\int_{\kappa B} |u|^{p} \Big)^{\frac{1}{p}}.
\end{multline}
Combining \eqref{vecFbound.eq}, \eqref{f1bound.eq} and \eqref{f23bound.eq} with \eqref{termsforcacc.eq}
and the definitions of $\vec{F}$ and $f$, we obtain \eqref{dubstarp.eq}. As we had reduced the proof of the statement of the Proposition to \eqref{dubstarp.eq}, we have thus shown our claim.\end{proof}

\subsection{Properties of solutions and their gradients on slices}
Our next goal is to study the $t$-regularity of our solutions as well as their properties on `slices', which are sets of the form $\{(x,t) : t = t_0\}$. Let us first note that $t-$derivatives of solutions are solutions.
\begin{proposition}[The $t$-derivatives of solutions are solutions]\label{tsolns.prop}
Let $\om\subset \ree$ be an open set, let $f,\vec F\in L^2_{\loc}(\om)$, and suppose that $u\in W^{1,2}_{\loc}(\om)$ satisfies $\cL u=f-\div\vec F$ in $\om$ in the weak sense. Assume further that $f_t:= \partial_t f\in L^2_{\loc}(\Omega)$ and $\vec F_t:=\partial_t\vec F\in L^2_{\loc}(\om)$. Then the function $v=\partial_t u$ lies in $W^{1,2}_{\loc}(\om)$ and satisfies $\cL v= f_t-\div F_t$ in $\Omega$ in the weak sense.
\end{proposition}
\begin{proof} 
Fix a ball $B\subset 2B\subset \om$ and consider the difference quotients
\begin{equation*}
u_h:=\dfrac{u(\cdot +he_{n+1})-u(\cdot)}{|h|} , \qquad |h|<\dist(B, \partial\om).
\end{equation*} 
We define $f_h$ and $\vec F_h$ similarly. By $t$-independence of the coefficients, we have that $\cL u_h= f_h-\div\vec F_h$ in $B$ for any such $h$. By the Caccioppoli inequality (Proposition \ref{classCaccioppoli.prop}), we obtain that for any $h$ as above, 
\begin{multline*}
\dint_B |\nabla u_h|^2  \lesssim \dint_{2B} \Big( \dfrac{1}{r(B)^2} |u_h|^2 + |\vec F_h|^2+ r(B)^2 |f_h|^2\Big)\\
 \lesssim   \dint_{2B} \Big( \dfrac{1}{r(B)^2} |\partial_t u|^2 + |\vec F_t|^2+ r(B)^2 |f_t|^2\Big).
\end{multline*}
In particular, the difference quotients of $\nabla u$ are bounded, which implies that $\partial_t u\in W^{1,2}_{\loc}(\om)$. Consequently, we must have that the difference quotients $u_h$ converge weakly (in $W^{1,2}_{\loc}(\om)$) to $v=\partial_tu$ (and similarly for $f_h$ and $\vec F_h$). From \eqref{solndef.eq} and the fact that $\cL u_h =f_h-\div\vec F_h$, we conclude that $\cL v = f_t-\div\vec F_t$, as desired.\end{proof}

We now check that $t-$derivatives of solutions are well-behaved on horizontal strips.
\begin{lemma}[Good integrability of the $t$-derivative of a solution on a strip]\label{Le1.5.lem} Denote $\Sigma_a^b: =\big\{ (x,t)\in \ree: a<t<b\big\}$. Suppose that $u$ and $v:= \partial_t u$ are as in Proposition \ref{tsolns.prop} with $\om = \Sigma_a^b$, and suppose further that $ v \in L^2(\Sigma_{a}^{b})$. Then $\nabla v \in L^2(\Sigma_{a'}^{b'})$ for each $a < a' < b' < b$.
\end{lemma}
\begin{proof} Let $\chi_R = \phi(x)\psi(t)$ be a product of infinitely smooth cut-off functions with $0 \le \phi_R, \psi \le 1$,  $\psi \equiv 1$ on $(a', b')$, $\psi \in C_c^\infty(a, b)$, and $\phi_R \equiv 1$ on $B_R$, $\phi_R \in C_c^\infty(B_{2R})$. 
Then, for all $R \gg \min\{a' - a, b - b'\}$, we claim that 
\begin{multline*}
\int_{a'}^{b'} \int_{B_R}|\nabla v|^2\, dx \, dt \lesssim  
\dint_{\ree}\chi_R^2|\nabla v|^2
\\   \lesssim \dint_{\ree}\Big(|v|^2 + |\vec F_t|^2 + |f_t|^2\Big)\big(|\nabla\chi_R|^2 + 1\big)
\\   \lesssim \frac{1}{(\min\{a' - a, b - b', 1\})^2}\int_a^b \int_{\rn}   \Big(|v|^2 + |\vec F_t|^2 + |f_t|^2\Big).
\end{multline*}
We provide the details of the second line in a moment; note that in the third line we used that the dominant contribution for the gradient of $\chi_R$ is its $t$ component when $R$ is large. Sending $R \to \infty$ finishes the proof modulo the aforementioned line.

To see the computation above, denote $\chi := \chi_R$ and observe that
\begin{multline*}
\dint_{\ree}\chi^2|\nabla v|^2 \lesssim \dint_{\bb R^{n+1}}\chi^2\f Re\big( A \nabla v\overline{ \nabla v }\big)\\\leq\f Re\Big[\dint_{\ree}A \nabla v\overline{\nabla (v \chi^2)}  - 2 \dint_{\ree}\chi\overline v  A \nabla v   \nabla \chi\Big]
=:\f Re[I + II].
\end{multline*}
Clearly,
$$|II| \lesssim  \epsilon\dint_{\ree} |\chi \nabla v|^2   + \frac{1}{\epsilon} \dint_{\ree} |\nabla \chi v|^2 ,$$
and the first term can be absorbed to the left-hand side. It remains to handle $I$. We use the equation $\cL v=f_t-\dv\vec F_t$ to write $I = I_1 + I_2 + I_3 + I_4$, where each $I_j$ is a term of the equation and each will be given explicitly below. First, note that
$$|I_4|: =\Big| \dint_{\ree} f_t\overline v \chi^2   \Big| \lesssim \dint_{\ree} |v\chi|^2 + \dint_{\ree} |f_t\chi|^2,$$
which handles this term. Next, we have that
\begin{equation*}
|I_3|:=\Big| \dint_{\ree}\vec F_t\overline{ \nabla(v \chi^2)}   \Big| \lesssim \dint_{\ree} |\vec F_t \nabla v \chi|^2  + \dint_{\ree}|\vec F_t \chi \nabla \chi v|.
\end{equation*}
We handle the first term as in $II$, and we handle the second term as $I_4$. Moving on, we see that
\begin{equation*}
|I_1|  := \Big|\dint_{\ree} B_1 v\overline{\nabla(v\chi^2)}  \Big|  \lesssim  \dint_{\ree} |(B_1 v\chi) \nabla v \chi|   +  \dint_{\ree}|(B_1 v \chi)\nabla\chi v| .
\end{equation*}
Both of the terms above are handled by using the smallness of $B_1$ as in the proof of the Caccioppoli inequality. Now, for the last term, we have that
\begin{equation*}
|I_2|: = \Big| \dint_{\ree} B_2 \nabla v \chi^2\overline v  \Big| \lesssim  \dint_{\ree} \big|(B_2\chi v) \nabla v \chi\big| ,
\end{equation*}
so that we may handle this term exactly as we did $I_1$.\end{proof}

\begin{remark}\label{slicescontinuitysolutionsrmk.rmk} We may bring the above lemma and Lemma \ref{ContSlices.lem} together to conclude that if $u$ solves $\cL u=0$ in $\Sigma_a^b$, then automatically we have the transversal H\"older continuity of its gradient, and  $u \in C_{\loc}^{\alpha'}\Big((a,b), L^{\tfrac{2n}{n-2}}(\rn)\Big)$ for some $\alpha > 0$.
\end{remark}

Next, we present a formula for our equation on a slice.  Recall that
$\tilde A$ denotes the $(n+1)\times n$ submatrix of $A$ consisting of the first $n$ columns of $A$.

\begin{proposition}[Integration by parts on slices for $\cL$]\label{IBPonslicesprop.prop}
Let $u\in Y^{1,2}(\Sigma_a^b)$ and suppose that $\cL u=g$ in $\Sigma_a^b$\, for some $g\in C_c^\infty(\ree)$. Then, for every $t\in (a,b)$ and $\varphi\in W^{1,2}(\rn)$, the identity
\begin{multline*}
\int_\rn\Big((A(x) \nabla u(x,t))_{\|} + (B_1)_\parallel u(x,t)\Big) \cdot  \overline{\nablap\varphi(x)}\, dx  + \int_\rn B_2(x)\cdot \nabla u(x,t) \overline{\varphi(x)}\, dx\\= \int_\rn\Big(\vec{A}_{n+1,\cdot}(x)\cdot \partial_t\nabla u(x,t) + (B_1(x))_\perp \partial_t u(x,t)\Big) \overline{\varphi(x)}\, dx + \int_\rn g(x,t)\overline{\varphi(x)}\, dx
\end{multline*}
holds. If $v,\partial_t v\in Y^{1,2}(\Sigma_a^b)$, and $\cL^* v=0$ in $\Sigma_a^b$ for some $g\in C_c^\infty(\rn)$, then for every $t\in (a,b)$ and $\varphi\in W^{1,2}(\rn)$, the identity 
\begin{multline*}
\int_\rn\Big[\nablap \varphi\cdot \overline{((\overline{B}_2)_{\parallel} v(t))} + \tilde{A}\nablap \varphi \cdot \overline{\nabla v(t)} + B_1 \varphi \cdot \overline{\nabla v(t)}\Big]\\=
  \int_\rn\Big[ \varphi \overline{(\overline{B}_2)_\perp \dno v(t)} +\varphi\vec{A}_{\cdot,n+1}\overline{\nabla\dno v(t)}\Big]
\end{multline*}
holds. Finally, for $v$ and $\varphi$ as above, we also have the identity
\begin{multline*}
\int_\rn \nablap \varphi \cdot \overline{(A^*\nabla v(t))_{\|}}  =  \int_\rn \varphi \cdot \overline{\vec{A}^*_{n+1,\cdot}\dno\nabla v(t)}-\int_{\bb R^n}\nablap \varphi \cdot \overline{(\overline{B}_2)_\parallel v(t)}\\  + \int_\rn\varphi \overline{(\overline{B}_2)_\perp v(t)} -  \int_\rn \varphi \overline{\overline{B}_1 \cdot \nabla v(t)}.
\end{multline*}
\end{proposition}

\begin{proof}
Fix $\varphi\in C_c^\infty(\rn)$ and $t\in (a,b)$. Let $\varphi_\varepsilon(x,s):= \varphi(x)\eta_\varepsilon(t-s)$ with $\varepsilon<\min\{ b-t,t-a\}$, and where $\eta_\varepsilon(\cdot)= \varepsilon^{-1}\eta(\cdot/\varepsilon)$, $\eta\in C_c^\infty(-1,1)$, $\int_{\bb R} \eta=1$. In particular, $\varphi_\varepsilon\in C_c^\infty(\Sigma_a^b)$ is an admissible test function in the definition of the weak solution. Thus, from the definition of $\cL u =g$, we have that
\begin{multline*}
\dint_{\ree}\Big\{\Big(\big(A(x) \nabla u(x,s)\big)_{\|} + (B_1)_\parallel u(x,s)\Big) \cdot  \overline{\nablap\varphi_\varepsilon(x,s)}  +  B_2(x)\cdot \nabla u(x,s) \overline{\varphi_\varepsilon(x,s)}\Big\}\, dxds\\=
 \dint_{\ree}\Big(\vec{A}_{n+1,\cdot}(x)\partial_s\nabla u(x,s) + (B_1(x))_\perp\partial_su(x,s)+g(x,s)\Big)\overline{\varphi_\varepsilon(x,s)}\, dxds.
\end{multline*}
Notice, for instance, that the map
\begin{equation*}
t\mapsto \int_{\rn}\Big(\big(A(x)\nabla u(x,t)\big)_{\parallel}+ (B_1)_\parallel(x)u(x,t)\Big)\cdot \overline{\nablap \varphi(x)}\,dx
\end{equation*}
is continuous in $(a,b)$, owing to Lemma \ref{ContSlices.lem} and the continuity of the duality pairings in each of its entries. A similar statement holds for all the other integrals. The desired conclusion now follows from the fact that for any continuous function $h: (a,b)\to \CC$, we have that $\lim_{\varepsilon\to 0}\int_\RR \eta_\varepsilon(t-\cdot)h  = h(t)$,  for each  $t\in (a,b)$.\end{proof}

As in \cite{AAAHK}, but now employing Lemma \ref{Lpcaccop.prop}, the $t-$independence of our coefficients allows us to obtain $L^p$ estimates on cubes lying in horizontal slices.

\begin{lemma}[$L^p$ estimates on slices; {\cite[Proposition 2.1]{AAAHK}}]\label{Lpcacconslices.lem}
Let $t\in\bb R$, $Q \subset \rn$ be a cube, and $I_Q$ be the box $I_Q=4Q\times(t-\ell(Q),t+\ell(Q))$. Let $p \ge 2$ with $|p -2|$ small enough that the conclusion of Lemma \ref{Snei.lem} holds. Suppose  that  $u\in W^{1,2}(I_Q)$ satisfies $\cL u = 0$ in $I_Q$.  Then the estimates
\begin{equation}\label{Lpgradslices}
\Big( \frac{1}{|Q|} \int_Q |\nabla u (x,t)|^p \Big)^{1/p} \lesssim  \Big( \frac{1}{|Q^{*}|} \dint_{Q^{*}} |\nabla u (x,t)|^p \Big)^{1/p},
\end{equation}
and
\begin{equation}\label{Lpcacconslices}
\Big( \frac{1}{|Q|} \int_Q |\nabla u (x,t)|^p \Big)^{1/p} \lesssim_p \frac{1}{\ell(Q)} \Big( \frac{1}{|Q^{**}|} \dint_{Q^{**}} |u (x,t)|^p \Big)^{1/p}
\end{equation}
hold, where $Q^*:=2Q\times(t-\ell(Q)/4,t+\ell(Q)/4)$ is an $(n+1)-$dimensional rectangle, and $Q^{**}:=3Q\times(t-\ell(Q)/2,t+\ell(Q)/2)$ is a slight dilation of $Q^*$.
\end{lemma}

In \cite{AAAHK}, the analogue of the preceding lemma is proved in the purely second order case.  However,
the argument there extends almost verbatim to the present situation, given Lemma \ref{Lpcaccop.prop}.  We omit the details.

Let us consider how the shift operator acts on $\cL^{-1}$. For each $\tau\in\bb R$, denote by $\n T^{\tau}$ the (positive) \emph{shift by }$\tau$ in the $t-$direction: If $u\in C_c^{\infty}(\bb R^{n+1})$, then $(\n T^{\tau}u)=u(\cdot,\cdot+\tau)$. More generally, if $f\in\n D'$ is a distribution, we define the distribution $\n T^{\tau}f$ by $\langle\n T^{\tau}f,\varphi\rangle=\langle f,\n T^{-\tau}\varphi\rangle$,  for each $\varphi\in\n D$. 

\begin{proposition}\label{prop.tder} Suppose that $u\in W^{1,2}_{\loc}(\bb R^{n+1}_+)$ solves $\cL u=0$ in $\bb R^{n+1}_+$. Then
\begin{enumerate}[i)]
\item\label{item.tderlm} Let $f\in(\Ya)^{\ast}$ and fix $s\in\bb R$. Then $\n T^{s}\cL ^{-1}f\in\Ya$ and satisfies $\n T^s\cL ^{-1}f=\cL ^{-1}\n T^{s}f$.
\item\label{item.tderlocshift} Let $s>0$. Then $\n T^su\in W^{1,2}_{\loc}(\bb R^{n+1}_+)$ and $\cL \n T^su=0$ in $\bb R^{n+1}_+$.
\item\label{item.tderloc} We have that $D_{n+1}u\in W^{1,2}_{\loc}(\bb R^{n+1}_+)$ and $\cL D_{n+1}u=0$ in $\bb R^{n+1}_+$.
\item\label{item.tdershift} For any $s>0$, we have that $D_{n+1}\n T^su\in Y^{1,2}(\bb R^{n+1}_+)\cap L^2(\bb R^{n+1}_+)=W^{1,2}(\bb R^{n+1}_+)$. In particular, for any $t>0$, the trace $\Tr_t D_{n+1}u$ is an element of $H^{\frac12}(\bb R^n)=L^2(\bb R^n)\cap\Hf$. Moreover, for each $t>0$, the estimate
\begin{equation}\label{eq.l2gradslice}
\Vert t\Tr_t\nabla\partial_tu\Vert_{L^2(\bb R^n)}\lesssim\Vert u\Vert_{Y^{1,2}(\bb R^{n+1}_{t/2})}
\end{equation}
holds. In particular, for each $s>0$ we have that
\begin{equation}\label{eq.supl2gradshift}
\sup_{t\geq0}\Vert(t+s)\Tr_t\nabla\partial_t\n T^su\Vert_{L^2(\bb R^n)}\lesssim\Vert u\Vert_{\Y}.
\end{equation}
Finally, for each $t > 0$ and $\zeta\in\Hfm$, we have the identity
\begin{equation}\label{eq.tdermovetder}
(\Tr_tD_{n+1}u,\zeta)=\frac{d\,}{dt}(\Tr_tu,\zeta).
\end{equation}
\end{enumerate}
\end{proposition}

\begin{proof} The proofs of \ref{item.tderlm}), \ref{item.tderlocshift}), and \ref{item.tderloc}) are very similar to the proof of Proposition \ref{tsolns.prop}, and are thus omitted. We prove \ref{item.tdershift}), and to this end fix $s>0$. By assumption, it is clear that $\n T^su\in\Y$, and by \ref{item.tderlocshift}), we have that $\cL \n T^su=0$ in $\bb R^{n+1}_+$. Hence, by \ref{item.tderloc}), we have that $D_{n+1}\n T^su\in W^{1,2}_{\loc}(\bb R^{n+1}_+)$ and $\cL D_{n+1}\n T^su=0$ in $\bb R^{n+1}_+$. Let $\bb G(s/2)$ be a grid of pairwise disjoint cubes $R\subset\bb R^{n+1}_s$ such that $\bb R^{n+1}_s=\cup_{R\in\bb G(s/2)}R$ and $\ell(R)=\frac s2$. Consider the estimate
\begin{gather*}
\dint_{\bb R^{n+1}_+}|\nabla D_{n+1}\n T^su|^2=\dint_{\bb R^{n+1}_s}|\nabla D_{n+1}u|^2=\sum\limits_{R\in\bb G(s/2)}\dint_R|\nabla D_{n+1}u|^2\\ \lesssim\sum\limits_{R\in\bb G(s/2)}\frac1{s^2}\dint_{\tilde R}|D_{n+1}u|^2\lesssim\frac1{s^2}\Vert D_{n+1}u\Vert_{L^2(\bb R^{n+1}_{s/2})}^2\leq\frac1{s^2}\Vert u\Vert_{\Y}^2,
\end{gather*}
which proves that $\nabla D_{n+1}\n T^su\in L^2(\bb R^{n+1}_+)$. Since $D_{n+1}\n T^su\in L^2(\bb R^{n+1}_+)$ by the assumption that $u\in\Y$, it is proven that $D_{n+1}\n T^su\in W^{1,2}(\bb R^{n+1}_+)$. Hence, for each $t\geq0$, $\Tr_t D_{n+1}\n T^su\in H^{\frac12}(\bb R^n)$. But $\Tr_t D_{n+1}\n T^su=\Tr_{t+s} D_{n+1}u$. The estimate (\ref{eq.l2gradslice}) is true by Caccioppoli on slices (Proposition \ref{Lpcacconslices.lem}), as follows: break $\bb R^n$ into a grid $\bb G_n(t/2)$ of cubes $Q\subset\bb R^n$, $\ell(Q)=t/2$, and use Caccioppoli on slices in each cube.

It remains to check the identity (\ref{eq.tdermovetder}), so fix $t>0$. We have seen that $\Tr_{\tau}D_{n+1}u\in\Hf$ for each $\tau>0$. Fix $\zeta\in\Hfm$, and define $g(\tau):=(\Tr_{\tau}u,\zeta)$ for each $\tau>0$. We will show that $g$ is differentiable at $t$, and compute its derivative. To this end, note that
\begin{equation*}
\tfrac{g(t+h)-g(t)}h=\tfrac{(\Tr_{t+h}u,\zeta)-(\Tr_tu,\zeta)}h=\big(\Tr_t\tfrac{\n T^hu-u}h,\zeta\big)=\big(\Tr_0\tfrac{\n T^h\n T^tu-\n T^tu}h,\zeta\big).
\end{equation*}
By our previous computations, we have that $
\frac{\n T^h\n T^tu-\n T^tu}h\longrightarrow D_{n+1}\n T^tu$   in $\Y$   as $h\ra0$, which implies that $\Tr_0\Big(\frac{\n T^h\n T^tu-\n T^tu}h\Big)\longrightarrow\Tr_0D_{n+1}\n T^tu$  in $\Hf$  as $h\ra0$, and hence we have that $\tfrac{g(t+h)-g(t)}h\longrightarrow(\Tr_0 D_{n+1}\n T^tu,\zeta)=(\Tr_t D_{n+1}u,\zeta)$ as $h\ra0$. This finishes the proof.\end{proof}

\section{Abstract Layer Potential Theory}\label{AbsLPthry.sec}

In this section, we develop the abstract layer potential theory. Our methods often closely follow the constructions of Ariel Barton \cite{Bar}; but see also \cite{Ro}.

\begin{definition}[Single layer potential] Define the \emph{single layer potential of $\cL $} as the operator $\m S^\cL :\Hfm\ra\Ya$ given by $\m S^\cL :=\big(\Tr_0\circ (\cL^{-1})^*\big)^*$, which is well defined by virtue of Lemma \ref{lm.ytrace} and Proposition \ref{laxmilgram.prop}. For $t\in\bb R$, we denote $\m S^\cL_t:=\Tr_t\circ\m S^\cL$. When the operator under consideration is clear from the context, we will sometimes drop the superscript, so that we write $\m S=\m S^{\cL}$. For each $t\in\bb R$, ${\bf f}:\bb R^n\ra\bb C^{n+1}$ and $\vec{f}:\bb R^n\ra\bb C^n$, define $(\m S_t^\cL\nabla_{\|})\vec{f}:=-\m S_t^\cL(\text{div}\vec{f})$, $\m S_{t}^{\cL}D_{n+1}:=-\partial_t\m S_{t}^{\cL}$, and 
$(\m S_t^\cL\nabla){\bf f}=(\m S_t^\cL\nabla_{\|}){\bf f}_{\|}+\m S_t^\cL D_{n+1}{\bf f}_{n+1}$.
\end{definition}

Let us elucidate a few properties of this ``abstract'' single layer potential.

\begin{proposition}[Properties of the single layer potential]\label{prop.slproperties} Fix $\gamma\in\Hfm$. The following statements hold.
\begin{enumerate}[i)]
\item\label{item.slpropertieslm} The function $\m S^\cL\gamma\in\Ya$ is the unique element in $\Ya$ such that
\begin{equation}\label{eq.laxmilgramforsl}
B_\cL[\m S^\cL\gamma,\Phi]=\langle\gamma,\Tr_0\Phi\rangle,\qquad\text{for all }\Phi\in\Ya.
\end{equation}
Accordingly, $\m S^L:\Hfm\ra\Ya$ is a bounded linear operator. 
\item\label{item.slpropertiesws} The function $\m S^\cL\gamma$ satisfies $\cL\m S^\cL\gamma=0$ in $\Omega$, where $\Omega=\bb R^{n+1}_+,\bb R^{n+1}_-$.
\item\label{item.slpropertiesbetterws}  Suppose that $\gamma$ has compact support. Then $\cL\m S^\cL\gamma=0$ in $\bb R^{n+1}\backslash\supp \gamma$.
\item\label{item.slpropertiesbound} Define $p_-,p_+$ as in Proposition \ref{prop.fracemb} and suppose that $\gamma\in L^{p_-}(\bb R^n)$. Then the bound $\lVert\Tr_t\m S^{\cL}\gamma\rVert_{L^{p_+}(\rn)} \lesssim \lVert\gamma\rVert_{L^{p_-}(\rn)}$ holds for each $t\in\bb R$.
\item\label{item.slpropertiesadj} For each $t\in\bb R$, the operators $\m S^\cL_t$ and $\m S^{\cL^*}_{-t}$ are adjoint   to one another. That is, for each $\gamma,\psi\in\Hfm$, the identity $\langle\m S^\cL_t\gamma,\psi\rangle=\langle\gamma,\m S^{\cL^*}_{-t}\psi\rangle$ holds.
\item\label{item.slpropertieschar} For each $t\in\bb R$, we have the characterization
\begin{equation}\label{eq.slshift}
\n T^{-t}\m S^\cL\gamma=\big(\Tr_t\circ(\cL^{-1})^*\big)^*.
\end{equation}
\item\label{item.slpropertiesmovetder}  For each $t\in\bb R\backslash\{0\}$, we have that $\Tr_tD_{n+1}\m S^\cL\gamma\in\Hf$. Moreover, for each $t\in\bb R\backslash\{0\}$ and each $\zeta\in\Hfm$, we have that $\langle\Tr_tD_{n+1}\m S^\cL\gamma,\zeta\rangle=\frac{d\,}{dt}\langle\m S^\cL_t\gamma,\zeta\rangle=-\langle\gamma,\Tr_{-t}D_{n+1}\m S^{\cL^*}\zeta\rangle$.
\item\label{item.slpropertiesgradadj} Let $t\in\bb R\backslash\{0\}$. Let ${\bf g}=(\vec{g_{\|}},g_{\perp}):\bb R^n\ra\bb C^{n+1}$ be such that  $g_{\parallel}, g_{\perp}\in C_c^{\infty}(\bb R^n)$. In the sense of distributions, we have the adjoint relation
\begin{equation}\label{eq.slgradadj}
\langle\nabla\m S_t^\cL\gamma,{\bf g}\rangle_{\n D',\n D}=\langle\gamma,(\m S_{-t}^{\cL^*}\nabla){\bf g}\rangle_{\Hfm, \Hf}.
\end{equation}
\end{enumerate}
\end{proposition}

\noindent\emph{Proof.} Fix $\gamma\in\Hfm$. Proof of \ref{item.slpropertieslm}).  Since $\Tr_0:\Ya\ra\Hf$ is a bounded linear operator, then $T_{\gamma}:=\langle\gamma,\Tr_0\cdot\rangle$ is a bounded linear functional on $\Ya$. By the Lax-Milgram theorem, there exists a unique $u_\gamma\in\Ya$ such that $B_{\cL}[u_\gamma,\Phi]=\langle T_\gamma,\Phi\rangle=\langle\gamma,\Tr_0\Phi\rangle$, for all $\Phi\in\Ya$. Now let $\Psi\in(\Ya)^*$ be arbitrary, and observe that
\begin{multline*}
\langle\Psi,\m S^\cL\gamma\rangle=\big\langle\Psi,\big(\Tr_0\circ (\cL^{-1})^*\big)^*\gamma\big\rangle=\big\langle\Tr_0\circ (\cL^{-1})^*\Psi,\gamma\big\rangle=\overline{\langle T_\gamma,(\cL^*)^{-1}\Psi\rangle}\\=\overline{B_\cL[u_\gamma,(\cL^*)^{-1}\Psi]}=\overline{\langle\cL u_\gamma,(\cL^*)^{-1}\Psi\rangle}=\overline{\langle u_\gamma,\Psi\rangle}=\langle \Psi,u_\gamma\rangle.
\end{multline*}

Proof of \ref{item.slpropertiesws}). Let $\Phi\in C_c^{\infty}(\bb R^{n+1}_+)$, and let $\tilde\Phi$ be an extension of $\Phi$ to $C_c^{\infty}(\bb R^{n+1})$ with $\tilde\Phi\equiv0$ on $\bb R^{n+1}\backslash\supp\Phi$. In particular, $\Tr_0\tilde\Phi\equiv0$. Then (\ref{eq.laxmilgramforsl}) gives that $B_\cL[\m S^\cL\gamma,\Phi]=B_\cL[\m S^\cL\gamma,\tilde\Phi]=0$. Since $\Phi$ was arbitrary, the claim follows.

Proof of \ref{item.slpropertiesbetterws}). Let $\Omega:=\bb R^{n+1}\backslash\supp \gamma$, and let $\Phi\in C_c^{\infty}(\Omega)$. Let $\tilde\Phi$ be an extension of $\Phi$ to $C_c^{\infty}(\bb R^{n+1})$ with $\tilde\Phi\equiv0$ on $\bb R^{n+1}\backslash\supp \gamma$. In particular, the  supports of $\tilde\Phi$ and $\gamma$ are disjoint. It follows that $\langle\gamma,\Tr_0\tilde\Phi\rangle=0$. Using (\ref{eq.laxmilgramforsl}) now yields the result.

Proof of \ref{item.slpropertiesbound}). By the boundedness of $\s^{\cL}$  and the Sobolev embeddings, we have that
\[
\lVert \s^{\cL}_t g  \rVert_{L^{p_+}(\rn)} \lesssim \lVert \s^{\cL}_t g \rVert_{\Hf} \lesssim \| \s^{\cL} g\|_{\yot} \lesssim \lVert g \rVert_{\Hfm} \lesssim \lVert g \rVert_{L^{p_-}(\rn)}.
\]

Proof of \ref{item.slpropertiesadj}). Fix $t\in\bb R$ and $\gamma,\zeta\in\Hfm$. By the Lax-Milgram theorem, there exists a unique $v^{\zeta,t}\in\Ya$ such that $B_{\cL^*}[v^{\zeta,t},\Phi]=\langle\zeta,\Tr_t\Phi\rangle$, for all $\Phi\in\Ya$. Observe that
\begin{align*}
\langle\Tr_t\m S^\cL\gamma,\zeta\rangle=\overline{\langle\zeta,\Tr_t\m S^\cL\gamma\rangle}=\overline{B_{\cL^*}[v^{\zeta,t},\m S^\cL\gamma]}=B_\cL[\m S^\cL\gamma,v^{\zeta,t}]=\langle\gamma,\Tr_0v^{\zeta,t}\rangle. 
\end{align*}
Thus it suffices to show that $\Tr_0v^{\zeta,t}$ and $\m S^{\cL^*}_{-t}\zeta$ coincide as elements in $\Hf$. In turn, this will follow if we prove that $\m S^{\cL^*}\zeta=\n T^tv^{\zeta,t}=v^{\zeta,t}(\cdot,\cdot+t)$,  in $\Ya$. Let $\Phi\in\Ya$ be arbitrary. Note then that $\n T^t\Phi$ also lies in $\Ya$. By the $t-$independence of the coefficients of  $\cL$ and a change of variables we have that
\begin{gather}\nonumber
B_{\cL^*}[\n T^tv^{\zeta,t},\n  T^t\Phi]=B_{\cL^*}[v^{\zeta,t},\Phi]=\langle\gamma,\Tr_t\Phi\rangle=\langle\gamma,\Tr_0\n T^t\Phi\rangle.
\end{gather}
By (\ref{eq.laxmilgramforsl}) with $\cL$ replaced by $\cL^*$ throughout, $\m S^{\cL^*}\zeta$ is the unique element of $\Ya$ for which the above identity can  hold for all $\Phi\in\Ya$, as desired.

Proof of \ref{item.slpropertieschar}). In \ref{item.slpropertiesadj}), we proved that for each $\gamma\in\Hfm$, $\m S^\cL\gamma=\n T^t\cL^{-1}(T_{\gamma}^t)$, where $T_{\gamma}^t\in(\Ya)^*$ is given by $\langle T_{\gamma}^t,\Phi\rangle=\langle\gamma,\Tr_t\Phi\rangle$ for $\Phi\in\Ya$. Hence $\n T^{-t}\m S^\cL\gamma=\cL^{-1}(T_{\gamma}^t)$. Reproduce the proof of \ref{item.slpropertieslm}) in reverse to obtain the claim.

Proof of \ref{item.slpropertiesmovetder}). Let $t>0$ (the case $t<0$ is analogous). By \ref{item.slpropertiesws}) we have that $\cL\m S^\cL\gamma=0$ in $\bb R^{n+1}_+$. Therefore, using Proposition \ref{prop.tder} \ref{item.tdershift}) we have that $\Tr_{\tau}D_{n+1}\m S^\cL\gamma\in\Hf$ for each $\tau>0$. Using (\ref{eq.tdermovetder}) and \ref{item.slpropertiesadj}), we calculate that
\begin{multline*}
\tfrac{d\,}{d\tau}\langle\Tr_{\tau}\m S^\cL\gamma,\zeta\rangle\big|_{\tau=t}=\overline{\tfrac{d\,}{d\tau}\langle\zeta,\Tr_{\tau}\m S^\cL\gamma\rangle}\big|_{\tau=t}=\overline{\tfrac{d\,}{d\tau}\langle\Tr_{-\tau}\m S^{\cL^*}\zeta,\gamma\rangle}\big|_{\tau=t}\\ =-\overline{\tfrac{d\,}{d(-\tau)}\langle\Tr_{-\tau}\m S^{\cL^*}\zeta,\gamma\rangle}\big|_{-\tau=-t}=-\overline{\langle\Tr_{-t}D_{n+1}\m S^{\cL^*}\zeta,\gamma\rangle}=-\langle\gamma,\Tr_{-t}D_{n+1}\m S^{\cL^*}\zeta\rangle.
\end{multline*}

Proof of \ref{item.slpropertiesgradadj}). It is clear by an easy induction procedure that (\ref{item.slpropertiesmovetder}) holds for higher $t-$derivatives in the expected manner. Note that
\begin{multline*}
\langle\nabla\m S^\cL_t\gamma,{\bf g}\rangle_{\n D',\n D}=\langle\nabla_{\|}\m S^\cL_t\gamma,\vec{g_{\|}}\rangle_{\n D',\n D}+\langle\Tr_tD_{n+1}\m S^\cL\gamma,g_{\perp}\rangle_{\n D',\n D}\\ =-\langle\m S^\cL_t\gamma,\divg\vec{g_{\|}}\rangle_{\n D',\n D}-\langle\gamma,\Tr_{-t}D_{n+1}\m S^{\cL^*}g_{\perp}\rangle_{\Hfm,\Hf}\\ =-\langle\gamma,\m S^{\cL^*}_{-t}\divg\vec{g_{\|}}\rangle_{\Hfm,\Hf}+\langle\gamma,(\m S^{\cL^*}_{-t}D_{n+1})g_{\perp}\rangle_{\Hfm,\Hf} =\langle\gamma,(\m S^{\cL^*}_{-t}\nabla){\bf g}\rangle.
\end{multline*}\hfill{$\square$}

In preparation for defining the double layer potential, let us make the following remark.

\begin{remark*} Given $\varphi\in\Hf$, there exists $\Phi\in\Ya$ with $\Tr_0\Phi=\varphi$ and \\ $\Vert\Phi\Vert_{\Ya}\lesssim\Vert\varphi\Vert_{\Hf}$.
\end{remark*}

For a fixed $u\in\Y$, let $\n F^+_u$ be the functional on $\Ya$ defined by
\[
\langle\n F^+_u,v\rangle:=B_{\cL, \bb R^{n+1}_+}[u,v]=\dint\limits_{\bb R^{n+1}_+}\Big[A\nabla u\cdot\overline{\nabla v}+B_1u\cdot\overline{\nabla v}+B_2\cdot\nabla u\overline{v}\Big],
\]
for each $v\in\Ya$. Then $\n F^+_u$ is clearly bounded on $\Ya$. We define $B_{\cL,\bb R^{n+1}_-}$ and $\n F^-_u$ in a similar way (using $\bb R^{n+1}_-$ instead of $\bb R^{n+1}_+$), and we note that if $u\in\Ya$, then $\cL u=\n F^+_u+\n F^-_u$.

\begin{definition}[Double layer potential] Given $\varphi\in\Hf$, let $\Phi\in\Ya$ be any extension of $\varphi$ to $\bb R^{n+1}$. Define $\m D^{\cL,+}(\varphi):=-\Phi\big|_{\bb R^{n+1}_+}~+~\cL^{-1} (\n F^+_{\Phi} )~\big|_{\bb R^{n+1}_+}$ (see below for a proof that this is well defined). We call the operator $\m D^{\cL,+}:\Hf\ra\Y$ the \emph{double layer potential} associated to operator $\cL$ on the upper half-space. Analogously, we define $\m D^{\cL,-}$, the double-layer potential associated to operator $\cL$ on the lower half-space, by extending $\varphi$ to $\bb R^{n+1}_-$. We define $\m D^{\cL^*,\pm}$ similarly, by replacing $\cL$ with $\cL^*$.
\end{definition}

\begin{proposition}[Properties of the double layer potential]\label{prop.dlproperties} Fix $\varphi\in\Hf$ and let $\Phi$ be any $\Ya-$extension of $\varphi$ to $\bb R^{n+1}$ with $\Tr_0\Phi=\varphi$. The following statements hold.
\begin{enumerate}[i)]
\item\label{item.dlpropertieswd} The double layer potential $\m D^{\cL,+}$ is well defined.
\item\label{item.dlpropertieschar} We have the characterizations
\begin{equation}\label{eq.dl2}
\m D^{\cL,+}\varphi=-\cL^{-1}(\n F^-_\Phi)\big|_{\bb R^{n+1}_+},\qquad\m D^{\cL,-}\varphi=-\cL^{-1}(\n F^+_\Phi)\big|_{\bb R^{n+1}_-}.
\end{equation}
\item\label{item.dlpropertiesbound} The bound $\Vert\m D^{\cL,+}\varphi\Vert_{Y^{1,2}(\bb R^{n+1}_{+})}\lesssim\Vert\varphi\Vert_{\Hf}$	holds.
\item\label{item.dlpropertiesws} The function $\m D^{\cL,+}\varphi$ satisfies $\cL\m D^{\cL,+}\varphi=0$ in the weak sense in $\bb R^{n+1}_+$.
\end{enumerate}
\end{proposition}

\noindent\emph{Proof.} Proof of \ref{item.dlpropertieswd}). Let $\Phi,\Phi'\in\Ya$ be any two extensions of $\varphi$ to $\bb R^{n+1}$. Then $(\Phi-\Phi')(\cdot,0)=0$. If $w$ is defined as $w|_{\bb R^{n+1}_+}=\Phi-\Phi'$ with $w|_{\bb R^{n+1}_-}\equiv0$, then $w\in\Ya$. Thus observe that $\langle\cL w,\Psi\rangle=B_\cL[w,\Psi]=\langle\n F^+_{\Phi-\Phi'},\Psi\rangle$, for all $\Psi\in\Ya$, whence we conclude that $w=\cL^{-1}(\n F^+_{\Phi-\Phi'})$. Hence
\begin{multline*}
\big[-\Phi~+~\cL^{-1}(\n F^+_{\Phi})\big]_{\bb R^{n+1}_+}-\big[-\Phi'~+~\cL^{-1}(\n F^+_{\Phi'} ) \big]_{\bb R^{n+1}_+}\\ =\big[\Phi'-\Phi+\cL^{-1}(\n F^+_{\Phi}-\n F^+_{\Phi'})\big]_{\bb R^{n+1}_+}=\big[\Phi'-\Phi+\cL^{-1}(\n F^+_{\Phi-\Phi'})\big]_{\bb R^{n+1}_+}\equiv0.
\end{multline*}

Proof of \ref{item.dlpropertieschar}). Simply note that
\begin{equation}\nonumber
\m D^{\cL,+}\varphi=\big[-\Phi+\cL^{-1}(\n F^+_{\Phi})\big]_{\bb R^{n+1}_+}=\big[\cL^{-1}(-\cL\Phi+\n F^+_{\Phi} )\big]_{\bb R^{n+1}_+}=\big[\cL^{-1}(-\n F^-_{\Phi} )\big]_{\bb R^{n+1}_+}.
\end{equation}

Proof of \ref{item.dlpropertiesbound}). Owing to (\ref{eq.dl2}) we write
\begin{align}\nonumber
\Vert\m D^{\cL,+}\varphi\Vert_{\Y}=\Vert \cL^{-1}(\n F_{\Phi}^-)\Vert_{\Y}\lesssim\Vert\n F^-_{\Phi}\Vert_{(\Ya)^*}.
\end{align}
Let $0\neq\Psi\in\Ya$. We have
\begin{equation}\nonumber
|(\n F^-_{\Phi},\Psi)|=|B_{\cL,\bb R^{n+1}_-}[\Phi,\Psi]|\lesssim\Vert\Phi\Vert_{\Ym}\Vert\Psi\Vert_{\Ya},\nonumber
\end{equation}
whence we deduce that $\Vert\n F^-_{\Phi}\Vert_{(\Ya)^*}\lesssim\Vert\Phi\Vert_{\Ym}\lesssim\Vert\varphi\Vert_{\Hf}$. Putting these estimates together we obtain the desired result.

Proof of \ref{item.dlpropertiesws}). Let $\Psi\in C_c^{\infty}(\bb R^{n+1}_+)$ and extend it as a function in $\Psi\in C_c^{\infty}(\bb R^{n+1})$ so that $\Psi \equiv 0$ in $\bb R^{n+1}_-$. Observe that
\begin{multline}
B_{\cL,\bb R^{n+1}_+}[\m D^{\cL,+}\varphi,\Psi]=B_{\cL,\bb R^{n+1}_+}[-\cL^{-1}(\n F^-_{\Phi}),\Psi]=B_\cL[-\cL^{-1}(\n F^-_{\Phi}),\Psi]\nonumber\\ =-\langle\n F^-_{\Phi},\Psi\rangle=-B_{\cL,\bb R^{n+1}_-}[\Phi,\Psi]\equiv0.
\end{multline}\hfill{$\square$}

We may now introduce the definition of the conormal derivative. First let us make the quick observation that since $Y^{1,2}_0(\bb R^{n+1}_+)\hookrightarrow\Y$, then we have a surjection $(Y^{1,2}(\bb R^{n+1}_+))^*\ra(Y^{1,2}_0(\bb R^{n+1}_+))^*$ given by restriction of the test space for the functional. In particular, if $f\in(Y^{1,2}(\bb R^{n+1}_+))^*$, then we can also think of $f\in(Y^{1,2}_0(\bb R^{n+1}_+))^*$.

\begin{definition}[Conormal derivative]\label{def.conormal} Suppose that $u\in\Y$, $f\in(Y^{1,2}(\bb R^{n+1}_+))^*$ (note carefully that this space is not $(Y^{1,2}_0(\bb R^{n+1}_+))^*$), and that $\cL u=f$ in $\bb R^{n+1}_+$ in the  sense   that for each $\Phi\in C_c^{\infty}(\bb R^{n+1}_+)$, the identity 
\begin{equation}\label{solndef.eq2}
B_{\cL, \bb R^{n+1}_+}[u,\Phi]=\langle f,\Phi\rangle_{(Y_0^{1,2}(\bb R^{n+1}_+))^*, Y_0^{1,2}(\bb R^{n+1}_+)}
\end{equation}
holds.  Define the \emph{conormal derivative} of $u$ associated to $\cL$ with respect to $\bb R^{n+1}_+$, $\partial_{\nu}^{\cL,+}u\in\Hfm$, by
\begin{equation}\nonumber
\langle\partial_{\nu}^{\cL,+}u,\varphi\rangle=B_{\cL,\bb R^{n+1}_+}[u,\Phi]-\langle f,\Phi\rangle_{(Y^{1,2}(\bb R^{n+1}_+))^*,\Y}~,\qquad\varphi\in\Hf,
\end{equation}
where $\Phi\in\Y$ is any bounded extension of $\varphi$ to $\bb R^{n+1}_+$. Note that we also define the objects $\partial_{\nu}^{\cL^*,+}u,\partial_{\nu}^{\cL,-}u,\partial_{\nu}^{\cL^*,-}u$ similarly. 
\end{definition}

When $f=\tilde f-\dv\tilde F$ and $\tilde f, |\tilde F|$ verify the assumptions in (\ref{eq.extendcond}) (with $\Omega=\bb R^{n+1}_+$, $D=\bb R^n$, and $I=(0,\infty)$), then  the sense (\ref{solndef.eq2}) of weak solutions coincides with the one previously given in Definition \ref{weakform.def} (see Remark \ref{rm.extend}). In particular, if $f\equiv0$, the two senses (\ref{solndef.eq}), (\ref{solndef.eq2}) of weak solutions coincide, and there is no ambiguity.

Let us show that $\partial_{\nu}^{\cL,+}u$ is well defined. Let $\Phi,\Phi'$ be $\Y-$extensions of $\varphi$ with $\Tr_0\Phi=\Tr_0\Phi'=\varphi$. Then $\Phi-\Phi'\in Y^{1,2}_0(\bb R^{n+1}_+)$, and so
\[
B_{\cL,\bb R^{n+1}_+}[u,\Phi]-B_{\cL,\bb R^{n+1}_+}[u,\Phi']=B_{\cL,\bb R^{n+1}_+}[u,\Phi-\Phi']=\langle f,\Phi-\Phi'\rangle_{(Y^{1,2}_0(\bb R^{n+1}_+))^*, Y^{1,2}_0(\bb R^{n+1}_+)}
\]
since $u$ solves $\cL u=f$ in $\bb R^{n+1}_+$ in the sense   (\ref{solndef.eq2}). Finally, observe that
\begin{multline}\nonumber
\langle f,\Phi\rangle_{(Y^{1,2}(\bb R^{n+1}_+))^*,\Y}-\langle f,\Phi'\rangle_{(Y^{1,2}(\bb R^{n+1}_+))^*,\Y}\\=\langle f,\Phi-\Phi'\rangle_{(Y^{1,2}(\bb R^{n+1}_+))^*,\Y}=\langle f,\Phi-\Phi'\rangle_{(Y^{1,2}_0(\bb R^{n+1}_+))^*, Y^{1,2}_0(\bb R^{n+1}_+)},
\end{multline}
so that, upon subtracting these two identities, we see that $\partial_{\nu}^{\cL,+}u$ does not depend on the particular extension $\Phi$ taken. It remains to show that  $\partial_{\nu}^{\cL,+}u\in\Hfm$. Observe that
\begin{multline*}
|\langle\partial_{\nu}^{\cL,+}u,\varphi\rangle| \leq|B_{\cL,\bb R^{n+1}_+}[u,\Phi]|+|\langle f,\Phi\rangle_{(Y^{1,2}(\bb R^{n+1}_+))^*,\Y}|\\  \lesssim\big(\Vert u\Vert_{\Y}+\Vert f\Vert_{(Y^{1,2}(\bb R^{n+1}_+))^*}\big)\Vert\Phi\Vert_{\Y}\\  \lesssim\big(\Vert u\Vert_{\Y}+\Vert f\Vert_{(Y^{1,2}(\bb R^{n+1}_+))^*}\big)\Vert\varphi\Vert_{\Hf}.
\end{multline*}

It will also be helpful to consider conormal derivatives on slices other than $t=0$, denoted $\partial_{\nu,t}^{\cL,\pm}$. The definition is entirely analogous.

The following identities tie these  definitions of the conormal derivatives together.

\begin{lemma}\label{lm.conormals} Let $\gamma\in\Hfm$. The following statements are true.
	\begin{enumerate}[i)]
		\item\label{item.conormals+} Suppose that $u\in\Y$ solves $\cL u=0$ in $\bb R^{n+1}_+$ in the weak sense. Then for any $t>0$, $\partial_{\nu}^{\cL,+}\n T^tu=\partial_{\nu,t}^{\cL,+}u$. Moreover, for any $t>0$, $\partial_{\nu,t}^{\cL,+}u\in L^2(\bb R^n)$, and we have the identity
		\begin{equation}\label{eq.conormalchar}
		\partial_{\nu,t}^{\cL,+}u=-e_{n+1}\cdot\Tr_t[A\nabla u+B_1u]\qquad\text{in }L^2(\bb R^n).
		\end{equation}
		\item\label{item.conormals-} Suppose that $u\in\Ym$ solves $\cL u=0$ in $\bb R^{n+1}_-$. Then for any $t>0$, $\partial_{\nu}^{\cL,-}\n T^{-t}u=\partial_{\nu,-t}^{\cL,-}u$.
		\item\label{item.conormalspm} Let $t>0$. Then for each $\gamma\in\Hfm$, the identity $-\partial_{\nu,-t}^{\cL,-}\m S^\cL\gamma=\partial_{\nu,-t}^{\cL,+}\m S^\cL\gamma$	holds in the space $\Hfm$.
	\end{enumerate}
\end{lemma}

\noindent\emph{Proof.} Proof of \ref{item.conormals+}) and \ref{item.conormals-}).   Let $\varphi\in\Hf$, and $\Phi\in\Y$ is any extension of $\varphi$. Thenm
\begin{multline*}
\langle\partial_{\nu}^{\cL,+}\n T^{t}u,\varphi\rangle=B_{\cL,\bb R^{n+1}_+}[\n T^tu,\Phi]=B_{\cL,\bb R^{n+1}_t}[u,\n T^{-t}\Phi]\\ =\langle\partial_{\nu,t}^{\cL,+}u,\Tr_t\n T^{-t}\Phi\rangle=\langle\partial_{\nu,t}^{\cL,+}u,\varphi\rangle.
\end{multline*}
We turn our attention now to (\ref{eq.conormalchar}). By Remark \ref{slicescontinuitysolutionsrmk.rmk}, we have that $F(x,t)=-e_{n+1}\cdot\Tr_t[A\nabla u+B_1u ]$ is continuous in $t$ taking values in $L^2(\bb R^n)$. In order to prove the lemma we will regularize our coefficients and solution simultaneously.

Let $P_{\ep}$ be an $(n+1)-$dimensional approximate identity; that is, $P_{\ep}(f)=\eta_{\ep}*f$, where $\eta_{\ep}(X)=\frac1{\ep^{n+1}}\eta(X/\ep)$ ($X\in\bb R^{n+1}$), $\eta\in C_c^{\infty}(B(0,1))$, $\eta$ non-negative and radially decreasing with $\int_{\bb R^{n+1}}\eta=1$. We claim that
\begin{equation}\label{eq.claimpe}
-e_{n+1}\cdot P_{\ep}(A\nabla u+B_1u)(x,t_0)\longrightarrow-e_{n+1}\cdot(A\nabla u+B_1u)(x,t_0)
\end{equation}
strongly in $L^2(\bb R^n)$. Assume (\ref{eq.claimpe}) for a moment. Then to show \ref{item.conormals+}) and \ref{item.conormals-}) in the lemma, it is enough to show that for every $\Phi\in C_c^{\infty}(\bb R^{n+1})$ with $\Phi(x,t_0)=\varphi(x)$, we have that
\begin{equation}\nonumber
\lim_{\ep\ra0}\int_{\bb R^n}-e_{n+1}\cdot P_{\ep}(A\nabla u+B_1u)(x,t_0)\overline{\varphi(x)}\,dx=\dint_{\bb R^{n+1}_{t_0}}A\nabla u\overline{\nabla\Phi}+B_1u\overline{\nabla\Phi}+B_2\cdot\nabla u\overline{\Phi}.
\end{equation}
To prove the above equality, first define for any cube $Q\subset\bb R^n$, $R_Q^{t_0}:=Q\times[t_0,t_0+\ell(Q)]$. Now choose any cube $Q\subset\bb R^n$ such that $\supp{\Phi}\cap\{t\geq t_0\}\subset\overline{R_{\frac12Q}^{t_0}}$. Integrating by parts, we have for $0<\ep\ll\min\{\ell(Q),t_0\}$ the identity
\begin{multline}\label{eq.passtolimit}
\int_{\bb R^n}-e_{n+1}\cdot P_{\ep}(A\nabla u+B_1u)(x,t_0)\overline{\varphi(x)}\,dx=\dint_{R_Q^{t_0}}\divg [P_{\ep}(A\nabla u+B_1u)\overline{\Phi} ]\\ =\dint_{R_Q^{t_0}}\divg [P_{\ep}(A\nabla u+B_1u) ]\overline{\Phi}+\dint_{R_Q^{t_0}}P_{\ep}(A\nabla u+B_1u)\cdot\overline{\nabla\Phi}.
\end{multline}
Now let $X=(x,t)\in\supp\Phi\cap\{t\geq t_0\}$, and $\ep<\frac{t_0}2$. Then, since $\cL u=0$ in $\bb R^{n+1}_+$, we have that
\begin{multline*}
\divg_X[P_{\ep}(A\nabla u+B_1u)](X) =\divg_X\dint_{\bb R^{n+1}}\eta_{\ep}(X-Y)(A\nabla_Y u+B_1u)(Y)\,dY\\  =-\dint_{\bb R^{n+1}}\nabla_Y\eta_{\ep}(X-Y)(A\nabla_Y u+B_1u)(Y)\,dY\\  =\dint_{\bb R^{n+1}_+}\eta_{\ep}(X-Y)B_2\nabla_Yu(Y)\,dY=P_{\ep}(B_2\nabla u)(X),
\end{multline*}
and therefore the identity
\begin{equation}\label{eq.claimpe2}
\dint_{R_Q^{t_0}}\divg [P_{\ep}(A\nabla u+B_1u) ]\overline{\Phi}=\dint_{R_Q^{t_0}}P_{\ep}(B_2\nabla u)\overline{\Phi}
\end{equation}
holds. Finally, we want to pass in the limit as $\ep\ra0$ the identity (\ref{eq.passtolimit}) while using (\ref{eq.claimpe2}), so we use the Lebesgue dominated convergence theorem. Observe that for some $p>1$, $|A\nabla u|+|B_1u|+|B_2\nabla u|\in L^p(U_Q^{t_0})$, where $U_Q^{t_0}:=R_Q^{t_0}+B(0,\frac{t_0}4)$ (the $\frac{t_0}4$-neighborhood of $R_Q^{t_0}$). It follows that for $\ep\in(0,\frac{t_0}4)$,
\[
P_{\ep}(A\nabla u+B_1u)(x,t)+P_{\ep}(B_2\nabla u)(x,t)\leq\n M\big(\big[|A\nabla u|+|B_1u|+|B_2\nabla u|\big]{\mathbbm 1}_{U_Q^{t_0}}\big)(x,t)
\]
for all $(x,t)\in R_Q^{t_0}$, where $\n M$ is the usual Hardy-Littlewood maximal operator in $\ree$. Hence we have that
\begin{multline*}
\lim_{\ep\ra0}\int_{\bb R^n}-e_{n+1}\cdot P_{\ep}(A\nabla u+B_1u)(x,t_0)\overline{\varphi(x)}\,dx\\ =\lim_{\ep\ra0}\dint_{R_Q^{t_0}}P_{\ep}(A\nabla u +B_1u)\overline{\nabla\Phi}+P_{\ep}(B_2\nabla u)\overline{\Phi}\\ =\dint_{R_Q^{t_0}}A\nabla u +B_1u\overline{\nabla\Phi}+B_2\nabla u\overline{\Phi}.
\end{multline*}

Thus it remains to prove (\ref{eq.claimpe}). Set $F_{\ep}(x,t):=-e_{n+1}\cdot P_{\ep}(A\nabla u+B_1u)(x,t)$, $F(x,t):=-e_{n+1}\cdot(A\nabla u+B_1u)(x,t)$. For $\ep<\frac{t_0}2$,   we have that
\begin{multline*}
\limsup_{\ep\ra0}\Vert F_{\ep}(\cdot,t_0)-F_0(\cdot,t_0)\Vert_{2}\\ =\limsup_{\ep\ra0}\Big(\int_{\bb R^n}\Big|\dint_{\bb R^{n+1}}\big[F_0(x-\ep y,t_0-\ep s)-F_0(x,t_0)\big]\eta_{\ep}(y,s)\,dy\,ds\Big|^2\,dx\Big)^{\frac12}\\ \leq\limsup_{\ep\ra0}\dint_{\bb R^{n+1}}\eta(y,s)\Big[\int_{\bb R^n}|F_0(x-\ep y,t_0-\ep s)-F_0(x,t_0)|^2\,dx\Big]^{\frac12}\,dy\,ds\\ \leq\limsup_{\ep\ra0}\sup_{|y|,|s|<1}\Vert F_0(\cdot-\ep y,t_0-\ep s)-F_0(\cdot,t_0)\Vert_2\\ \leq\limsup_{\ep\ra0}\sup_{|\hat y|,|\hat s|<\ep}\Vert F_0(\cdot-\hat y,t_0-\hat s)-F_0(\cdot-\hat y,t_0)\Vert_2+\Vert F_0(\cdot-\hat y,t_0)-F_0(\cdot,t_0)\Vert_2,
\end{multline*}
which drops to $0$ as $\ep\ra0$, finishing the proof.

Proof of \ref{item.conormalspm}). Let $\varphi\in\Hf$ and let $\Phi\in\Ya$ be any extension of $\varphi$. Note that $\cL\m S^{\cL}\gamma=0$ in $\bb R^{n+1}_{-,-t}$, while $\cL\m S^\cL\gamma=T_{\gamma}$ in $\bb R^{n+1}_{+,-t}$  in the sense (\ref{solndef.eq2}), where $T_{\gamma}\in(\Ya)^*$ is the distribution given by $\langle T_{\gamma},\Psi\rangle=\langle \gamma,\Tr_0\Psi\rangle$, for $\Psi\in\Ya$. Then, 
\begin{multline*}
\langle\partial_{\nu,-t}^{\cL,+}\m S^\cL\gamma,\varphi\rangle=B_{\cL,\bb R^{n+1}_{+,-t}}[\m S^\cL\gamma,\Phi]-\langle\gamma,\Tr_0\Phi\rangle\\ =-B_{\cL,\bb R^{n+1}_{-,-t}}[\m S^\cL\gamma,\Phi]+B_{\cL}[\m S^\cL\gamma,\Phi]-\langle\gamma,\Tr_0\Phi\rangle\\ =-B_{\cL,\bb R^{n+1}_{-,-t}}[\m S^\cL\gamma,\Phi]+\langle\gamma,\Tr_0\Phi\rangle-\langle\gamma,\Tr_0\Phi\rangle=-\langle\partial_{\nu,-t}^{\cL,-}\m S^\cL\gamma,\varphi\rangle.
\end{multline*}\hfill{$\square$}

\subsection{Green's formula and jump relations}

Let us remark that the functional $\n F^+_u$ makes sense even if we only have that $u\in\Y$ and $u\notin\Ya$. Also, if $\Omega\subset\bb R^{n+1}$  is an open set with Lipschitz boundary, and $f\in(Y^{1,2}(\Omega))^*$, define the functional ${\mathbbm 1}_{\Omega}f \in(Y^{1,2}(\bb R^{n+1}))^*$ by $\langle{\mathbbm 1}_{\Omega}f,\Psi\rangle:=\langle f,{\mathbbm 1}_{\Omega}\Psi\rangle$  for each $\Psi\in Y^{1,2}(\bb R^{n+1})$.

\begin{theorem}[Green's formula] \label{thm.greenformula} Suppose that $u\in\Y$ solves $\cL u=f$ in $\bb R^{n+1}_+$ for some $f\in(Y^{1,2}(\bb R^{n+1}_+))^*$  in the sense (\ref{solndef.eq2}). Then the following statements hold.
	\begin{enumerate}[i)]
		\item\label{item.slgreen} We have the identity
		\begin{equation}\label{eq.slgreen}
		\m S^\cL(\partial_{\nu}^{\cL,+}u)=\cL^{-1}(\n F^+_u)-\cL^{-1}({\mathbbm 1}_{\bb R^{n+1}_+}f)\qquad\text{in }\Ya.
		\end{equation}
		\item\label{item.greenformula} The identity $u=-\m D^{\cL,+}(\Tr_0u)~+~\m S^\cL(\partial_{\nu}^{\cL,+}u)|_{\bb R^{n+1}_+}~+~\cL^{-1}({\mathbbm 1}_{\bb R^{n+1}_+} f)|_{\bb R^{n+1}_+}$	holds in   $\Y$.
		\item\label{item.greenformula2} We have that $-\cL^{-1}({\mathbbm 1}_{\bb R^{n+1}_+}f)|_{\bb R^{n+1}_-}=\m D^{\cL,-}(\Tr_0 u)+\m S^\cL(\partial_{\nu}^{\cL,+}u)|_{\bb R^{n+1}_-}$  in   $\Ym$.
		\item\label{item.greenformula3} Suppose that $\cL u=0$ in $\ree_-$. Then   $\cD^{\cL,+} (\tr_0 u)= -\s^{\cL}(\partial_{\nu}^{\cL,-}u)$	holds in $\reu$.
	\end{enumerate}
\end{theorem}

\noindent\emph{Proof.} Proof of \ref{item.slgreen}). Let $\Psi\in(\Ya)^*$. Then
\begin{multline*}
\langle\Psi,\m S^\cL\partial_{\nu}^{\cL,+}u\rangle= \langle\Tr_0(\cL^*)^{-1}\Psi,\partial_{\nu}^{\cL,+}u\rangle \\ =\overline{B_{\cL,\bb R^{n+1}_+}[u,(\cL^*)^{-1}\Psi]}-\overline{\langle f,(\cL^*)^{-1}\Psi\rangle}_{(\Y)^*,\Y} \\ =\overline{\langle\n F^+_u, (\cL^{-1})^*\Psi\rangle}-\overline{\langle{\mathbbm 1}_{\bb R^{n+1}_+}f,(\cL^{-1})^*\Psi\rangle}=\overline{\langle\cL^{-1}(\n F^+_u), \Psi\rangle}-\overline{\langle\cL^{-1}({\mathbbm 1}_{\bb R^{n+1}_+}f),\Psi\rangle} \\ =\langle\Psi,\cL^{-1}(\n F^+_u)-\cL^{-1}({\mathbbm 1}_{\bb R^{n+1}_+}f)\rangle.
\end{multline*}

Proof of \ref{item.greenformula}). Let $\Psi\in(\Ya)^*$ have compact support within $\bb R^{n+1}_+$. Using (\ref{eq.slgreen}), we have that
\begin{multline*}
\langle\Psi,\m S^\cL\partial_{\nu}u-\m D^{\cL,+}\Tr_0u\rangle \\ =\langle\Psi,\cL^{-1}(\n F^+_u)\rangle-\langle\Psi,\cL^{-1}({\mathbbm 1}_{\bb R^{n+1}_+}f)\rangle-\big[-\langle\Psi,u|_{\bb R^{n+1}_+}\rangle+\langle\Psi,\cL^{-1}(\n F^+_u)|_{\bb R^{n+1}_+}\rangle\big] \\ = \langle\Psi,u-\cL^{-1}({\mathbbm 1}_{\bb R^{n+1}_+}f)\rangle.
\end{multline*}

Proof of \ref{item.greenformula2}). Let $\Psi\in(\Ya)^*$ have compact support within $\bb R^{n+1}_-$. Using (\ref{eq.dl2}) and (\ref{eq.slgreen}), we have that
\begin{multline*}
\langle\Psi,\m D^{\cL,-}(\Tr_0 u)+\m S^\cL(\partial_{\nu}^{\cL,+}u)\rangle=\langle\Psi,-\cL^{-1}(\n F^+_u)|_{\bb R^{n+1}_-}+\cL^{-1}(\n F^+_u)-\cL^{-1}({\mathbbm 1}_{\bb R^{n+1}_+}f)\rangle\\ =-\langle\Psi,\cL^{-1}({\mathbbm 1}_{\bb R^{n+1}_+}f)\rangle.
\end{multline*}

The proof of  \ref{item.greenformula3}) is the same as \ref{item.greenformula2}), and is thus omitted.\hfill{$\square$}

Let us now consider some adjoint relations for the double layer potential. First, for any $u\in\Ya$, denote by ${\n F^*_u}^+\in(\Ya)^*$ the functional given by $\langle{\n F^*_u}^+,v\rangle:=B_{\cL^*,\bb R^{n+1}_+}[u,v]$ for $v\in\Ya$.

\begin{proposition}\label{prop.adj} We have the following identities.
	\begin{enumerate}[i)]
		\item\label{item.dlconormaladjoint} For each $\varphi,\psi\in\Hf$, the identity $\langle\partial_{\nu}^{\cL,+}\m D^{\cL,+}\varphi,\psi\rangle=\langle\varphi,\partial_{\nu}^{\cL^*,+}\m D^{\cL^*,+}\psi\rangle$ holds.
		\item\label{item.dladjoint} For each $\gamma\in\Hfm$, $\varphi\in\Hf$, $t\geq0$, the adjoint relation
		\begin{equation}\label{eq.dladjoint}
		\langle\gamma,\Tr_t\m D^{\cL,+}\varphi\rangle=-\langle\partial_{\nu}^{\cL^*,-}\n T^{-t}\m S^{\cL^*}\gamma,\varphi\rangle=-\langle\partial_{\nu,-t}^{\cL^*,-}\m S^{\cL^*}\gamma,\varphi\rangle=\langle\partial_{\nu,-t}^{\cL^*,+}\m S^{\cL^*}\gamma,\varphi\rangle
		\end{equation}
		holds. In the case that $t=0$, we may write
		\begin{equation}\label{eq.dladjointt=0}
		\langle\gamma,\Tr_0\m D^{\cL,+}\varphi\rangle=-\langle\gamma,\varphi\rangle+\langle\partial_{\nu}^{\cL^*,+}\m S^{\cL^*}\gamma,\varphi\rangle.
		\end{equation}
		\item\label{item.dlmovetder}  Fix $\varphi\in H^{\frac12}(\bb R^n)$. For each $t>0$, and every $\zeta\in\Hfm$, we have the identity \\$	\langle\Tr_tD_{n+1}\m D^{\cL,+}\varphi,\zeta\rangle=\frac{d\,}{dt} \langle\Tr_t\m D^{\cL,+}\varphi,\zeta\rangle=\langle\varphi,\partial_{\nu,-t}^{\cL^*,-}D_{n+1}\m S^{\cL^*}\zeta\rangle_{L^2,L^2}$.
		\item\label{item.dlgradadj} Fix $t>0$. Let ${\bf g}=(\vec{g_{\|}},g_{\perp}):\bb R^n\ra\bb C^{n+1}$ be such that  $g_{\parallel}, g_{\perp}\in C_c^{\infty}(\bb R^n)$. In the sense of distributions, we have the adjoint relation
		\begin{equation}\label{eq.dlgradadj}
		\langle\nabla\Tr_t D_{n+1}\m D^{\cL,+}\varphi,{\bf g}\rangle_{\n D',\n D}=\langle\varphi,D_{n+1}\partial_{\nu,-t}^{\cL^*,-}(\m S^{\cL^*}\nabla){\bf g}\rangle_{L^2,L^2}
		\end{equation}
	\end{enumerate}
\end{proposition}

\noindent\emph{Proof.} Proof of \ref{item.dlconormaladjoint}). Let $\Phi,\Psi$ be extensions of $\varphi,\psi$ respectively to $\Ya$. Then,
\begin{multline*}
\langle\partial_{\nu}^{\cL,+}\m D^{\cL,+}\varphi,\psi\rangle=B_{\cL,\bb R^{n+1}_+}[\m D^{\cL,+}\varphi,\Psi]=-B_{\cL,\bb R^{n+1}_+}[\Phi,\Psi]+B_{\cL,\bb R^{n+1}_+}[\cL^{-1}(\n F^+_{\Phi}),\Psi]\\ =-\overline{B_{\cL^*,\bb R^{n+1}_+}[\Psi,\Phi]}+\overline{B_{\cL^*,\bb R^{n+1}_+}[\Psi,\cL^{-1}(\n F^+_{\Phi})]}\\ =\overline{B_{\cL^*,\bb R^{n+1}_+}[\m D^{\cL^*,+}\psi,\Phi]}-\overline{B_{\cL^*,\bb R^{n+1}_+}[(\cL^*)^{-1}({\n F^*_{\Psi}}^+),\Phi]}+\overline{B_{\cL^*,\bb R^{n+1}_+}[\Psi,\cL^{-1}(\n F^+_{\Phi})]}\\ =\langle\varphi,\partial_{\nu}^{\cL^*,+}\m D^{\cL^*,+}\psi\rangle+\overline{\Big[-B_{\cL^*,\bb R^{n+1}_+}[(\cL^*)^{-1}({\n F^*_{\Psi}}^+),\Phi]+B_{\cL^*,\bb R^{n+1}_+}[\Psi,\cL^{-1}(\n F^+_{\Phi})]\Big]},
\end{multline*}
where in the first equality we used the definition of the conormal derivative, in the second equality we used the definition of the double layer potential. Hence it suffices to show that $B_{\cL^*,\bb R^{n+1}_+}[\Psi,\cL^{-1}(\n F^+_{\Phi})]=B_{\cL^*,\bb R^{n+1}_+}[(\cL^*)^{-1}({\n F^*_{\Psi}}^+),\Phi]$. Simply note that
\begin{multline*}
B_{\cL^*,\bb R^{n+1}_+}[\Psi,\cL^{-1}(\n F^+_{\Phi})] =\langle{\n F^*_{\Psi}}^+, \cL^{-1}(\n F^+_{\Phi})\rangle=\langle(\cL^{-1})^*({\n F^*_{\Psi}}^+),\n F^+_{\Phi}\rangle\\  =\overline{B_{\cL,\bb R^{n+1}_+}[\Phi,(\cL^{-1})^*({\n F^*_{\Psi}}^+)]}=B_{\cL^*,\bb R^{n+1}_+}[(\cL^*)^{-1}({\n F^*_{\Psi}}^+),\Phi],
\end{multline*}
where in the first equality we used the definition of the functional ${\n F^*_{\Psi}}^+$, and in the third equality we used the definition of $\n F^+_{\Phi}$. The desired identity follows.

Proof of \ref{item.dladjoint}). Let $\gamma\in\Hfm$, $\varphi\in\Hf$, and let $\Phi\in\Ya$ be an extension of $\varphi$ such that $\Tr_0\Phi=\varphi$. By the definition of $\m D^{\cL,+}\varphi$, we have that $\langle\gamma,\Tr_t\m D^{\cL,+}\varphi\rangle=-\langle\gamma,\Tr_t\Phi\rangle+\langle\gamma,\Tr_t \cL^{-1}(\n F^+_{\Phi})\rangle$. By (\ref{eq.slshift}), we have that
\begin{multline*}
\langle\gamma,\Tr_t \cL^{-1}(\n F^+_{\Phi})\rangle=\langle(\Tr_t\circ \cL^{-1})^*\gamma,\n F^+_{\Phi}\rangle=\langle\n T^{-t}\m S^{\cL^*}\gamma,\n F^+_{\Phi}\rangle\\ =\overline{\langle\n F^+_{\Phi},\n T^{-t}\m S^{\cL^*}\gamma\rangle}=\overline{B_{\cL,\bb R^{n+1}_+}[\Phi,\n T^{-t}\m S^{\cL^*}\gamma]}=B_{\cL^*,\bb R^{n+1}_+}[\n T^{-t}\m S^{\cL^*}\gamma,\Phi]\\ =B_{\cL^*}[\n T^{-t}\m S^{\cL^*}\gamma,\Phi]-B_{\cL^*,\bb R^{n+1}_-}[\n T^{-t}\m S^{\cL^*}\gamma,\Phi]=\langle\gamma,\Tr_t\Phi\rangle-\langle\partial_{\nu}^{\cL^*,-}\n T^{-t}\m S^{\cL^*}\gamma,\varphi\rangle,
\end{multline*}
where in the last equality we used (\ref{eq.slshift}) combined with (\ref{eq.laxmilgramforsl}) for the first term, and for the second term we used the definition of the conormal derivative and the fact that $\cL\n T^{-t}\m S^{\cL^*}=0$ in $\bb R^{n+1}_-$. From this calculation, the first equality in (\ref{eq.dladjoint}) follows. The second and third equalities are straightforward consequences of Lemma \ref{lm.conormals}. To see that (\ref{eq.dladjointt=0}) is true, simply observe that  when $t=0$, we have that $\cL^*\m S^{\cL^*}\gamma=0$ in  $\bb R^{n+1}_+$ and in $\bb R^{n+1}_-$, whence   we deduce that $\langle\partial_{\nu}^{\cL^*,-}\n T^{-0}\m S^{\cL^*}\gamma+\partial_{\nu}^{\cL^*,+}\n T^{-0}\m S^{\cL^*}\gamma,\varphi\rangle=B_{\cL^*}[\m S^{\cL^*}\gamma,\Phi]=\langle\gamma,\Tr_0\Phi\rangle$. Adding and subtracting $\langle\partial_{\nu}^{\cL^*,+}\n T^{-t}\m S^{\cL^*}\gamma,\varphi\rangle$ to the right-hand side of (\ref{eq.dladjoint}) now proves the claim.

Proof of \ref{item.dlmovetder}). Let $t>0$. By Proposition \ref{prop.dlproperties} \ref{item.dlpropertiesws}), we have that $\cL\m D^{\cL,+}\varphi=0$ in $\bb R^{n+1}_+$. Therefore, using Proposition \ref{prop.tder} \ref{item.tdershift}) we see that $\Tr_{\tau}D_{n+1}\m D^{\cL,+}\varphi\in H^{\frac12}(\bb R^n)$ for each $\tau>0$. Similarly, we have that $\Tr_{-\tau}\nabla D_{n+1}\m S^{\cL^*}\zeta\in L^2(\bb R^n)$ for each $\tau>0$. Using (\ref{eq.tdermovetder}) and \ref{item.dladjoint}), we calculate that $
\frac{d\,}{d\tau}(\Tr_t\m D^{\cL,+}\varphi,\zeta)\big|_{\tau=t}=-\frac{d\,}{d\tau}(\varphi,\partial_{\nu,-t}^{\cL^*,-}\m S^{\cL^*}\zeta)\big|_{\tau=t}$. Now we use the characterization of the conormal derivative, (\ref{eq.conormalchar}), to obtain
\begin{multline*}
-\frac{d\,}{d\tau}(\varphi,\partial_{\nu,-t}^{\cL^*,-}\m S^{\cL^*}\zeta)\big|_{\tau=t} =-\frac{d\,}{d\tau}\big(\varphi,\big[e_{n+1}\cdot\Tr_{-\tau}(A^*\nabla+\overline{B_2})\m S^{\cL^*}\zeta\big]\big)_{2,2}\big|_{\tau=t}\\ =\big(\varphi,\big[e_{n+1}\cdot\Tr_{-t}(A^*\nabla+\overline{B_2})D_{n+1}\m S^{\cL^*}\zeta\big]\big)_{2,2}.
\end{multline*}

Finally, \ref{item.dlgradadj}) follows from \ref{item.dlmovetder}) similarly as in Proposition \ref{prop.slproperties} \ref{item.slpropertiesgradadj}).\hfill{$\square$}

Let us now establish  standard jump relations.

\begin{proposition}[Jump relations]\label{prop.jump} Let $\varphi\in\Hf$ and $\gamma\in\Hfm$. 
\begin{enumerate}[i)]
\item\label{item.dljump} The identity $\Tr_0\m D^{\cL,+}\varphi+\Tr_0\m D^{\cL,+}\varphi=-\varphi$ holds in $\Hf$.
\item\label{item.sljump} The identity $\partial_{\nu}^{\cL,+}\m S^\cL\gamma+\partial_{\nu}^{\cL,-}\m S^\cL\gamma=\gamma$ holds in $\Hfm$.
\item\label{item.dlcont} The identity $\partial_{\nu}^{\cL,+}\m D^{\cL,+}\varphi=\partial_{\nu}^{\cL,-}\m D^{\cL,-}\varphi$ holds in $\Hfm$.
\item\label{item.slcont} The identity $\Tr_0(\m S^\cL\gamma|_{\bb R^{n+1}_+})=\Tr_0(\m S^\cL\gamma|_{\bb R^{n+1}_-})$ holds in $\Hf$.
\end{enumerate}
\end{proposition}

\noindent\emph{Proof.} The statement \ref{item.slcont}) is immediate from the fact that $\m S^\cL\gamma\in\Ya$. The statement \ref{item.sljump}) follows from the definition of the conormal derivative and the fact that $\cL\m S^\cL\gamma=0$ in $\bb R^{n+1}\backslash\{t=0\}$.

Proof of \ref{item.dljump}). Let $\gamma\in\Hfm$, and let $\Phi\in\Ya$ be any extension of $\varphi$. Using (\ref{eq.dladjointt=0}), we see that
\begin{multline*}
\langle\gamma,\Tr_0[\m D^{\cL,+}\varphi+\m D^{\cL,-}\varphi]\rangle=-2\langle\gamma,\varphi\rangle+\langle\partial_{\nu}^{\cL^*,-}\m S^{\cL^*}\gamma+\partial_{\nu}^{\cL^*,+}\m S^{\cL^*}\gamma,\varphi\rangle\\ =-2\langle\gamma,\varphi\rangle+B_{\cL^*}[\m S^{\cL^*}\gamma,\Phi]=-2\langle\gamma,\varphi\rangle+\langle\gamma,\Tr_0\Phi\rangle=-\langle\gamma,\varphi\rangle.
\end{multline*}

Proof of \ref{item.dlcont}). Let $\psi\in\Hf$, and let $\Phi,\Psi\in\Ya$ be extensions of $\varphi,\psi$ respectively such that $\Tr_0\Phi=\varphi, \Tr_0\Psi=\psi$. Also recall that $\cL\m D^{\cL,+}\varphi=0$ in $\bb R^{n+1}_+$, and $\cL\m D^{\cL,-}\varphi=0$ in $\bb R^{n+1}_-$. Then,
\begin{multline*}
\langle\partial_{\nu}^{\cL,+}\m D^{\cL,+}\varphi,\psi\rangle=B_{\cL,\bb R^{n+1}_+}[\m D^{\cL,+}\varphi,\Psi]=-B_{\cL,\bb R^{n+1}_+}[\Phi,\Psi]+B_{\cL,\bb R^{n+1}_+}[\cL^{-1}(\n F^+_{\Phi}),\Psi]\\ =-B_{\cL,\bb R^{n+1}_+}[\Phi,\Psi]+B_{\cL}[\cL^{-1}(\n F^+_{\Phi}),\Psi]-B_{\cL,\bb R^{n+1}_-}[\cL^{-1}(\n F^+_{\Phi}),\Psi]\\ =-B_{\cL,\bb R^{n+1}_+}[\Phi,\Psi]+\langle\n F^+_{\Phi},\Psi\rangle-B_{\cL,\bb R^{n+1}_-}[\cL^{-1}(\cL\Phi),\Psi]+B_{\cL,\bb R^{n+1}_-}[\cL^{-1}(\n F^-_{\Phi}),\Psi]\nonumber\\ =-B_{\cL,\bb R^{n+1}_-}[\Phi,\Psi]+B_{\cL,\bb R^{n+1}_-}[\cL^{-1}(\n F^-_{\Phi}),\Psi]=B_{\cL,\bb R^{n+1}_-}[\m D^{\cL,-}\varphi,\Psi] =\langle\partial_{\nu}^{\cL,-}\m D^{\cL,-}\varphi,\psi\rangle.
\end{multline*}\hfill{$\square$}

\subsection{Initial $L^2$ estimates for the single layer potential}
We now establish several estimates for the single layer potential. This will allow us to prove the square function estimates, via a $Tb$ theorem, in the next section. We begin with a perturbation result.
\begin{proposition}[Initial slice estimates]\label{L2avgandsliceboundsperturb.prop}  The following statements hold provided that $\max\{\Vert B_1\Vert_n, \Vert B_2\Vert_n\}$ is small enough, depending only on $n$, $\lambda$, and $\Lambda$.
	\begin{enumerate}[i)]
		\item\label{item.avgslice}  For each $f\in C_c^{\infty}(\bb R^n)$, each $a>0$, and each $m\geq1$, we have the estimate
		\begin{equation}\label{eq.avgslice}
		\dashint_a^{2a}\int_{\bb R^n}|t^m\nabla\partial_t^{m} \s^{\cL}f|^2\,dt\lesssim_m\Vert f\Vert_2^2.
		\end{equation}
		\item For each $f\in C_c^{\infty}(\bb R^n)$, each $t\geq0$, and each $m\geq2$, we have the estimate
		\begin{equation}\label{eq.sliceest}
		\Vert t^m\nabla\partial_t^m \s_t^{\cL}f\Vert_{L^2(\bb R^n)}\lesssim_m\Vert f\Vert_2.
		\end{equation}
	\end{enumerate}
\end{proposition}
\begin{proof}
First we see that the second estimate is a consequence of the first by the Caccioppoli inequality on slices \eqref{Lpcacconslices}. In particular, we have that
\begin{equation*}
\|t^m \nabla \partial_t^m \s_t^{\cL} f\|_2^2   = \sum_{Q \in \dd_t} \int_{Q} |t^m\nabla \partial_t^m \s_t^{\cL}f|^2 \, dx
 \lesssim \dashint_t^{2t} \int_{\rn}  |s^m\nabla \partial_s^{m-1} \s_s^{\cL} f|^2\, dx \, ds,
\end{equation*}
where $\dd_t$ is a grid of $n$-dimensional cubes of side length $t$. Thus it suffices to show \ref{item.avgslice}).

To this end, we know from \cite{AAAHK} that \ref{item.avgslice}) holds with $\s^{\cL}$ replaced by $\s^{\cL_0}$, where $\cL_0 = -\Delta$. Thus, to prove \ref{item.avgslice}), we show that $\dashint_a^{2a}\int_{\bb R^n}|t^m\nabla\partial_t^{m} (\s^{\cL} - \s^{\cL_0})f|^2\, dt\lesssim\Vert f\Vert_2$. Observe that
\begin{multline*}
\s^{\cL} - \s^{\cL_0}  = (\tr_0 \circ ((\cL^*)^{-1} - (\cL_0^*)^{-1}))^*
    = (\tr_0 \circ ((\cL_0^*)^{-1}(\cL_0^* - \cL^*)(\cL^*)^{-1}))^*
\\   = ((\cL_0^* - \cL^*)(\cL^*)^{-1})^*\s^{\cL_0}
   = \cL^{-1}(\cL_0 - \cL) \s^{\cL_0}
\\   = -\div(I - A)\nabla\s^{\cL_0}  -\cL^{-1}\div(B_1 \s^{\cL_0}) - \cL^{-1}B_2 \cdot \nabla \s^{\cL_0}.
\end{multline*}
Now let $f \in C_c^\infty(\rn)$. Then we have that
\begin{multline*}
 \dashint_a^{2a}\int_{\bb R^n}|t^m\nabla D_{n+1}^{m} (\s^{\cL} - \s^{\cL_0})f|^2\, dt ~ \lesssim
\dashint_a^{2a}\int_{\bb R^n}|t^m\nabla D_{n+1}^{m} \cL^{-1}\div(I - A)\nabla \s^{\cL_0}f)|^2\, dt 
\\    + \dashint_a^{2a}\int_{\bb R^n}|t^m\nabla D_{n+1}^{m} \cL^{-1}\div(B_1 \s^{\cL_0}f)|^2\, dt\\   + \dashint_a^{2a}\int_{\bb R^n}|t^m\nabla D_{n+1}^{m}  \cL^{-1}B_2 \cdot \nabla \s^{\cL_0}f|^2\, dt
   \qquad =: I+II+III.
\end{multline*}
We prove only the bound $II \lesssim \|f\|_2^2$ as the bounds for $I$ and $III$ are entirely analogous, and we will indicate the small differences after we bound $II$. Let $\psi = \psi(t)$ be such that $\psi \in C_c^\infty(-a/5, a/5)$, $\psi \equiv 1$ on $(-a/10, a/10)$, $0 \le \psi \le 1$, $\tfrac{d^k}{dt^k} \psi \lesssim_k (1/a)^k$. Writing
$1 = \psi + (1 - \psi)$, we have that
\begin{multline*}
II  \le \dashint_a^{2a}\int_{\bb R^n}|t^m\nabla D_{n+1}^{m} \cL^{-1}\div(B_1 \s^{\cL_0}f)|^2\, dt
\\  \le  \dashint_a^{2a}\int_{\bb R^n}|t^m\nabla D_{n+1}^{m} \cL^{-1}\div(\psi B_1 \s^{\cL_0}f)|^2\, dt
\\   \qquad + \dashint_a^{2a}\int_{\bb R^n}|t^m\nabla D_{n+1}^{m} \cL^{-1}\div((1-\psi)B_1 \s^{\cL_0}f)|^2\, dt =: II_1 + II_2.
\end{multline*}
To bound $II_1$, we notice that if $g = \div(\psi B_1 \s^{\cL_0}f)$, then $g \equiv 0$ on $\rn \times (a/5, \infty)$. It follows that each $D_{n+1}^k\cL^{-1}g =\cL^{-1} D_{n+1}^k g$, $k = 0, 1, \dots m$ is a (null) solution in $\rn \times (a/5, \infty)$. Let $\dd_a$ be a grid of $n$-dimensional cubes with side length $a$. Applying the Caccioppoli inequality $m$ times and using that $t \approx a$ on $(a,2a)$, we see that
\begin{multline*}
II_1  \lesssim a^{2m - 1} \int_a^{2a} \int_{\rn} |\nabla D_{n+1}^m\cL^{-1}\div(\psi B_1 \s^{\cL_0} f)|^2  
\\  \lesssim a^{2m -1} \sum_{Q \in \dd_a} \int_a^{2a} \int_Q  |\nabla D_{n+1}^m\cL^{-1}\div(\psi B_1 \s^{\cL_0} f)|^2  
\\   \lesssim a^{-1} \sum_{Q \in \dd_a} \int_{a/2}^{4a} \int_{2Q}  |D_{n+1}\cL^{-1}\div(\psi B_1 \s^{\cL_0} f)|^2 
   \lesssim a^{-1}  \int_{\re} \int_{\rn}  |\nabla\cL^{-1}\div(\psi B_1 \s^{\cL_0} f)|^2  
\\   \lesssim  a^{-1}  \int_{\re} \int_{\rn}| \psi B_1 \s^{\cL_0} f|^2     \lesssim \dashint_{-a/5}^{a/5} \int_{\rn}|B_1 \s^{\cL_0} f|^2   \lesssim \| f \|_{2}^2,
\end{multline*}
where we used that $\sup_{t \neq 0}\| B_1\s_t^{\cL_0}\|_{L^2(\rn) \to L^2(\rn)} \lesssim\sup_{t\neq0}\| \nabla \s_t^{\cL_0}\|_{L^2(\rn) \to L^2(\rn)} \le C$ (see \cite[Lemma 4.18]{AAAHK}) and that $\nabla \cL^{-1} \div : L^2(\ree) \to L^2(\ree)$.

Now we deal with $II_2$. Set $g = (1 -\psi)B_1 \s^{\cL_0}f$. Then, we have that
\begin{align*}
D_{n+1}^m g &= (1 -\psi)B_1D_{n+1}^m \s^{\cL_0}f  + \sum_{k = 1}^m \psi^{(k)}B_1 D_{n+1}^{m-k} \s^{\cL_0}f
 =: F_0 + \sum_{k = 1}^m F_k.
\end{align*}
where $\psi^{(k)}=\tfrac{d^k}{dt^k} \psi$. The triangle inequality yields that
\begin{align*}
II_2 \le \sum_{k = 0}^m  \dashint_a^{2a}\int_{\bb R^n}|t^m\nabla D_{n+1}^{m} \cL^{-1}\div(F_k)|^2 \, dt =: \sum_{k = 0}^m II_{2,k}.
\end{align*}
For $II_{2,k}$, $k =1,2 \dots m$, we use that $t \approx a$ in the region of integration, the properties of $\psi$,  and that 
$\nabla \cL^{-1} \div: L^2 \to L^2$ to obtain that
\begin{multline*}
II_{2,k}  \lesssim a^{2m - 1}\int_{\re}\int_{\rn} |\psi^{(k)} B_1 \partial_t^{m-k} \s^{\cL_0} f|^2  \, dt
\\   \lesssim a^{2m - 2k - 1}\int\limits_{a/10 \le |t| \le a/5} \int_{\rn} | B_1 \partial_t^{m-k} \s^{\cL_0} f|^2   \, dt
\\    \lesssim \dashint_{-a/5}^{a/5} \int_{\rn} |t^{m-k} B_1 \partial_t^{m-k} \s^{\cL_0} f|^2   \, dt
    \lesssim \| f \|_2^2,
\end{multline*}
where we used \cite[Lemma 2.10]{AAAHK} in the last line. Finally, to handle $II_{2,0}$, we use that $(1- \psi) = 0$ if $|t| < a/10$, and that $\nabla \cL \div: L^2 \to L^2$ to obtain that
\begin{multline*}
II_{2,0}  \lesssim a^{2m - 1} \int_{\re} \int_{\rn} |(1- \psi) B_1 \partial_t^m \s^{\cL_0} f|^2   \, dt
   \lesssim a^{2m - 1} \int_{|t| > a/10} \int_{\rn} |B_1 \partial_t^m \s^{\cL_0} f|^2   \, dt
\\  \lesssim \int_{|t| > a/10} \int_{\rn} | t^{m+1} B_1 \partial_t^m \s^{\cL_0} f|^2\, \frac{dt}{t}
   \lesssim \| f \|_{2}^2,
\end{multline*}
where we used  the estimate $\|| t^{m+1} B_1 \partial_t^m \s^{\cL_0} f|\|_2^2 \lesssim  \|| t^{m+1} \partial_t^m \nabla\s^{\cL_0} f|\|_2^2 \lesssim \| f \|_2^2$ in the last line. To see this last estimate, we simply use the ``travelling up'' procedure for square functions (see Proposition \ref{lm.travelup} below) and that $\cL_0 = \Delta$ has good square function estimates.  
We now observe   that handling the term $III$  amounts to replacing the use of the mapping property $\nabla \cL^{-1} \div: L^2 \to L^2$ by the fact that $\nabla \cL^{-1} B_2: L^2 \to L^2$. The term $I$ is handled exactly the same way, using the $L^\infty$ bound for $(I - A)$, without appealing to the mapping properties of multiplication by $B_1$.\end{proof}

\begin{remark}\label{singlelayerbddonL2.rmk}
Note that, from now on, it makes sense to write the objects appearing in \eqref{eq.avgslice} and \eqref{eq.sliceest} for $f$ in $L^2(\bb R^n)$ after we have made extensions by continuity.
\end{remark}

Before proceeding, we will need some identities improving on the duality results in Section \ref{AbsLPthry.sec} for the single and double layers. To ease the notation, we will use $(G)_t$ to denote the trace at $t$ of a function $G$ defined in $\reu$.

\begin{proposition}[Further distributional identities of the layer potentials]\label{dualwtder.prop}
	For any $t\neq 0$ and $m\geq1$, the following statements are true.
	\begin{enumerate}[i)]
		\item For any $f\in C_c^\infty(\rn)$ and any $\vec g\in L^2(\rn;\cee)$, we have that 
		\begin{equation*}
		\dfrac{d^m}{dt^m} \langle \nabla \s^{\cL}_t f,\vec g\rangle = \langle (D_{n+1}^m\nabla \s^{\cL}[f])_t,\vec g \rangle.
		\end{equation*}
		\item For any $f\in L^2(\rn)$ and any $\vec g\in C_c^\infty(\rn;\cee)$, we have that
		\begin{equation*}
		\dfrac{d^m}{dt^m} \langle f, ((\s^{\cL^*}\nabla)[\vec g])_{-t}\rangle = (-1)^m\langle f, (D_{n+1}^m (\s^{\cL^*}\nabla)[\vec g])_{-t}\rangle.
		\end{equation*}
		\item If $m \ge 2$, then for every $f\in L^2(\rn)$ and  $\vec g\in L^2(\rn,\cee))$, we have the identity
		\begin{equation*}
		\langle (D_{n+1}^m\nabla \s^{\cL}[f])_t,\vec g \rangle= (-1)^m\langle f, (D_{n+1}^m (\s^{\cL^*}\nabla)[\vec g])_{-t}\rangle.
		\end{equation*}
	\end{enumerate}
\end{proposition}

\begin{proof} Let us first show the identities with $f\in C_c^\infty(\rn)$ and $\vec g\in C_c^\infty(\rn;\cee)$. For the first equality, note that $u:=\s^{\cL}[f]\in \yot$ and $\cL u=0$ in $\ree\setminus \{ x_{n+1}=0\}$. In particular, $\partial_t u \in W^{1,2}(\Sigma_a^b)$ for any $a<b$ such that $0\notin [a,b]$, by Lemma \ref{Le1.5.lem}. By iteration we have that $\partial_t^m \nabla u \in L^2(\Sigma_a^b)$. In particular, arguing as in Lemma \ref{ContSlices.lem}, we realize that the map $t\mapsto  \nabla u(\cdot, t)$ is smooth (with values in $L^2(\rn;\cee)$). The first equality for $m=1$ then boils down to proving the weak convergence of the difference quotients to the derivative in $L^2(\rn)$; that is, showing that
	\begin{equation*}
	\lim_{h\to 0} \dfrac{\nabla u(t+h)-\nabla u(t)}{h} = \partial_t\nabla u (t), \qquad \textup{ weakly in } L^2(\rn).
	\end{equation*}
	But this follows from the smoothness of our map. The case of general $m$ now follows by induction.

For the second equality, by definition we have that 
\begin{equation*}
((\s^{\cL^*}\nabla)[\vec g])_s = -(\s^{\cL^*}[\divp g_{\|}])_s -(D_{n+1}\s^{\cL^*}[g_{\perp}])_s,
\end{equation*} 
and since $\vec g\in C_c^\infty(\rn; \cee)$, we can apply the same argument as above to conclude that 
\begin{equation*}
\dfrac{d^m}{dt^m} \langle f,((\s^{\cL^*}\nabla)[g])_{-t}\rangle = (-1)^m\langle f, (D_{n+1}^m (\s^{\cL^*}\nabla)[g_{\perp}])_{-t}\rangle.
\end{equation*}

The third equality now follows by duality: For $f\in C_c^\infty(\rn)$ and $g\in C_c^\infty(\rn;\cee)$, we have that
\begin{multline*}
\langle (D_{n+1}^m\nabla \s^{\cL}[f])_t,\vec g \rangle  =\dfrac{d^m}{dt^m} \langle \nabla \s^{\cL}_t f,\vec g\rangle =\dfrac{d^m}{dt^m} \langle f,((\s^{\cL}\nabla)[\vec g])_{-t}\rangle\\
= (-1)^m\langle f, (D_{n+1}^m (\s^{\cL}\nabla)[\vec g])_{-t}\rangle
\end{multline*}

Finally, the identities are extended to the respective $L^2$ spaces via a straightforward density argument using Proposition \ref{L2avgandsliceboundsperturb.prop}.\end{proof}

We now present an off-diagonal decay result.

\begin{proposition}[Good off-diagonal decay]\label{nabSL2LpODdecay}
Let $Q \subset \rn$ be a cube and  $g \in L^2(Q)$ with $\supp g \subseteq Q$. If $p\in[2,p_+]$ is such that $|p -2|$ is small enough   that Lemma \ref{Snei.lem} holds, we have  that
\begin{equation}\nonumber
\Big(\int_{R_0} |t^m (\partial_t)^m \nabla \s_t^{\cL} g(x)|^p \, dx \Big)^\frac{1}{p} \lesssim 2^{-(m + 1)}t^m \ell(Q)^{-n(1/2 - 1/p)}\ell(Q)^{-m} \lVert g \rVert_{L^2(Q)},
\end{equation}
provided $t\approx \ell(Q)$. Moreover, for any $k\geq 1 $ and any $t\in \RR$, the estimate
$$\Big(\int_{R_k} |t^m (\partial_t)^m \nabla \s_t^{\cL} g(x)|^p \, dx \Big)^\frac{1}{p} \lesssim 
2^{nk \alpha}2^{-k(m + 1)}t^m \ell(Q)^{-n(1/2 - 1/p)}\ell(Q)^{-m} \lVert g \rVert_{L^2(Q)},$$
where $\alpha = \alpha(p) = \frac{1}{p}(1 - \tfrac{p}{p_+})$ and the annular regions  $R_k=R_k(Q)$ are defined by $R_0:= 2Q$, $R_k: = 2^{k+1}Q\setminus 2^kQ$, for all $k\geq 1$. In particular, if $t \approx \ell(Q)$ we have that
$$\Big(\int_{R_k} |t^m (\partial_t)^m \nabla \s_t^{\cL}g(x)|^p \, dx \Big)^\frac{1}{p} \lesssim 
2^{nk \alpha}2^{-k(m + 1)} \ell(Q)^{-n(1/2 - 1/p)} \lVert g \rVert_{L^2(Q)}.$$
\end{proposition}

By a straightforward duality argument, from the above proposition we deduce

\begin{corollary}\label{L2LpODdecay} Define $\Theta_{t,m} := t^m \partial_t^m (\s_t \nabla)$. Let $g \in L^2(Q)$ and suppose that $p\in[2,p_+]$ is such that  $|p -2|$ is small enough so that Lemma \ref{Snei.lem} holds. Then for $q = \tfrac{p}{p-1}$ and $k \ge 1$, we have that
\begin{equation}\nonumber
\| \Theta_{t,m} (f\bbm 1_{R_k})\|_{L^2(Q)} \lesssim 2^{nk \alpha}2^{-k(m + 1)}t^m \ell(Q)^{-n(1/q - 1/2)}\ell(Q)^{-m} \| f\|_{L^q(R_k)},
\end{equation}
where $\alpha = \alpha(p)$ is as in Proposition \ref{nabSL2LpODdecay}.
Moreover, if $t \approx \ell(Q)$, then for all $k \ge 0$,
\begin{multline*}
\| \Theta_{t,m} (f\bbm 1_{R_k})\|_{L^2(Q)} \lesssim 
2^{nk \alpha}2^{-k(m + 1)} \ell(Q)^{-n(1/q - 1/2)} \lVert f \rVert_{L^q(R_k)}
\\ \approx 2^{nk \alpha}2^{-k(m + 1)} t^{-n(1/q - 1/2)} \lVert f \rVert_{L^q(R_k)}.
\end{multline*}
\end{corollary}

\begin{proof}[Proof of Proposition \ref{nabSL2LpODdecay}] Notice that $g\in L^2(Q)\subset L^{2n/(n+1)}(Q)\subset\Hfm$, so that $\s_t^{\cL} g$ is well defined.

We treat first the case $k\geq 1$.  Fix a small parameter $\delta = \delta(m) >0$ and set $\tilde{R}_k= (2+\delta)^{k+1}Q\setminus(2-\delta)^kQ$ be a small (but fixed) dilation of $R_k$. We may use that $\partial_t^m u$ is a solution (see Proposition \ref{tsolns.prop}), a slight variant of Lemma \ref{Lpcacconslices.lem} adapted to annular regions, and Proposition \ref{Lpcaccop.prop} to see that  
\begin{equation*}
\Big(\int_{R_k} |t^m (\partial_t)^m \nabla \s_t^{\cL} g|^p\Big)^\frac{1}{p} 
 \lesssim \dfrac{t^m}{(2^k\ell(Q))^{1+1/p}}\Big( \dint_{I_{k,1}} |\partial_t^m \s_t ^{\cL} g|^p\Big)^{1/p}, 
\end{equation*}
where $I_k:= \{ (y,s)\in \ree: y\in \tilde{R}_{k,1}, \, s\in (t-2^k\ell(Q), t+2^k\ell(Q)\}$ and $\tilde{R}_{k,j}$ is defined as $\tilde{R}_k$ but with $\delta/(m+2-j)$ instead of $\delta$ (so that, in particular, $\tilde{R}_{k,m+1}=\tilde{R}_k$). Now, applying the $(n+1)$-dimensional $L^p$ Caccioppoli $m$ times (see Proposition \ref{Lpcaccop.prop}), we further obtain   that
\begin{equation*}
\Big(\int_{R_k} |t^m\partial_t^m \nabla \s_t^{\cL} g|^p \Big)^\frac{1}{p} 
 \lesssim \dfrac{t^m}{(2^k\ell(Q))^{m+1+1/p}}\Big( \dint_{I_{k,m+1}} | \s_t ^{\cL} g|^p\Big)^{1/p}.
\end{equation*}
Now, using H\"older's inequality in $t$ and the mapping properties of $\s_t^{\cL}$ we see that
\begin{multline*}
\Big(\int_{R_k} |t^m\partial_t^m \nabla \s_t^{\cL} g |^p   \Big)^\frac{1}{p} 
 \lesssim \frac{t^m}{[2^k\ell(Q)]^{m+1}} \sup_{t \in (-2^{k}\ell(Q), 2^{k}\ell(Q))} \lVert \s_t^{\cL} g \rVert_{L^{p}(\tilde{R_k})}
\\  \lesssim \frac{t^m[2^k\ell(Q)]^{n\alpha}}{[2^k\ell(Q)]^{m+1}} \sup_{t \in (-2^{k}\ell(Q), 2^{k}\ell(Q))} \lVert \s_t^{\cL} g \rVert_{L^{p_+}(\tilde{R_k})}
   \lesssim \frac{t^m[2^k\ell(Q)]^{n\alpha}}{[2^k\ell(Q)]^{m+1}} \lVert g \rVert_{L^{p_-}(Q)} 
\\  \lesssim \frac{t^m[2^k\ell(Q)]^{n\alpha}}{[2^k\ell(Q)]^{m+1}} |Q|^{\frac{1}{2n}} \lVert g \rVert_{L^{2}(Q)}. 
\end{multline*}
The case $k=0$ is treated similarly, except that we impose the restriction $t\approx \ell(Q)$ to guarantee that we are far away from the support of $g$.\end{proof}

For the most part, the case $q = p = 2$ in the above proposition will be enough for our purposes; however, the introduction of error terms in the $Tb$ theorem below will necessitate a certain quasi-orthogonality result for which we use the case $p > 2 > q$.

\begin{lemma}[Quasi-orthogonality]\label{QOforSnabBI.lem}
Let $m \ge n$ and let $\m Q_s$ be a CLP family (see Definition \ref{CLP.def}). Then there exist $\gamma, C > 0$ such that for all $s < t$, we have that
\begin{equation}\label{QOforSnabBI.eq}
\| \Theta_{t,m} B_1 I_1\m Q_s^2 g\|_2 \le C\Big(\frac{s}{t}\Big)^\gamma \|\m Q_s g\|_2
\end{equation}
for all $g \in L^2(\rn)$, where $I_1$ is the standard fractional integral operator of order $1$. Here, $C$ and $\gamma$ depend on $m$, $n$, $\lambda$, $\Lambda$,  and the constants in the definition of $\m Q_s$.
\end{lemma}

\begin{proof}
Let us first note that if $\alpha(p)$ is given as in Proposition \ref{nabSL2LpODdecay}, then $\alpha(p)\leq1/(2n)$. Therefore, for all $k \ge 0$ and $Q$ with $\ell(Q) \approx t$, we have that
\begin{multline}\label{QOforSnabBIeq1.eq}
\| \Theta_{t,m} (f\bbm 1_{R_k})\|_{L^2(Q)}  \lesssim 2^{nk \alpha}2^{-k(m + 1)} t^{-n(1/q - 1/2)} \lVert f \rVert_{L^q(R_k)}
\\   \lesssim 2^{-k \beta}  t^{-n(1/q - 1/2)} \lVert f \rVert_{L^q(R_k)},
\end{multline}
for some $\beta \ge n/2 + 1$, where we use that $m \ge n$.  

We first establish a variant of \eqref{QOforSnabBI.eq} with a collection of CLP families. Let $\zeta \in C_c^\infty(B(0,\tfrac{1}{100}))$ be real, radial and have zero average. Define $\m Q_s^{(1)}f(x) := (\zeta_s \ast f)(x)$, where $\zeta_s(x) = s^{-n}\zeta(\tfrac{x}{s})$. Set $\m Q_s^{(2)} f := s^2 \Delta e^{s^2\Delta}f$. By re-normalizing $\zeta$ (multiplying by a constant) we may assume that
\begin{equation}\label{CRforQ1Q2.eq}
\int_0^\infty\m Q_s^{(1)}\m Q_s^{(2)} \frac{ds}{s}=I
\end{equation}
in the strong operator topology of $L^2$. Indeed,
\begin{multline*}
\n F\Big(\int_0^\infty\m Q_s^{(1)}\m Q_s^{(2)}f \, \frac{ds}{s} \Big)  = -\int_0^\infty \hat\zeta(s|\xi|) s^2 |\xi|^2 e^{-s^2|\xi|^2} \hat f(\xi) \, \frac{ds}{s}
\\   = -\hat f(\xi) \int_0^\infty \hat \zeta(s) s^2 e^{-s^2} \, \frac{ds}{s},
\end{multline*}
where $\hat \zeta$ is the Fourier transform of $\zeta$ and we abused notation by regarding $\zeta$ and hence $\hat \zeta$ as a function of the radial variable. Then, to achieve the desired reproducing formula, \eqref{CRforQ1Q2.eq}, we may renormalize $\zeta$ so that $\int_0^\infty \hat \zeta(s) s^2 e^{-s^2} \, \frac{ds}{s}=-1$. Let $q < 2$ be such that the conclusion of Corollary \ref{L2LpODdecay} holds. We will show that for all $s < t$,
\begin{equation}\label{QOforSnabBIstar.eq}
\|\Theta_{t,m} B_1 I_1\m Q_s^{(1)}\m Q_s^{(2)} g \|_2 \lesssim \Big(\frac{s}{t}\Big)^{n(1/q - 1/2)} \|\m Q_s^{(3)}\vec{R} g \|_2,
\end{equation}
where $\vec{R} = I_1 \nabla_\parallel$ is the vector valued Riesz transform on $\rn$ and $\m Q_s^{(3)}\vec{f} := se^{s^2 \Delta}\div_\parallel\vec{f}$.  Before proving \eqref{QOforSnabBIstar.eq}, we establish a ``local hypercontractivity'' estimate. For $Q \subset \rn$ a cube and $s < \ell(Q)$, we have that
\begin{equation}\label{ODhyper2.eq}
\|\m Q_s^{(1)}h \|_{L^{\frac{nq}{n-q}}(R_k(Q))} \lesssim s^{-n\big(\frac12-\frac{n-q}{nq}\big)}\|h \|_{L^2(B_k)}
\end{equation}
for all $k \ge 0$, where $R_0(Q) = 2Q$, $R_k(Q) = 2^{k+1}Q \setminus 2^kQ$ for $k \ge 1$,  and $B^k(Q) = B(x_Q, 2^{k + 2}\ell(Q)\sqrt{n})$. To verify \eqref{ODhyper2.eq}, we use that $s < \ell(Q)$, Young's convolution inequality, and the properties of $\zeta_s$. 

Now we are ready to prove \eqref{QOforSnabBIstar.eq}. Let $\dd_t$ be a grid of cubes on $\rn$ with side length $t$ and set $F = I_1 g$. Consider the estimate 
\begin{multline*}
\|\Theta_{t,m} B_1 I_1\m Q_s^{(1)}\m Q_s^{(2)} g \|_2  
   = \|\Theta_{t,m} B_1\m Q_s^{(1)}\m Q_s^{(2)} F \|_2
  = \Big( \sum_{Q \in \dd_t} \int_Q |\Theta_{t,m} B_1\m Q_s^{(1)}\m Q_s^{(2)} F|^2 \Big)^{1/2}
\\  \le  \sum_{k \ge 0} \Big( \sum_{Q \in \dd_t} \int_Q\big|\Theta_{t,m}\big([B_1\m Q_s^{(1)}\m Q_s^{(2)} F]\mathbbm{1}_{R_k(Q)}\big)(x)\big|^2 \, dx \Big)^{1/2}
\\   \lesssim \sum_{k \ge 0} 2^{-\beta k}t^{-n(1/q - 1/2)}\Big( \sum_{Q \in \dd_t} \Big(\int_{R_k(Q)} |B_1\m Q_s^{(1)}\m Q_s^{(2)} F |^q  \Big)^{2/q} \Big)^{1/2}
\\    \lesssim  \|B_1\|_{n}\sum_{k \ge 0} 2^{-\beta k}t^{-n(1/q - 1/2)}\Big( \sum_{Q \in \dd_t} \Big(\int_{R_k(Q)} |\m Q_s^{(1)}\m Q_s^{(2)} F |^{\frac{nq}{n-q}}  \Big)^{\frac{2(n-q)}{nq}} \Big)^{1/2}
\\  \lesssim \sum_{k \ge 0} 2^{-\beta k}t^{-n(1/q - 1/2)}s^{-n\big(\frac12-\frac{n-q}{nq}\big)}\Big( \sum_{Q \in \dd_t} \int_{B_k(Q)} |\m Q_s^{(2)} F |^{2}  \Big)^{1/2}
\\   \lesssim \sum_{k \ge 0} 2^{-\beta k}t^{-n(1/q - 1/2)}s^{-n\big(\frac12-\frac{n-q}{nq}\big)} s\Big( \sum_{Q \in \dd_t} \int_{B_k(Q)} |\m Q_s^{(3)} \nabla_\parallel F |^{2}  \Big)^{1/2}
\\  \lesssim \Big(\frac{s}{t}\Big)^{n(1/q - 1/2)} \sum_{k \ge 0} 2^{-\beta k + \frac{nk}{2}} \Big( \sum_{Q \in \dd_t}\int_Q \dashint_{B_k(Q)} |\m Q_s^{(3)} \nabla_\parallel F(x)|^{2} \, dx\, dy\Big)^{1/2}
\\ \lesssim \Big(\frac{s}{t}\Big)^{n(1/q - 1/2)} \sum_{k \ge 0} 2^{-\beta k + \frac{nk}{2}} \Big( \int_{\rn} \dashint_{|x -y| < 2^kt} |\m Q_s^{(3)} \nabla_\parallel F(x)|^{2} \, dx\, dy\Big)^{1/2}
\\  \lesssim \Big(\frac{s}{t}\Big)^{n(1/q - 1/2)}\|\m Q_s^{(3)}\nabla_\parallel F \|_2
 = \Big(\frac{s}{t}\Big)^{n(1/q - 1/2)}\|\m Q_s^{(3)}\vec{R} g\|_2,
\end{multline*}
where first we used that $I_1g=F$, then Minkowski's inequality in the second line,   \eqref{QOforSnabBIeq1.eq} in the third line,    H\"older's inequality in the fourth line,   \eqref{ODhyper2.eq} in the fifth line, and the mapping properties of the Hardy-Littlewood maximal function in the last line. The above estimate proves \eqref{QOforSnabBIstar.eq}.
 
Now we are ready to pass to an arbitrary CLP family $\m Q_s$. We may obtain, using the Cauchy-Schwarz inequality and \eqref{CRforQ1Q2.eq}, that
\begin{multline*}
\|\Theta_{t,m}B_1I_1\m Q_s^2 g\|_2   = \int_{\rn}\Big| \int_0^\infty \Theta_{t,m} B_1 I_1\m Q_\tau^{(1)}\m Q_\tau^{(2)} \m Q_s^2 g(x) \, \frac{d\tau}{\tau}\Big|\, dx 
\\   \le C_\gamma \int_{\rn} \int_0^\infty \max\Big(\frac{s}{\tau}, \frac{\tau}{s}\Big)^\gamma
|\Theta_{t,m} B_1 I_1\m Q_\tau^{(1)}\m Q_\tau^{(2)} \m Q_s^2 g(x)|^2 \, \frac{d\tau}{\tau} \, dx
=: I + II + III,
\end{multline*}
where $I,II,III$ are, respectively, the integrals over the intervals $\tau < s< t$, $s\leq \tau \leq t$, and $s<t<\tau$.  On the other hand, note that the kernel of $\m Q_s^{(3)} \vec{R}$ is, up to a constant multiple, the inverse Fourier transform of $s|\xi| e^{-s^2|\xi|^2}$.  Therefore, if we set $\m Q_s^{(4)} =\m Q_s^{(3)}\vec{R}$, then we have that
\begin{equation}\label{Q4QsQO.eq}
\max\big\{\|\m Q_\tau^{(4)}\m Q_s f\|_2,  \|\m Q_\tau^{(2)}\m Q_s f\|_2\big\}\lesssim \min\Big(\frac{\tau}{s}, \frac{s}{\tau}\Big)^{2\gamma}\Vert f\Vert_2,
\end{equation}
for some $\gamma > 0$ (and hence all smaller $\gamma$). For convenience, set $\sigma = n(1/q - 1/2)$ and we assume that $\gamma$ above is such that $\gamma < 2\sigma$. By \eqref{QOforSnabBIstar.eq} and \eqref{Q4QsQO.eq}, we have that
\[I \lesssim \int_0^s \Big(\frac{s}{\tau}\Big)^\gamma 
\Big(\frac{\tau}t \Big)^{2\sigma} \|\m Q_\tau^{(4)}\m Q^2_s h \|_2^2  \frac{d\tau}{\tau} \lesssim
\Big(\frac st \Big)^{2\sigma} \|Q_s h\|^2_2\,,
\]
and observe that $\tau <s$ in the present scenario. Similarly, we have that 
\[ II \lesssim \int_s^t  \Big(\frac{\tau}{s}\Big)^\gamma 
\Big(\frac{\tau}t \Big)^{2\sigma} \Big(\frac{s}{\tau}\Big)^{2\gamma} \|\m Q_s h \|_2^2  \frac{d\tau}{\tau}
\lesssim
\Big(\frac st \Big)^{\gamma} \|\m Q_s h\|^2_2\,, \]
since in particular, $\gamma < 2\sigma$.
Finally, by \eqref{eq.sliceest} and the mapping $B_1 I_1: L^2(\rn) \to L^2(\rn)$, we have that $\Theta_{t,m} B_1 I_1 \m Q_\tau^{(1)} :L^2(\rn) \to L^2(\rn)$ uniformly in $t$ and $\tau$, and thus it follows that
\begin{multline*}
III    \lesssim \int_t^\infty  \Big(\frac{\tau}{s}\Big)^\gamma 
\| Q_\tau^{(2)} Q_s h (x)\|_2^2 \frac{d\tau}{\tau}    \lesssim \int_t^\infty  \Big(\frac{\tau}{s}\Big)^\gamma 
\Big(\frac{s}{\tau}\Big)^{2\gamma} \| Q_s h\|_2^2 \frac{d\tau}{\tau} 
  \lesssim \Big(\frac st \Big)^{\gamma} \|Q_s h\|^2_2, 
\end{multline*}
where we used \eqref{Q4QsQO.eq}.\end{proof}

We conclude this section with the following proposition, which summarizes the off-diagonal decay given by Proposition \ref{nabSL2LpODdecay} and Corollary \ref{L2LpODdecay}.

\begin{proposition}\label{ODforsevop.prop}
For $m \in \N$, $m \ge \tfrac{n+1}{2} + 2$ , the operators $t^m\partial_t^m (\s_t\nabla)$, $t^m \partial_t^{m+1}\s_t$
and $\Theta_t'$ defined by 
$$\Theta_t'\vec{g}(x) := (t^m \partial_t^m \s_t\nabla \widetilde A \vec{g} + t^m \partial_t^m \s_t[ {B_2}_\parallel g])(x)$$
have good off diagonal-decay in the sense of Definition \ref{goodODdecay.def} with the implicit constants depending on $n$, $m$, $\lambda$, and $\Lambda$, provided that $\max\{\|B_1\|_n, \|B_2\|_n\} < \eps_0$, where $\eps_0$ depends on $n$, $\lambda$, and $\Lambda$.  
\end{proposition}
\begin{proof}
By Corollary \ref{L2LpODdecay} with $p =2$,   for any cube $Q \subset \rn$ and $k \ge 2$ we have that
$$\| \Theta_{t,m} (f\bbm 1_{R_k})\|_{L^2(Q)}^2 \lesssim 2^{-k}\Big( \frac{t}{2^k \ell(Q)}\Big)^{2m}\|f\|_{L^2(R_k)}^2,$$
where $R_k = R_k(Q) = 2^{k+1}Q\setminus 2^kQ$. Thus, for all $t \in (0, C\ell(Q))$, it follows that
$$\| \Theta_{t,m} (f\bbm 1_{R_k})\|_{L^2(Q)}^2 \lesssim 2^{-kn}\Big( \frac{t}{2^k \ell(Q)}\Big)^{2m-(n-1)}\|f\|_{L^2(R_k)}^2,$$
so that if $m \ge \tfrac{n+1}{2}$, we obtain the estimate
\begin{equation}\label{ODL2estThetatmeq.eq}
\| \Theta_{t,m} (f\bbm 1_{R_k})\|_{L^2(Q)}^2 \lesssim 2^{-kn}\Big( \frac{t}{2^k \ell(Q)}\Big)^{2}\|f\|_{L^2(R_k)}^2.
\end{equation}
This bound provides the desired good off-diagonal decay for $t^m\partial_t^m (\s_t\nabla)$, $t^m \partial_t^{m+1}\s_t$ and $t^m \partial_t^m \s_t\nabla \widetilde A \vec{g}$ in the sense of Definition \ref{goodODdecay.def}. To obtain the good off-diagonal decay for the remainder of $\Theta_t'$, $t^m \partial_t^m (\s_t {B_2}_\parallel)$, we return to the proof of Proposition \ref{nabSL2LpODdecay} and make a slight modification. Let $\eta$ be a smooth cut-off adapted to $R_k$; that is, $\eta \equiv 1$ on $R_k$, $\eta \in C^\infty_c(\widetilde R_k)$ and $|\nabla \eta| \lesssim \tfrac{1}{\ell(Q)}$, where $\widetilde R_k$ is as in Proposition \ref{nabSL2LpODdecay}. Then for $g \in L^2(Q)$ with $\supp g \subseteq Q$,   from H\"older's inequality and the Sobolev embedding on $\rn$ we have that
\begin{multline*}
\|t^m \partial_t^m  {B_2}_\parallel \s_t g\|_{L^2(R_k)}   \lesssim \|\eta t^m \partial_t^m \s_t g\|_{L^{\frac{2n}{n-2}}(\rn)}^2
\\    \lesssim \|(\nabla\eta) t^m \partial_t^m \s_t g\|_{L^2(\widetilde R_k)}^2 + 
\| t^m \partial_t^m \nabla \s_t g\|_{L^2(\widetilde R_k)}^2
    \lesssim \|t^{m -1} \partial_t^m \s_t g\|_{L^2(\widetilde R_k)}^2 + 
\| t^m \partial_t^m \nabla \s_t g\|_{L^2(\widetilde R_k)}^2
\\   \lesssim \|(\Theta_{t,m-1})^* g\|_{L^2(\widetilde R_k)}^2 + 
\| (\Theta_{t,m})^*  g\|_{L^2(\widetilde R_k)}^2.
\end{multline*}
Dualizing these estimates, the off-diagonal decay for $t^m \partial_t^m (\s_t {B_2}_\parallel)$ follows from the off-diagonal decay in \eqref{ODL2estThetatmeq.eq}, provided that $m \ge \tfrac{n+1}{2} + 1$.\end{proof}

Before continuing on to the next section we make two remarks.

\begin{remarks}

i) In the next section, we will use the off diagonal decay of the operators in Proposition \ref{ODforsevop.prop} or \emph{similar} ones.  The proof of good off-diagonal decay for these operators is entirely analogous to those above.

ii) As seen above, there may be some loss of $t$-derivatives (and hence decay) in our operators when we obtain certain estimates. Therefore, when proving the  first square function estimate (Theorem \ref{Tbargument.thrm}), we ensure that $m \ge n + 10 > \tfrac{n+1}{2} + 10$ so that Lemma \ref{QOforSnabBI.lem} and Proposition \ref{ODforsevop.prop} hold.
\end{remarks}

\section{Square function bounds via $Tb$ Theory}

The goal of this section is to prove Theorem \ref{perturbfromlargem.thrm}.

\subsection{Reduction to high order $t$-derivatives}\label{ss5.1}

We will adapt the methods of \cite{GH,Hof-May-Mour} to prove the   square function bound in Theorem \ref{perturbfromlargem.thrm} for $m$ large:
\begin{theorem}[Square function bound for high $t-$derivatives]\label{Tbargument.thrm}
For each $m \in \N$ with $m \ge n + 10$, we have the estimate
$$\dint_{\ree_+} \big|t^m (\partial_t)^{m+1}\s_t^{\cL} f(x)\big|^2\,  \frac{dx \, dt}{t} \leq C \| f \|_{L^2(\rn)}^2,$$
where $C$ depends on $m$, $n$, $\lambda$, and $\Lambda$, provided that $\max\{\| B_1\|_n, \|B_2\|_n\}$ is sufficiently small depending on $n$, $\lambda$, and $\Lambda$. Under the same hypotheses, the analogous bounds hold for $\cL$ replaced by $\cL^*$, and for $\bb R^{n+1}_+$ replaced by $\bb R^{n+1}_-$.
\end{theorem}

Let us see that we may reduce the proof of Theorem \ref{perturbfromlargem.thrm} to that of Theorem \ref{Tbargument.thrm}. First, it is a fact that square function estimates for solutions $u$ of $\cL u=0$ ``travel up'' the $t-$derivatives:  
\begin{lemma}[Square function bound ``travels up'' $t$-derivatives]\label{lm.travelup} Fix $m,k\in\bb N$ with $m>k\geq1$. Suppose that $u\in W^{1,2}_{\loc}(\bb R^{n+1}_+)$ solves $\cL u=0$ in $\bb R^{n+1}_+$ in the weak sense. Then there exists a constant $C$ depending only on $m$, $n$, $\lambda$, $\Lambda$, and $\max\{\| B_1\|_n, \|B_2\|_n\}$, such that $\||t^m\partial_t^{m-1}\nabla u\||\leq C\||t^k\partial_t^{k} u\||$.
\end{lemma}	

The proof of the previous lemma is very straightforward (decompose into Whitney cubes and then use the Caccioppoli inequality), and thus omitted.

Now, the following proposition (and Lemma \ref{lm.travelup}\footnote{Lemma \ref{lm.travelup} is used to show that $\ep_0$ can be chosen independently of $m$.}) immediately allow us to reduce proof of Theorem  \ref{perturbfromlargem.thrm} to that of Theorem \ref{Tbargument.thrm}, and is a partial converse to Lemma \ref{lm.travelup}. Recall that $L^2(\bb R^n)\subset\Hfm$.

\begin{proposition}[Square function bound ``travels down'' $t$-derivatives]\label{lm.passtograd} 
The following 
estimates hold, where the implicit constants depend on $m$, $k$, $\lambda$, and $\Lambda$.
\begin{enumerate}[i)]
   \item For each $f\in L^2(\bb R^n)$ and each $m\geq1$, 
$\||t^m\partial_t^m\nabla\m Sf\||\lesssim_m\||t^{m+1}\partial_t^{m+1} \nabla \m Sf\||+\Vert f\Vert_2$. 
		\item For each $f\in L^2(\bb R^n)$ and each $m>k\geq1$,
		\begin{equation}\label{eq.traveldown}
		\||t^k\partial_t^k\nabla\m Sf\||\lesssim_{m}\||t^m\partial_t^{m+1}\m Sf\||+\Vert f\Vert_2.
		\end{equation}
	\end{enumerate}
\end{proposition}

\begin{proof}   One may obtain $ii)$ as a consequence of $i)$ via induction on $m$, using 
Caccioppoli's inequality on Whitney boxes after increasing the number of $t$ derivatives appropriately. 
So it suffices to prove $i)$.
Fix $m\in\bb N$, $N>0$ large, $\epsilon>0$ small and let $f\in L^2(\bb R^n)$. Let $\psi\in C_c^{\infty}(0,\infty)$ be a non-negative function which satisfies
\begin{align*}
\psi&\equiv1\quad\text{on }\big(\epsilon,\tfrac1{\epsilon}\big),\qquad &&\psi\equiv0\quad\text{on }\big(0,\tfrac{\epsilon}2\big)\cup\big(\tfrac2{\epsilon},\infty\big),\\|\psi'|&\leq\tfrac4{\epsilon}\quad\text{on }\big(\tfrac{\epsilon}2,\epsilon\big),\qquad &&|\psi'|\leq2\epsilon\quad\text{on }\big(\tfrac1{\epsilon},\tfrac2{\epsilon}\big).
\end{align*}
Since $\m Sf\in Y^{1,2}(\bb R^{n+1})$ and $\cL\m Sf=0$ in $\bb R^{n+1}_+$ in the weak sense, then $\partial_t^m\m Sf\in W^{1,2}_{\loc}(\bb R^{n+1})$ and $\cL\partial_t^m\m Sf=0$ in $\bb R^{n+1}_+$ in the weak sense. Observe that
\begin{equation}\nonumber
\int_{B(0,N)}\int_{\epsilon}^{1/\epsilon}t^{2m-1}|\partial_t^m\nabla\m Sf|^2\,dt\leq\int_{B(0,N)}\int_{\epsilon/2}^{2/\epsilon}t^{2m-1}|\partial_t^m\nabla\m Sf|^2\psi\,dt,
\end{equation}
and notice per our observations in Proposition \ref{L2avgandsliceboundsperturb.prop} that the right-hand side above is finite. Now,
\begin{multline}\nonumber
\int_{B(0,N)}\int_{\epsilon/2}^{2/\epsilon}t^{2m-1}|\partial_t^m\nabla\m Sf|^2\psi\,dt=\int_{B(0,N)}\int_{\epsilon/2}^{2/\epsilon}t^{2m-1}\partial_t^m\nabla\m Sf\overline{\partial_t^m\nabla\m Sf}\psi\,dt\\=-\frac1{2m}\int_{B(0,N)}\int_{\epsilon/2}^{2/\epsilon}t^{2m}\partial_t(\partial_t^m\nabla\m Sf\overline{\partial_t^m\nabla\m Sf}\psi)\,dt\\\leq\frac1m\int_{B(0,N)}\int_{\epsilon/2}^{2/\epsilon}t^{2m}|\partial_t^{m+1}\nabla\m Sf||\partial_t^m\nabla\m Sf|\psi\,dt\\+\frac1m\int_{B(0,N)}\Big[\dashint_{\epsilon/2}^{\epsilon}t^{2m}|\partial_t^m\nabla\m Sf|^2\,dt+\dashint_{1/\epsilon}^{2/\epsilon}t^{2m}|\partial_t^m\nabla\m Sf|^2\,dt\Big].
\end{multline}
The last two terms are controlled by (\ref{eq.avgslice}). As for the first term, note that $2m=\frac{2m-1}2+\frac{2m+1}2$, and we use Cauchy's inequality and absorb one of the resulting summands to the left-hand side. 
Sending $N\ra\infty$ and $\epsilon\searrow0$ yields the desired result.\end{proof}

Combining Lemma \ref{tTboundlem.lem} below and Theorem \ref{Tbargument.thrm}, we will also obtain the following result.

\begin{theorem}[Square function bound for $\m S\nabla$]\label{fullsqfnlargem.thrm}
	For each $m \in \N$, with $m \ge n + 10$,
	\begin{equation}\label{eq.sqfnest}
	\dint_{\ree_+} \Big|t^m (\partial_t)^m(\s_t \nabla) \vec{f}(x)\Big|^2\,  \frac{dx \, dt}{t} \lesssim \| \vec{f} \|_{L^2(\rn)},
	\end{equation}
	where $C$ depends on $m$, $n$, $\lambda$, $\Lambda$, provided that $\max\{\| B_1\|_n, \|B_2\|_n\}$ is 
	 sufficiently small depending on $m$,  $n$, $\lambda$, $\Lambda$.
	These results hold for $\cL^*$ and in the lower half space as the hypotheses are symmetric.
\end{theorem}

\subsection{Setup for the $Tb$ argument and testing functions}

Having reduced matters to proving Theorem \ref{Tbargument.thrm}, we fix $m \in \N$ with $m \ge n + 10$. We define the space $H$ to be the subspace of $L^2(\bb R^n)^n$ consisting of the gradients of $Y^{1,2}(\bb R^n)$-functions. That is, $H=\{h': h' = \nabla F,\,F \in Y^{1,2}(\rn)\}$. For $h'\in H$ and $h^0 \in L^2(\rn)$, we set $h =(h',h^0)$ and define for each $t\in\bb R\backslash\{0\}$,
\begin{gather*}
\Theta_t^0 h^0:= t^{m}\partial_t^{m + 1} \s_th^0,\qquad\text{and}
\\ \Theta'_t h':= t^{m}\partial_t^{m}(\s_t \nabla)\tilde{A}h' + t^{m}(\partial_t)^{m}\s_t ( {B_2}_{\parallel}\cdot h'),
\end{gather*}
where we recall that $\tilde A$ is the $(n+1)\times n$ submatrix of $A$ consisting of the first $n$ columns of $A$.
We let $\Theta_t := (\Theta'_t, \Theta^0_t):H\times L^2(\bb R^n)\ra L^2(\bb R^n)$, which acts on  $h = (h',h^0)$ via the identity $\Theta_th=  \Theta'_th' + \Theta^0_th^0$. 

For each $t>0$, we also define an auxiliary operator $\Theta^{(a)}_t:L^2(\bb R^n,\bb C^{n+1})\ra L^2(\bb R^n)$ which acts on  $g = (g', g^0)$ via $\Theta_t^{(a)} g = t^{m}(\partial_t)^{m}(\s_t\nabla) (g', g^0)$. This auxiliary operator will play the role of an error term that allows us to integrate by parts. Accordingly, define $\widehat\Theta_t$ acting on functions $h = (h', h^0, h'') \in H \times L^2(\rn, \mathbb{C}) \times L^2(\rn, \mathbb{C}^{n+1})$ via
$$\widehat\Theta_t h(x)  = \Theta_t (h', h_0)(x) + \Theta^{(a)}_th''(x).$$

We need to define appropriate testing functions for our family $\{\Theta_t\}$. Let $\tau \in (0, 1/40)$ be a small parameter to be chosen later, and let $\widetilde\Psi$ be a smooth cut-off function in $\bb R^{n+1}$ with the following properties:
\begin{gather*}
\widetilde\Psi \in C_c^\infty\big(\big[-\tfrac{1}{1000}, \tfrac{1}{1000}\big]^n \times\big[-\tfrac{\tau}2,\tfrac{\tau}2\big]\big),\qquad \widetilde\Psi \equiv 1 \text{ on }\big[-\tfrac{1}{2000}, \tfrac{1}{2000}\big]^n \times\big[-\tfrac{\tau}4,\tfrac{\tau}4\big]
\\ 0 \le \widetilde\Psi \le 1,\qquad |\nabla \widetilde\Psi | \lesssim 1/\tau.
\end{gather*}
Let $\Psi := c_{n,\tau} \widetilde\Psi$ where $c_{n,\tau}$ is chosen so that $\lVert \Psi \rVert_1 = 1$. Hence $\Psi$ is a normalization of $\widetilde\Psi$. For any cube $Q \subset \rn$, we define the measurable functions 
$$\Psi_Q(X):= \frac{1}{\ell(Q)^{n+1}} \Psi\Big(\frac{1}{\ell(Q)}[X - (x_Q,0)]\Big), \quad \big(\text{ note that } \lVert \Psi_Q \rVert_1 = 1\big),$$
$$\Psi_Q^\pm(y,s): = \Psi_Q\Big(y~,~ s \mp \frac{3}{2}\tau\ell(Q)\Big),$$
and $\Psi_Q^{s'}(y,s):= \Psi_Q(y, s + s')$, for each $s'\in\bb R$. Let us make a few observations about $\widetilde\Psi$ and $\Psi$. The fact that
$$\mathbbm{1}_{\big[-\tfrac{1}{2000}, \tfrac{1}{2000}\big]^n \times\big[-\tfrac{\tau}4,\tfrac{\tau}4\big]} \le \widetilde\Psi \le \mathbbm{1}_{\big[-\tfrac{1}{1000}, \tfrac{1}{1000}\big]^n \times \big[-\tfrac{\tau}2,\tfrac{\tau}2\big]}$$
forces that $c_{n,\tau} \approx  \frac{1}{\tau}$ and that $\lVert \widetilde\Psi \rVert_{2_*} \approx \tau^{\frac{1}{2_*}}$. Consequently, $\lVert \Psi \rVert_{2_*}\approx \tau^{-1 + 1/2_*}$, and 
\begin{equation*} 
\lVert \Psi_Q \rVert_{2_*} \approx \tau^{-1 + 1/2_*}[\ell(Q)^{n+1}]^{-1 + 1/2_*} = [\tau \ell(Q)^{n+1}]^{-1/2 + 1/(n+1)}.
\end{equation*}
Of course, the same $L^{2_*}$ estimate holds for $\Psi_Q^\pm$ and $\Psi_Q^{s'}$.
Now, we define for any cube and $s' \in \RR$ the quantities
\begin{equation*}
F_Q^\pm:= \cL^{-1}(\Psi^\pm_Q),\qquad F_Q:= F_Q^+ - F_Q^-,\qquad F_Q^{s'}:= \cL^{-1}(\Psi_Q^{s'}).
\end{equation*}
By our previous observations and the fact that $L^{2_*}(\bb R^{n+1})$ embeds continuously into $(Y^{1,2}(\bb R^{n+1}))^*$, we easily see that for any cube $Q$ and any $s'\in\bb R$, the estimate
\begin{equation}\label{gradestFQ}
\max\big\{\lVert \nabla F_Q \rVert_2, \lVert \nabla F^\pm_Q \rVert_2, \lVert \nabla F_Q^{s'} \rVert_2\big\}\lesssim
 [\tau \ell(Q)^{n+1}]^{-1/2 + 1/(n+1)}
\end{equation}
holds. Notice that we have
\begin{equation*}
\Psi_Q^+(y,s) - \Psi_Q^-(y,s)= - \int_{-\tfrac{3}{2} \tau \ell(Q)}^{\tfrac{3}{2} \tau \ell(Q)} \partial_{s'}\Psi(y, s + s') \, ds' = - \int_{-\tfrac{3}{2} \tau \ell(Q)}^{\tfrac{3}{2} \tau \ell(Q)} \partial_{s}\Psi_Q^{s'}(y,s) \, ds'.
\end{equation*}
Therefore, the identity 
\begin{multline}\label{slideFq.eq}
F_Q = - \int_{-\tfrac{3}{2} \tau \ell(Q)}^{\tfrac{3}{2} \tau \ell(Q)} \cL^{-1}(D_{n+1}\Psi_Q^{s'}) \, ds' = - \int_{-\tfrac{3}{2} \tau \ell(Q)}^{\tfrac{3}{2} \tau \ell(Q)}\partial_t \cL^{-1}(\Psi_Q^{s'}) \, ds'
\\ =- \int_{-\tfrac{3}{2} \tau \ell(Q)}^{\tfrac{3}{2} \tau \ell(Q)} \partial_t F_Q^{s'} \, ds'
\end{multline}
is valid in $Y^{1,2}(\bb R^{n+1})$. For convenience, we write   $(\nabla_{y,s} u)(y,0) := (\nabla_{y,s} u(y,s))\big|_{s = 0}$. We are now ready to define our testing functions $b_Q =(b'_Q, b_Q^0)$. Let $b^0_Q$ be defined via $b^0_Q(y) := |Q|(\partial_{\nu}^{\cL,-}F_Q)(y,0)$, where  
\begin{multline*}
\partial_{\nu}^{\cL,-}u(y,0)= e_{n+1} \cdot[A(\nabla_{y,s} u)(y,0) - B_1 u(y,0)]
\\   = e_{n+1} \cdot[A(\nabla_{y,s} u)(y,0)] - (B_1)_\perp u(y,0).
\end{multline*}
We define $b'_Q$ via $b'_Q: = |Q| \nabla_\parallel F_Q(y,0)$, while we define the auxiliary testing function $b_Q^{(a)}$ via $b_Q^{(a)}:= |Q|B_1 F_Q(y,0)$. 

We will define a measure for each cube $Q$ that corresponds to a smoothened characteristic function. 
We do this exactly as in \cite{GH}. Let $\hm>0$  to be chosen. For each cube, we let $d\mu_Q = \phi_Q \, dx$, where $\phi_Q: \rn \to [0,1]$ is a smooth bump function supported in $(1 + \hm)Q$ with $\phi_Q \equiv 1$ on $(1/2)Q$. Clearly, we can choose $\phi_Q$ so that $\phi_Q \gtrsim \hm$ on $Q$ and $\lVert\nabla\phi_Q \rVert_{L^\infty} \lesssim 1/\ell(Q)$. We also let $\Phi_Q: \ree \to [0,1]$ be a smooth extension of $\phi_Q$; that is, $\Phi_Q(y,0) = \phi_Q(y)$, with $\Phi_Q$ supported in $I_{(1 + \hm)Q}$ and $\Phi_Q \equiv 1$ on $I_{(1/2)Q}$, where  for any cube $Q \subset \rn$, we let $I_Q = Q \times (-\ell(Q), \ell(Q))$ denote the ``double Carleson box'' associated to $Q$. We may also ensure that $\lVert \nabla \Phi_Q \rVert_{L^\infty(\ree)} \lesssim 1/\ell(Q)$. 

\subsection{Properties of the testing functions}

The testing functions defined above enjoy the following essential properties which justify their use in the $Tb$ argument.

\begin{proposition}[Properties of the testing functions]\label{allbQproperties.lem} Let $b_Q = (b'_Q, b^0_Q)$, $\hat b_Q$, and $\widehat\Theta_t$ be as above. For any $\eta > 0$, there exists $\tau \in (0,1)$ depending on $n$, $\lambda$, $\Lambda$, $\eta$, and $C_0 = C_0(m, \tau)$, and there exists a measure $\mu_Q$  as described above, such that for each cube $Q$, the estimates
\begin{equation}\label{destbQ1.eq}
\int_{\rn} |b_Q|^2  \le C_0 |Q|
\end{equation}
\begin{equation}\label{destbQ2.eq}
\int_0^{\ell(Q)} \int_{Q} \big|\widehat\Theta_t \hat b_Q(x)\big|^2 \, \frac{dx \, dt}{t} \le C_0|Q|
\end{equation}
\begin{equation}\label{destbQ3.eq}
\frac{1}{2} \le \f R e \Big(\frac1{\mu_Q(Q)}\int_{Q} b^0_Q \, d\mu\Big)
\end{equation}
\begin{equation}\label{destbQ4.eq}
\Big|\frac1{\mu_Q(Q)}\int_{Q} b'_Q \, d\mu_Q \Big| \le \frac{\eta}{2},
\end{equation}
hold, provided that $\max\{\| B_1\|_n, \|B_2\|_n\}=\eps_m < \tau$.
\end{proposition}

We note that while the smallness of $\eps_m =\max\{\| B_1\|_n, \|B_2\|_n\}$ apparently 
depends on $m$ at this point, we 
may prove Theorem \ref{Tbargument.thrm} for a fixed sufficiently large $m$,
and then use Lemma
\ref{lm.travelup}  and Proposition \ref{lm.passtograd}  to remove any dependence on $m$ in the bound for 
$\max\{\| B_1\|_n, \|B_2\|_n\}$.  For now, throughout the $Tb$ argument, we shall
continue to use $\eps_m$ to denote this quantity.

We will establish several preliminary lemmas in anticipation of the proof of the above proposition.

\begin{lemma}[Estimate of the $L^2$ norm of $b_Q$]\label{bqest1.lem} The estimate
$$\int_{\rn} |b_Q|^2\lesssim \tau^{-2 + 2/(n+1)}|Q|,$$
holds, where the implicit constant depends on $n$, $\lambda$, and $\Lambda$.
\end{lemma}
\begin{proof}
Set $a: = \tfrac{\tau\ell(Q)}{1000}$ and observe that $F_Q$ solves $\cL F_Q=0$ in the strip $\{(x,t): |t| < 50a\}$. Let $\bb G_a$ be the grid of pairwise disjoint $n$-dimensional cubes with sides of length $a$ parallel to the coordinate axes, and for each $P \in \bb G_a$, define the $(n+1)-$dimensional box $P^*:= 2P \times [-2\ell(P), 2\ell(P)]$. Applying Lemma \ref{Lpcacconslices.lem} and the estimate \eqref{gradestFQ}, we obtain that
\begin{multline*}
\int_{\rn} |\nabla F_Q(\cdot,0) |^2 = \sum_{P \in \bb G_a} \int_{P} |\nabla F_Q(\cdot,0)|^2  
 \lesssim \frac{1}{a}  \sum_{P \in \bb G_a} \dint_{P^*} |\nabla F_Q|^2
\\  \lesssim \frac{1}{a} \lVert \nabla F_Q \rVert_2^2  \lesssim \frac{1}{a} [\tau \ell(Q)^{n+1}]^{-1 + 2/(n+1)}
\lesssim \tau^{-2 + 2/(n+1)}|Q|^{-1},
\end{multline*}
where we used that $a \approx \tau\ell(Q)$ and the bounded overlap of $\{P^*\}_{P \in\bb G_a}$. Upon multiplying the above inequality by $|Q|^2$,   we have the desired estimate up to controlling $\lVert\, |Q| (B_1)_\perp F_Q(\cdot,0)\rVert_{L^2(\rn)}^2$. We have already shown that $\lVert \nabla_\parallel F_Q(\cdot,0)\rVert_{L^2(\bb R^n)} < \infty$, and from Lemma \ref{ContSlices.lem} and Lemma \ref{Le1.5.lem}, we have that $F_Q(\cdot, 0) \in L^{2^*}(\rn)$, so that $F(\cdot,0) \in Y^{1,2}(\rn)$. From this, we can deduce the estimate $\lVert F_Q(\cdot,0) \rVert_{L^{\frac{2n}{n-2}}(\rn)} \lesssim \lVert \nabla_\parallel F_Q(\cdot,0) \rVert_{L^2(\rn)}$. Consequently, we may use 
the estimate for $\lVert \nabla F_Q(\cdot,0) \rVert_{L^2(\rn)}^2$ obtained above and 
H\"older's inequality to show that
\begin{multline*}
\int_{\rn}  |(B_1)_\perp F_Q(\cdot,0)|^2\le \lVert B_1 \rVert_{n}^2 \lVert F_Q(\cdot, 0) \rVert_{L^{\frac{2n}{n-2}}(\rn)}^2 \lesssim \eps_m^2  \lVert \nabla F_Q(\cdot, 0) \rVert_{L^2(\rn)}^2\\ \lesssim\eps_m^2 \tau^{-2 + 2/(n+1)}|Q|^{-1}.
\end{multline*}
Upon multiplying the previous estimates by $|Q|^2$, we easily obtain the claimed inequality from the ellipticity of $A$.\end{proof}

The next lemma says that we have a Carleson estimate by including the error term.
\begin{lemma}[Good behavior of $\hat b_Q$ vis-\`a-vis Carleson norm of $\hat\Theta_t$]\label{bqest2.lem}
Let $b'_Q$, $b^0_Q$, and $b^{(a)}_Q$ be as above. Then, if $\hat b_Q = (b'_Q,b^0_Q, b^{(a)}_Q)$, we have the estimate
$$\int_0^{\ell(Q)} \int_{Q} \Big|\widehat\Theta_t \hat b_Q(x)\Big|^2 \, \frac{dx \, dt}{t} \leq C|Q| \tau^{-\beta},$$
where $\beta = 2 + 2m- 2/(n+1) > 0$, and $C$ depends on $m$, $n$, $\lambda$, and $\Lambda$.
\end{lemma}
\begin{proof} First, let us show the identity
\begin{equation}\label{greenformulaestFQ.cl}
\widehat\Theta_t \hat b_Q(x) = |Q| t^{m}(\partial_t)^{m+1}F_Q^-,\qquad\text{on }\bb R^{n+1}_+.
\end{equation}
By (an analogue of) Theorem \ref{thm.greenformula} \ref{item.greenformula2}), to show the above identity, it suffices to show that for each $t>0$, the representation
\begin{equation*}
\hat{\Theta}_t \hat{b}_Q = |Q|t^m\partial_t^{m+1}\big( \s^{\cL}_t(\partial_{\nu}^{\cL,-}F_Q)+ \cD^{\cL,+}_t( \tro F_Q)\big)
\end{equation*}
holds in $L^2(\rn)$. For notational convenience, we will write $F_Q^0:=\tro F_Q$. By definition, we have that for any $f\in C_c^\infty(\rn)$,
\begin{multline*}
\langle \hat{\Theta}_t \hat{b}_Q, f\rangle  = \langle |Q|t^m(D_{n+1}^{m+1}\s^{\cL})_t [\partial_{\nu}^{\cL,-} F_Q], f\rangle  \\
+~ \langle |Q|t^m (D_{n+1}^m (\s^{\cL}\nabla)[\tilde{A}\nablap F_Q^0 + B_1F_Q^0])_t , f\rangle ~+~\langle |Q|t^m (D_{n+1}^m\s^{\cL} [ {B_2}_{\|} \cdot \nablap F_Q^0])_t, f \rangle\\
= \langle |Q|t^m(D_{n+1}^{m+1}\s^{\cL})_t [\partial_{\nu}^{\cL,-}F_Q], f\rangle\\ +~ (-1)^m\langle \tilde{A}\nablap F_Q^0+ B_1F_Q^0, |Q|t^m (D_{n+1}^m\nabla \s^{\cL^*}[f])_{-t}\rangle\\+~(-1)^m \langle{B_2}_{\|}\cdot \nablap F_Q^0, |Q|t^m (D_{n+1}^m \s^{\cL^*}[f])_{-t}\rangle .
\end{multline*}
Therefore, it suffices to show that 
\begin{multline*}\label{IBP1.eq} 
\langle |Q|t^m(D_{n+1}^{m+1} \cD^{\cL,+} [F_Q^0])_t, f\rangle =  (-1)^m\langle \tilde{A}\nablap F_Q^0+ B_1F_Q^0, |Q|t^m (D_{n+1}^m\nabla \s^{\cL^*}[f])_{-t}\rangle\\
 \quad + (-1)^m \langle{B_2}_{\|}\cdot \nablap F_Q^0, |Q|t^m (D_{n+1}^m \s^{\cL^*}[f])_{-t}\rangle\\
 =: (-1)^m|Q|t^m I_t. 
\end{multline*}
We  rewrite $I_t$ as follows, using Proposition \ref{IBPonslicesprop.prop}, and the fact that $F_Q^0\in W^{1,2}(\rn)$,
\begin{multline*}
I_t  = \langle \tilde{A}\nablap F_Q^0 +B_1F_Q^0, (\nabla D^m_{n+1} \s^{\cL^*} [f])_{-t}\rangle
 +\langle {B_2}_{\|}\cdot \nablap F_Q^0, (D_{n+1}^m\s^{\cL^*} [f])_{-t}\rangle\\
 = \langle \nablap F_Q^0,\big((A^*\nabla D_{n+1}^m \s^{\cL^*}[f])_{\|}\big)_{-t} + \overline{{B_2}_{\|}} (D_{n+1}^m\s^{\cL^*}[f])_{-t}\rangle\\
\quad + \langle F_Q^0, \overline{B_1} (\nabla D_{n+1}^m\s^{\cL^*}[f])_{-t}\rangle\\
 = (-1)^{m+1}\Big\langle F_Q^0, D_{n+1}^{m+1}\big(\vec A^*_{n+1,\cdot}\nabla(\s^{\cL^*}[f])_{-s} \big)_{s=t}
 + D_{n+1}^{m+1} \big(\overline{{B_2}_\perp} (\s^{\cL^*}[f])_{-s}\big)_{s=t}\Big\rangle\\
 = (-1)^{m+1} \Big.\dfrac{d^{m+1}}{ds^{m+1}}\Big|_{s=t} \big\langle F_Q^0,\vec A^*_{n+1,\cdot}\nabla(\s^{\cL^*}[f])_{-s} + \overline{B_2}_\perp (\s^{\cL^*}[f])_{-s}\big\rangle \\
  = (-1)^{m+1} \Big. \dfrac{d^{m+1}}{dt^{m+1}} \Big|_{s=t} \langle F_Q^0, \partial_{\nu,-s}^{\cL^*,-} (\s^{\cL^*}[f])\rangle 
  =(-1)^{m+2} \Big. \dfrac{d^{m+1}}{dt^{m+1}} \Big|_{s=t} \langle \cD_{s}^{\cL, +}[F_Q^0], f\rangle\\
 = (-1)^{m+2} \langle (D_{n+1}^{m+1} \cD^{\cL,+}[F_Q^0])_t, f\rangle,
\end{multline*}
where we used (i) in Lemma \ref{lm.conormals} in the fifth equality,   we used (ii) of Proposition \ref{prop.adj} in the sixth equality, and we justify the handling of the $t$-derivatives via Proposition \ref{dualwtder.prop}. This concludes the proof of the identity \eqref{greenformulaestFQ.cl}.

Now, we let $a= \frac{\tau \ell(Q)}{1000}$ as before, and note that $(\partial_t)^{m+2}F_Q^-$ is a solution in the half space $\{(x,t) : t > 50a\}$. For $P \in \bb G_a$ and $t \ge 0$, we set
\[
P_t^* = 2P \times \big(t - \tfrac{a}{20}, t + \tfrac{a}{20}\big),\quad\text{and}\quad P_t^{**} = 4P \times \big(t - \tfrac{a}{5}, t + \tfrac{a}{5}\big).
\]
Then using \eqref{Lpgradslices} and then Proposition \ref{Lpcaccop.prop} repeatedly ($m +1$ times) , we obtain for $t \in (0, \ell(Q)]$
\begin{multline*}
 \int_{Q}\big|\widehat\Theta_t \hat b_Q\big|^2 \le \int_{\rn}\big||Q| t^{m}(\partial_t)^{m+1}F_Q^-(\cdot,t)\big|^2= t^{2m}|Q|^2 \sum_{P \in\bb G_a}\int_{P}\big|(\partial_t)^{m+1}F_Q^-(\cdot,t)\big|^2
 \\  \lesssim t^{2m}|Q|^2 a^{-1}\sum_{P \in\bb G_a}\dint_{P_t^*} |(\partial_t)^{m+1}F_Q^-|^2 \lesssim  t^{2m}|Q|^2 a^{-1 - 2m}\sum_{P \in \bb G_a}\dint_{P_t^{**}} |\partial_tF_Q^-|^2
 \\ \lesssim t^{2m}|Q|^2 a^{-1 - 2m} \lVert \nabla F_Q^- \rVert_2^2 \lesssim t^{2m}|Q|^2 a^{-1 - 2m} [\tau \ell(Q)^{n+1}]^{-1 + 2/(n+1)} \lesssim |Q| \tau^{-\beta} \Big(\frac{t}{\ell(Q)} \Big)^{2m},
\end{multline*}
where we used the bounded overlap of $\{P_t^{**}\}_{P \in \bb G_a}$. Hence, we see that
\begin{equation*}
\int_0^{\ell(Q)} \int_{Q} \Big|\widehat\Theta_t \hat b_Q(x)\Big|^2 \, \frac{dx \, dt}{t} \lesssim |Q| \tau^{-\beta} \int_0^{\ell(Q)}  \Big(\frac{t}{\ell(Q)} \Big)^{2m}\, \frac{dx \, dt}{t} \quad \lesssim |Q| \tau^{-\beta}.
\end{equation*}\end{proof}

Observe that Lemma \ref{bqest1.lem} and the properties of $\mu_Q$ allow us to establish that
\begin{equation}\label{smalloffQ.eq}
 \int_{\rn \setminus Q} |b_Q| \,  d\mu_Q\le |(1 + \hm)Q\setminus Q|^{1/2}\lVert b_Q\rVert_{L^2(\rn)} \lesssim\hm^{1/2} \tau^{-1 + 1/(n+1)}|Q|.
\end{equation}
Let us furnish a smallness estimate for $b'_Q$.

\begin{lemma}[Almost atomic behavior of $b_Q'$]\label{bQest4.lem} 

Let $b'_Q$ and $\mu_Q$ be as above. Then
\begin{equation}\label{bQ4smallest1.eq}
\Big| \int_{\rn} b'_Q \, d\mu_Q \Big| \lesssim |Q| \tau^{1/2 + 1/(n+1)},
\end{equation}
where the implicit constant depends on $n$, $\lambda$, and $\Lambda$. In particular,
\begin{equation}\label{bQest4.eq}
\Big|\frac1{\mu_Q(Q)}\int_{Q} b'_Q \, d\mu_Q \Big| \lesssim \tau^{1/2 + 1/(n+1)} + \hm^{1/2} \tau^{-1 + 1/(n+1)}.
\end{equation}
\end{lemma}
\begin{proof}
We first show how to derive \eqref{bQest4.eq} from the first inequality. We have that
$$\Big| \int_{Q} b'_Q \, d\mu_Q\Big| \le \Big| \int_{\rn} b'_Q \, d\mu_Q \Big|  +  \int_{\rn \setminus Q} |b_Q| \,  d\mu_Q,$$
so that \eqref{bQest4.eq} readily follows from \eqref{bQ4smallest1.eq}, \eqref{smalloffQ.eq}, and the fact that $\mu_Q(Q) \ge (1/2)^n|Q|$.
It remains to show \eqref{bQ4smallest1.eq}. To this end, we utilize the properties of $\phi_Q$, \eqref{slideFq.eq}, \eqref{Lpgradslices} and H\"older's inequality to see that
\begin{multline*} 
\Big| \int_{\rn} b'_Q \, d\mu_Q \Big| = |Q| \Big| \int_{\rn} \nabla_{\parallel} F_Q(\cdot,0) \phi_Q\Big|
 = |Q|  \Big| \int_{\rn} F_Q(\cdot,0) \nabla \phi_Q\Big|
\\ \lesssim \ell(Q)^{n-1} \int_{(1 + \hm)Q \setminus (1/2)Q} |F_Q(\cdot,0)| 
\\  \lesssim \ell(Q)^{n-1} \int_{(1 + \hm)Q \setminus (1/2)Q} \Big |\int_{-\tfrac{3}{2} \tau \ell(Q)}^{\tfrac{3}{2} \tau \ell(Q)} \partial_t F_Q^{s'}(y,0) \, ds' \Big| \, dy 
\\  \lesssim \ell(Q)^{n-1} \int_{-\tfrac{3}{2} \tau \ell(Q)}^{\tfrac{3}{2} \tau \ell(Q)} \int_{(1 + \hm)Q \setminus (1/2)Q} \Big | \partial_t F_Q^{s'}(y,0)  \Big| \, dy \, ds'
\\  \lesssim  \ell(Q)^{n-1}  \frac{\ell(Q)^{n/2}}{\ell(Q)^{1/2}}   \int_{-\tfrac{3}{2} \tau \ell(Q)}^{\tfrac{3}{2} \tau \ell(Q)} \Big( \dint_{I_{2Q} \setminus I_{(1/4)Q}} |\nabla F_Q^{s'}(Y)|^2 \, dY\Big)^{1/2}\,ds'
\\ \lesssim |Q| \tau^{1/2 + 1/(n+1)},
\end{multline*}
where we used \eqref{gradestFQ} in the last line and, in order to use \eqref{Lpgradslices}, we used that for $s \in (-\tfrac{3}{2} \tau \ell(Q), \tfrac{3}{2} \tau \ell(Q))$ each $F_Q^{s'}$ is a solution in $I_{2Q} \setminus I_{(1/4)Q}$.\end{proof} 

The last preliminary lemma we will need establishes a coercivity estimate for $b^0_Q$.

\begin{lemma}[Coercivity of $b_Q^0$]\label{bqest3.lem}
Let $b^0_Q$ and $d\mu_Q = \phi_Q \, dx$ as above. Suppose that $\ep_m>0$ is a small number depending on $m$. Then, if $\max\{\| B_1\|_n, \|B_2\|_n\} \le \eps_m$, the estimate 
$$\f Re\Big(\frac1{\mu_Q(Q)}\int_{Q} b^0_Q \, d\mu_Q\Big) \ge \Big(1 - C\big[\tau^{1/2 + 1/(n+1)} + \eps_m\tau^{-1/2 + 1/(n+1)} + \hm^{1/2} \tau^{-1 + 1/(n+1)}\big]\Big),$$
holds, where $C$ depends on $m$, $n$, $\lambda$, and $\Lambda$.
\end{lemma}
\begin{proof}
By the definitions of $\mu_Q$, $b^0_Q$, and the conormal derivative, we observe that
\begin{multline*} 
\int_{\rn} b^0_Q \, d\mu_Q = \int_{\rn} b^0_Q \phi_Q  =|Q| \int_{\rn} (\partial_{\nu}^{\cL,-}F_Q)(y,0) \phi_Q(y) \, dy
\\  = |Q|\Big(- \langle \Phi_Q, \cL F_Q \rangle_{\ree_-} + \dint_{\ree} A \nabla F_Q \cdot \nabla \Phi_Q + (B_1 F_Q)\cdot \nabla \Phi_Q +(B_2\cdot\nabla F_Q) \Phi_Q \Big)
\\ = |Q|(I + II).
\end{multline*}
Since $\supp\Psi_Q^+\cap\bb R^{n+1}_-=\varnothing$, $\Phi_Q \equiv 1$ on $\supp \Psi_Q^-$, and $\dint_{\ree_-} \Psi_Q^-= 1$, we have that
\begin{equation*} 
I = - \langle \Phi_Q, \cL F_Q \rangle_{\ree_-} = - \dint_{\ree_-} (-\Psi_Q^-) = 1.
\end{equation*}

To bound $II$,  we write $II = II_1 + II_2 + II_3$, where the $II_i$ correspond to each of the summands in the integral defining $II$. For the term, $II_1$, we use essentially the same estimates as in the previous lemma. In particular we use the properties of $\Phi_Q$, H\"older's inequality, the Caccioppoli inequality, and \eqref{slideFq.eq} to obtain that  
\begin{multline*} 
|II_1|  \le \dint_{\ree} \Big |A \nabla F_Q \cdot \nabla \Phi_Q\Big|   \lesssim \frac{1}{\ell(Q)} \dint_{I_{(1+ \hm)Q} \setminus I_{(1/2)Q}} |\nabla F|  
\\  \lesssim \ell(Q)^{(n-1)/2}\Big( \dint_{I_{(1+ \hm)Q} \setminus I_{(1/2)Q}} |\nabla F|^2  \Big)^{1/2}
  \lesssim \ell(Q)^\frac{(n-3)}{2}\Big( \dint_{I_{(1+ \hm)Q} \setminus I_{(1/2)Q}} |F|^2\Big)^{1/2}
\\  \lesssim \ell(Q)^\frac{(n-3)}{2}\Big( \dint_{I_{(1+ \hm)Q} \setminus I_{(1/2)Q}} \Big|\int_{-\tfrac{3}{2} \tau \ell(Q)}^{\tfrac{3}{2} \tau \ell(Q)} \partial_t F_Q^{s'}(Y) \, ds' \Big|^2\,dY \Big)^{1/2}
  \lesssim \tau^{1/2 + 1/(n+1)}.
\end{multline*}

To bound $II_2$, we use the estimate $\|B_1 F_Q\|_{2} \lesssim \eps_m \| \nabla F_Q \|_{2}$ and  \eqref{gradestFQ} to see that
\begin{multline*} 
|II_2|  \le \dint_{I_{2Q}} |(B_1 F) \cdot \nabla \Phi_Q|  \lesssim \frac{1}{\ell(Q)} \dint_{I_{2Q}} |B_1 F_Q|  
\\  \lesssim \frac{\ell(Q)^\frac{n+1}{2}}{\ell(Q)}  \Big(\dint_{I_{2Q}} |B_1 F_Q|^2\Big)^{1/2}  \lesssim \eps_m \ell(Q)^\frac{n-1}{2} \| \nabla F_Q \|_{2}
  \lesssim \eps_m \tau^{-1/2 + 1/(n+1)}.
\end{multline*}

To bound $II_3$ we use H\"older's inequality, $\| B_2 \|_{n} \le \eps_m$, and \eqref{gradestFQ} as follows:
\begin{multline*} 
|II_3| \le \int_{-2\ell(Q)}^{2\ell(Q)} \int_{2Q} |\nabla F_Q B_2| \le \eps_m \int_{-2\ell(Q)}^{2\ell(Q)} \Big(\int_{2Q} |\nabla F_Q|^\frac{n}{n-1}\Big)^\frac{n-1}{n}
\\ \lesssim \eps_m \ell(Q)^\frac{n-2}{2} \int_{-2\ell(Q)}^{2\ell(Q)} \Big(\int_{2Q} |\nabla F_Q|^2 \Big)^\frac{1}{2} \lesssim \eps_m \ell(Q)^\frac{n-1}{2}  \Big(\int_{-2\ell(Q)}^{2\ell(Q)}\int_{2Q} |\nabla F_Q|^2\Big)^\frac{1}{2}\\ \lesssim \eps_m \ell(Q)^\frac{n-1}{2} \| \nabla F_Q \|_{2} \lesssim  \eps_m \tau^{-1/2 + 1/(n+1)}.
\end{multline*}

Combining the previous estimates gives that
$$\f Re \Big(\int_{\rn} b^0_Q \, d\mu_Q \Big) \ge |Q|\Big(1 - C\big[\tau^{1/2 + 1/(n+1)} + \eps_m\tau^{-1/2 + 1/(n+1)}\big]\Big).$$
This estimate, in concert with \eqref{smalloffQ.eq} and the fact that $\mu_Q(Q) \le 1$, ends the proof.\end{proof}

With $\eps_m$ and $\hm$ at our disposal, we collapse the dependence of parameters to only $\tau$, 
leaving freedom to take $\eps_m$ even smaller. We ensure that $\eps_m < \tau$ and set $\hm = \tau^3$. Under these choices, we are ready to present the

\begin{proof}[Proof of Proposition \ref{allbQproperties.lem}] When the choices $\eps_m < \tau$ and $\hm = \tau^3$ are used in Lemma \ref{bqest3.lem}, we have that
	$$\f R e \Big(\frac1{\mu_Q(Q)}\int_{Q} b^0_Q \, d\mu\Big) \ge 1 - C\tau^{1/2 + 1/(n+1)},$$
	where $C$ depends on $n$, $\lambda$, $\Lambda$. Accordingly, we may pick $\tau$ small enough so that \eqref{destbQ3.eq} holds. The choice $\hm = \tau^3$ used in \eqref{bQest4.eq} gives that
	$$
	\Big|\frac1{\mu_Q(Q)}\int_{Q} b'_Q \, d\mu_Q \Big| \le C\tau^{1/2 + 1/(n+1)},$$
	where $C$ depends on $n$, $\lambda$, and $\Lambda$. Hence, we may guarantee that \eqref{destbQ4.eq} holds by choosing $\tau$ small depending on $C$ and $\eta$. Having chosen $\tau$ so that \eqref{destbQ3.eq} and \eqref{destbQ4.eq} hold, \eqref{destbQ1.eq} and \eqref{destbQ2.eq} follow from Lemma \ref{bqest1.lem} and Lemma \ref{bqest2.lem} respectively.\end{proof}

\subsection{Control of the auxiliary square functions}

As a last preliminary step to presenting the proof of the square function bound, we elucidate how to control the error terms involving $\Theta^{(a)}_t$ and $\Theta'_t$. 

\begin{proposition}[Control of error terms]\label{Carlesonerrorest.prop}
	Let $T_t$ be either $\Theta'_t$ or $\Theta^{(a)}_t$. Then, for each fixed $t>0$, $T_t 1$ is well defined as an element of $L^2_{\loc}(\rn)$. Moreover, we have the estimates
	\begin{equation}\label{Carerrest1boundleq.eq}
	\||T_t |\|_{op}  \leq C\||\Theta_t^0 |\|_{op} + 1,
	\end{equation}
	and
	\begin{equation}\label{Carerrest2boundleq.eq}
	\|T_t 1 \|_{\C}  \leq C \|\Theta_t^0 1\|_{\C} + 1.
	\end{equation}
	where $C$ depends on $m$, $n$, $\lambda$, and $\Lambda$, provided that $\max\{\| B_1\|_n, \|B_2\|_n\}$ is sufficiently small depending on $m$, $n$, $\lambda$, and $\Lambda$.
\end{proposition}
\begin{remark}
	We will operate under the assumption that  $T_t 1$ and $\Theta^0_t 1$ have finite $\lVert \cdot \rVert_{\C}$ norm. Indeed, otherwise for $\gamma > 0$, we replace $T_t1$ by $(T_t 1)_\gamma = (T_t 1)\mathbbm{1}_{\gamma < t \le 1/\gamma}$ and analogously for $\Theta^0_t 1$, and we observe that these truncated versions will always have finite $\lVert \cdot \rVert_{\C}$ norm under our hypotheses. 
\end{remark}

Proposition \ref{Carlesonerrorest.prop} will be a direct consequence of the following lemma.

\begin{lemma}[Control of gradient field terms]\label{tTboundlem.lem}Let $\widetilde \Theta_t := t^m \partial_t^m \s^{\cL}_t \nabla_\parallel$ for $m \in \N$, $m \ge n + 10$. Then 
	\begin{equation}\label{tTboundleq.eq}
	\|| \tT_t |\|_{op}  \lesssim \||\Theta_t^0 |\|_{op} + 1,
	\end{equation}
	and
	\begin{equation}\label{tT1boundleq.eq}
	\|\tT_t 1 \|_{\C}  \lesssim \|\Theta_t^0 1\|_{\C} + 1,
	\end{equation}
	where the constants depends on $m$, $n$, $\lambda$, and $\Lambda$, provided that $\max\{\| B_1\|_n, \|B_2\|_n\}$ is sufficiently small depending on $m$, $n$, $\lambda$, $\Lambda$.
\end{lemma}
\begin{proof}
	We note that \eqref{tT1boundleq.eq} follows from Lemma \ref{FSAAAHKL3.2.lem}, \eqref{tTboundleq.eq} and Proposition \ref{ODforsevop.prop}.
	The proof will follow the general scheme of \cite[Lemma 3.1]{Hof-May-Mour}, with modifications due to the first order terms. Write $L_\parallel:= \div_x A_\parallel \nabla_\parallel$ where $A_\parallel = (A_{i,j})_{1 \le i,j \le n}$. 
	By the Hodge decomposition for the operator $L_\parallel$, to prove \eqref{tTboundleq.eq} it is enough to show that  
	\begin{equation}\label{tTboundlredeq.eq}
	\dint_{\ree_+}\big|t^m\partial_t^m\s^{\cL}_t(\nabla_\parallel \cdot A_{\|} \nabla_\parallel F)(x)\big|^2 \, \frac{dx \, dt}{t}
	\lesssim (1 + \|| \Theta_t^0|\|^2_{op}) ,
	\end{equation}
	for all $F \in Y^{1,2}(\bb R^n)$ with $\|\nabla_{\|} F \|_{L^2}\lesssim1$ (dependence on $\lambda$ and $\Lambda$). We write
	\begin{multline*}
	t^m \partial_t^m\s^{\cL}_t \nabla_\parallel A_{\|}\nabla_{\|} F\\=\big\{ t^m \partial_t^m\s^{\cL}_t \nabla_\parallel A _\parallel - t^m\big(\partial_t^m \s^{\cL}_t \nabla_\parallel A_\parallel\big) P_t\big\} \nabla_\parallel F + t^m \big(\partial_t^m \s^{\cL}_t \nabla_\parallel  A_\parallel\big) P_t \nabla_\parallel F
	\\  =: R_t (\nabla_\parallel F) + t^m \big(\partial_t^m \s^{\cL}_t \nabla_\parallel \cdot A_\parallel\big) P_t \nabla_\parallel F,
	\end{multline*}
	where  $t^m (\partial_t^m \s_t \nabla_\parallel A_\parallel)$ is the (vector-valued) operator $t^m \partial_t^m \s_t \nabla_\parallel$ applied to $A_\parallel$, the latter understood as a vector function with components in $L^2_{loc}(\rn;\bb C^n)$, and $P_t$ is a nice approximate identity constructed as follows. Let $\zeta_t(x) = t^{-n} \zeta\Big(\frac{|x|}{t}\Big)$, where $\zeta \in C_c^\infty(B(0,1/2))$ is radial with $\int_{\rn} \zeta = 0$ and $\m Q_t f(x) = (\zeta_t \ast f)(x)$ satisfies the Calder\'{o}n reproducing formula
	\begin{equation*}
	\int_0^\infty\m Q_t^2 \, \frac{dt}{t} = I , \quad \text{ in the strong operator topology on } L^2.
	\end{equation*}
	Then $\m Q_s$ is a CLP family (see Definition \ref{CLP.def}) and we set $P_t := \int_t^\infty\m Q_s^2 \, \frac{ds}{s}$. 	Then $P_t$ is a nice approximate identity; that is, $P_t = (\varphi_t \ast f)(x)$ where $\varphi_t = t^{-n}\varphi\big(\frac{|\cdot|}{t}\big)$ and $\varphi \in C_c^\infty(B(0,1))$ is a radial function with $\int_{\rn} \varphi = 1$. 
	
	The term $t^m  \partial_t^m \s_t \nabla_\parallel \cdot A_\parallel P_t \nabla_\parallel F$ is the `main term' and we will apply the techniques of the solution to the Kato problem \cite{Katoproblem} to handle its contribution. For now, we focus on the remainder term $R_t (\nabla_\parallel F)$, which takes a bit of exposition due to the number of terms arising from the lower order terms in the differential operator $\cL$. To this end, we write
	\begin{multline*}
	R_t  =   t^m \partial_t^m\s^{\cL}_t \nabla_\parallel A _\parallel - t^m\big(\partial_t^m \s^{\cL}_t \nabla_\parallel A_\parallel\big) P_t  
	\\   = \big\{ t^m \partial_t^m\s^{\cL}_t \nabla_\parallel A _\parallel P_t- t^m \big(\partial_t^m \s^{\cL}_t \nabla_\parallel  A_\parallel\big) P_t \big\} +  t^m \partial_t^m\s^{\cL}_t \nabla_\parallel A _\parallel (I - P_t)
	=: R_t^{[1]} + R_t^{[2]}.
	\end{multline*}
	Observe that $R_t^{[1]} 1 = 0$, $R_t^{[1]}$ has sufficient off-diagonal decay (Proposition \ref{ODforsevop.prop}) and uniform $L^2$ boundedness (Proposition \ref{L2avgandsliceboundsperturb.prop}), and $\| R^{[1]}_t \nabla_x\|_{2 \to 2} \le C/t$. Then the square function bound
	$$\dint_{\ree_+} |R_t^{[1]}\nabla_\parallel F|^2 \, \frac{dx \, dt}{t} \lesssim \| \nabla_\parallel F \|_2^2$$
	follows from  Lemma \ref{AAAHKL3.5.lem} as desired. To control $R_t$ it remains to control $R_t^{[2]}$. Set $Z_t := I - P_t$ and define $\vec{b} := (A_{n+1,1}, \dots, A_{n+1,n})$. By using integration by parts on slices (Proposition \ref{IBPonslicesprop.prop}) and Proposition \ref{CLPW12.prop}, we obtain that 
	\begin{multline*}
	t^m \partial_t^m \s_t \nabla_\parallel  A_\parallel Z_t \nabla_\parallel F  =
	t^m \partial_t^m \s_t \nabla_\parallel A_\parallel \nabla_\parallel Z_t F
	\\  = t^m \partial_t^{m+1} (\s_t \nabla) \cdot \vec A_{\cdot,n+1} Z_t F
	\quad - t^m \partial_t^{m+1} \s_t(\vec b \nabla Z_t F)
	\quad + t^m\partial_t^m (\s_t \nabla)B_1 Z_t F
	\\  \quad - t^m\partial_t^m \s_t({B_2}_\parallel \nabla_\parallel Z_t F)
	\quad + t^m\partial_t^{m+1} \s_t({B_2}_\perp  Z_t F)
	\qquad =: J_1 + J_2 + J_3 + J_4 + J_5.
	\end{multline*}
	Note that, using   Plancherel's theorem, we have that
	\begin{equation}\label{plansqfnest.eq}
	\dint_{\ree_+} |t^{-1} (I - P_t)F(x)|^2 \, \frac{dx\, dt}{t} \lesssim \|\nabla_\parallel F \|_2^2.
	\end{equation}
	Since $t^{m+1} \partial_t^{m+1} (\s_t \nabla) :L^2 \to L^2$ uniformly in $t$, we easily obtain the associated square function bound for   $J_1$. To bound $J_2$, we write
	\begin{multline*}
	J_2  = -t^m\partial_t^{m +1} \s_t\big(\vec{b} \cdot \nabla_\parallel(I - P_t)F\big)
	\\   =  -t^m\partial_t^{m +1} \s_t\vec{b} \cdot \nabla_\parallel F + \{ t^m\partial_t^{m +1} \s_t\vec{b} P_t   - (t^m\partial_t^{m +1} \s_t\vec{b}) P_t\} \nabla_\parallel F + (t^m\partial_t^{m +1} \s_t\vec{b}) P_t \nabla_\parallel F 
	\\   =: J_{2,1} + J_{2,2} + J_{2,3}.
	\end{multline*}
	For $J_{2,1}$, we see that $J_{2,1} = \Theta_t^0\vec{b} \nabla_\parallel F$, whence
	$$\dint_{\ree_+}|t^m \partial_t^{m+1} \s_t \vec{b} \nabla_\parallel F|^2 \, \frac{dx \, dt}{t} \lesssim \|| \Theta_t^0 \||^2_{op} \| \nabla_\parallel F \|_2^2.$$
	Similarly, by Lemma \ref{FSAAAHKL3.2.lem} and Carleson's Lemma, we have that
	$$\dint_{\ree_+}|(t^m\partial_t^{m +1} \s_t\vec{b}) P_t \nabla_\parallel F|^2 \, \frac{dx \, dt}{t} \lesssim \|| \Theta_t^0 \||^2_{op} \| \nabla_\parallel F \|_2^2,$$
	so that the contribution from $J_{2,3}$ has the desired control.   Notice that $J_{2,2}$ is of the form $R_t \nabla_\parallel F$ where $R_t 1 = 0$, $R_t: L^2 \to L^2$ and $\|R_t \nabla_x \|_{L^2 \to L^2} \le C/t$ and $R_t$ good off-diagonal decay. Thus,  the desired square function bound for term $J_{2,2,}$ follows immediately from Lemma \ref{AAAHKL3.5.lem}.

	For term $J_3$, let $g$ be such that $I_1 g = F$ and $\|g\|_2 \approx \|\nabla_\parallel F\|_2$. Then using $t^m\partial_t^m (\s_t \nabla)= \Theta_t^{(a)}$, we have by Proposition \ref{QOforSnabBI.lem} that
	$$\| \Theta_t^{(a)} B_1I_1\m Q_s^2 g\|_{L^2(\rn)} \lesssim  \Big(\frac{s}{t}\Big)^\gamma\|\m Q_s g \|_{L^2(\rn)}$$
	for some $\gamma > 0$ independent of $ g$. 
	Then by standard estimates we obtain
	\begin{multline*}
	\dint_{\ree_+} \Big|t^m \partial_t^m (\s_t \nabla) B_1(I - P_t)F \Big|^2\, \frac{dx \, dt}{t}
	=  \dint_{\ree_+} \Big|t^m \partial_t^m (\s_t \nabla) B_1I_1(I - P_t)g \Big|^2\, \frac{dx \, dt}{t}
	\\   \lesssim  \dint_{\ree_+}\Big|t^m \partial_t^m (\s_t \nabla) B_1I_1\Big(\int_0^t\m Q^2_s g \frac{ds}{s}\Big) \Big|^2\, \frac{dx \, dt}{t}
	\\    \lesssim_{\gamma} \dint_{\ree_+} \int_0^t\Big(\frac{t}{s}\Big)^{\gamma/2}\Big|t^m \partial_t^m (\s_t \nabla) B_1 I_1\m Q_s^2 g\Big|^2 \, \frac{ds}{s} \, \frac{dx \, dt}{t}
	\\    \lesssim \int_0^\infty \int_s^\infty \Big(\frac{s}{t}\Big)^{\gamma/2}\|\m Q_s g \|_2^2 \, \frac{dt}{t} \, \frac{ds}{s}
	\lesssim_\gamma \int_0^\infty \|\m Q_s g \|_2^2\, \frac{ds}{s}
	\lesssim \| g \|_2^2 \approx \| \nabla_\parallel F \|_2^2,
	\end{multline*}
	where in the fourth inequality we used Cauchy's inequality in the $\frac{ds}{s}$ integral noting that $\int_0^t (s/t)^\gamma \frac{ds}{s} \lesssim C_\gamma$, and we used the square function estimate for the CLP family $\m Q_s$ (see Definition \ref{CLP.def}). This takes care of the contribution from $J_3$.
	
	Next, we handle $J_4$. We write $J_4$ as the sum of its pieces, as follows:
	\begin{multline*}
	J_4  = -t^m \partial_t^m \s_t{B_2}_\parallel \nabla_\parallel (I - P_t) F
	\\   = -t^m \partial_t^m \s_t{B_2}_\parallel \nabla_\parallel F + t^m \partial_t^m \s_t{B_2}_\parallel \nabla_\parallel P_t F
	= J_{4,1} + J_{4,2}.
	\end{multline*}
	For $J_{4,1}$, we observe that
	\begin{equation*}
	J_{4,1} = - t^m \partial_t^m \s_t {B_2}_\parallel \nabla_\parallel F  = - t^m \partial_t^m \s_t \div_\parallel \nabla_\parallel I_2 {B_2}_\parallel \nabla_\parallel F 
	= -\tT(\nabla_\parallel I_2 {B_2}_\parallel F)
	\end{equation*}
	and notice that $\| \nabla_\parallel I_2{B_2}_\parallel F\|_2 \lesssim \| B_2 \|_{n} \| \nabla_\parallel F \|_2$.
	Therefore,
	$$\dint_{\ree_+} | t^m \partial_t^m \s_t{B_2}_\parallel \nabla_\parallel F|^2 \, \frac{dx \, dt}{t} \lesssim \||\tT_t\||_{op}^2 \| B_2 \|_{n}^2 \| \nabla_\parallel F \|_2^2,$$
	and hence  $J_{4,1}$ can be hidden in \eqref{tTboundlredeq.eq}  when  $\| B_2 \|_{n}$ is   small. For $J_{4,2}$, we write
	\begin{multline*}
	J_{4,2}  =\big\{t^m\partial_t^m \s_t{B_2}_\parallel  P_t  - (t^m\partial_t^m \s_t {B_2}_\parallel)P_t\big\}\nabla_\parallel F +  (t^m\partial_t^m \s_t{B_2}_\parallel)P_t \nabla_\parallel F
	\\   = \widetilde R_t\nabla_\parallel F +  (t^m\partial_t^m \s_t{B_2}_\parallel)P_t \nabla_\parallel F.
	\end{multline*}
	We may handle $\widetilde R_t\nabla_\parallel F$ using Lemma \ref{AAAHKL3.5.lem}, as $\widetilde R_t$ satisfies the required hypotheses (see Propositions \ref{L2avgandsliceboundsperturb.prop} and \ref{ODforsevop.prop}). We see, in a similar fashion to $J_{4,1}$,   that $t^m\partial_t^m \s_t{B_2}_\parallel = \tT_t\nabla_\parallel I_2 {B_2}_\parallel$, and $\| \nabla_\parallel I_2{B_2}_\parallel\|_{BMO} \lesssim \|B_2 \|_n^2$. Noting that $\tT_t 1 = 0$, it follows from Lemma \ref{FSAAAHKL3.2.lem} and Carleson's Lemma that
	$$\dint_{\ree_+}\big|(t^m\partial_t^m \s_t{B_2}_\parallel)P_tF\big|^2 \, \frac{dx \, dt}{t} \lesssim (1 + \||\tT_t \||_{op}^2)\|B_2\|_n^2 \|\nabla_\parallel F\|_2^2,$$
	which can be  hidden in \eqref{tTboundlredeq.eq} when $\|B_2\|_n$ is sufficiently small.
	
	Finally, to handle $J_5$, rewrite it as $J_5 = t^{m+1}\partial_t^{m+1} \s_t{B_2}_\perp(\tfrac{1}{t}[I - P_t]F)$. Since\\ $t^{m+1}\partial_t^{m+1} \s_t  {B_2}_\perp : L^2 \to L^2$ uniformly in $t$, we may handle this term exactly as $J_1$ by using \eqref{plansqfnest.eq}.
	
	Having handled the remainder $R_t$, we have reduced matters to showing that the square function bound
	$$\dint_{\ree_+} |t^m (\partial_t^m \s_t \nabla_\parallel \cdot A_\parallel )(x)P_t \nabla_\parallel F(x)|^2 \, \frac{dx \, dt}{t} \lesssim \| \nabla_\parallel F\|_2^2$$
	holds for all $F \in Y^{1,2}(\bb R^n)$ with $\Vert\nabla_{\|}F\Vert_2\leq1$. By Carleson's Lemma, it is enough to show that
	\begin{equation}\label{CMesttTKato.eq}
	\sup_{Q} \frac{1}{|Q|} \int_0^{\ell(Q)} \int_{\rn} |t^m (\partial_t^m \s_t \nabla_\parallel \cdot A_\parallel )(x)|^2 \, \frac{dx \, dt}{t} \le C.
	\end{equation}
	
	In order to obtain \eqref{CMesttTKato.eq}, we appeal to the technology of the solution of the Kato problem \cite{Katoproblem}, and follow the argument of \cite{Hof-May-Mour}. By \cite{Katoproblem}, for each dyadic cube $Q$ there exists a mapping $F_Q : \rn \to \mathbb{C}^n$ such that
	\begin{enumerate}[i)]
		\item $\displaystyle \int_{\rn} |\nabla_\parallel F_Q|^2  \le C|Q|$
		\item $\displaystyle \int_{\rn}|L_\parallel F_Q|^2   \le \frac{|Q|}{\ell(Q)^2}$
		\item $\displaystyle \sup_Q \int_0^{\ell(Q)} \dashint_Q |\vec\zeta (x, t)|^2 \, \frac{dx \, dt}{t}
		\lesssim C \sup_Q \int_0^{\ell(Q)} \dashint_Q |\vec\zeta(x,t) E_t \nabla_\parallel F_Q|^2 \, \frac{dx \, dt}{t}$
	\end{enumerate}
	for each $\vec \zeta:  \ree_+ \to \mathbb{C}^n$, where $E_t$ denotes the dyadic averaging operator; that is, if $Q(x,t)$ is the minimal dyadic cube containing $x\in\bb R^n$ with side length at least $t$, then $E_t g(x) = \dashint_{Q(x,t)} g $. Here, we note that $\nabla_\parallel F_Q$ is the Jacobian of $F_Q$ and $\vec \zeta E_t \nabla_\parallel F_Q $ is a vector. Given such a family $\{F_Q\}_Q$, we see that by applying property $iii)$ with
	$\vec \zeta (x,t) = T_tA_\parallel$, where $T_t : = t^m \partial_t^m (\s_t \nabla_\parallel)$ it is enough to show that
	\begin{equation*} 
	\int_0^{\ell(Q)} \int_Q |(T_tA_\parallel)(x) E_t \nabla_\parallel F_Q(x)|^2\, \frac{dx \, dt}{t}
	\lesssim(1 + \||\Theta_t^0\||^2_{op})|Q|.
	\end{equation*} 
	Following \cite{AT, CM}, we write that
	\begin{multline*}
	(T_tA_\parallel) E_t \nabla_\parallel F_Q  = \{(T_tA_\parallel) E_t - T_tA_\parallel\}\nabla_\parallel F_Q + T_tA_\parallel \nabla_\parallel F_Q
	\\  = T_t A_\parallel(E_t - P_t)\nabla_\parallel F_Q + \{(T_t A_\parallel)P_t - T_tA_\parallel\} \nabla_\parallel F_Q + T_tA_\parallel\nabla_\parallel F_Q
	\\  =: R_t^{(1)}\nabla_\parallel F_Q +  R_t^{(2)}\nabla_\parallel F_Q + T_tA_\parallel \nabla_\parallel F_Q.
	\end{multline*}
	Observe that $R_t^{(2)} = -R_t$ from above, and we have already shown that $\|| R_t \||_{op} \lesssim (1 + \||\Theta_t^0\||_{op})$\footnote{We have shown that $\|| R_t \||_{op} \lesssim (1 + \||\Theta_t^0\||_{op}) + \epsilon\|| \tT_t \||_{op}$, where $\epsilon$ is at our disposal by the smallness of $\max\{\|B_1\|_n, \|B_2\|_n\}$, and this is enough for our purposes.}, so   that the desired bound holds from property $i)$ of $F_Q$. For the last term,  we have that  $T_t A_\parallel \nabla_\parallel F_Q = t^m \partial_t^m \s_t L_\parallel F_Q$, and we know that $ t^{m-1} \partial_t^m \s_t: L^2 \to L^2$ uniformly in $t$. Thus, by property $ii)$  of $F_Q$, we have that
	\begin{multline*}
	\int_0^{\ell(Q)} \int_Q|(T_tA_\parallel F_Q)(x)|^2 \, \frac{dx \, dt}{t}  \le
	\int_0^{\ell(Q)}\int_{\rn} |t^{m-1} \partial_t^m \s_t L_\parallel F(x)|^2 t\,dx \, dt
	\\   \lesssim \frac{|Q|}{\ell(Q)^2}\int_0^{\ell(Q)}t \, dt \lesssim |Q|,
	\end{multline*}
	which shows the desired bound for this term.
	
	To bound the contribution from $R_t^{(1)}$, we note that $T_t : L^2 \to L^2$ uniformly in $t$ and
	$$\dint_{\ree_+} |(E_t - P_t)g(x)|^2 \, \frac{dx\, dt}{t} \lesssim \| g\|_2^2$$
	for $g \in L^2(\rn)$. Therefore,  
	\begin{multline*}
	\int_0^{\ell(Q)} \int_Q|R_t^{(1)} \nabla_\parallel F_Q|^2 \, \frac{dx\, dt}{t}  \le
	\int_0^{\ell(Q)} \int_{\rn}|T_tA_{\|}(E_t - P_t) \nabla_\parallel F_Q|^2\, \frac{dx \, dt}{t}
	\\   \lesssim \int_0^{\ell(Q)} \int_{\rn}|(E_t - P_t) \nabla_\parallel F_Q|^2\, \frac{dx \, dt}{t}
	\\   \lesssim \|\nabla_\parallel F \|_2^2 \lesssim C|Q|,
	\end{multline*}
	where we used the ellipticity of $A$ in the second inequality, and property $i)$ of $F_Q$ in the last inequality. This controls the contribution from $R_t^{(1)}$ and finishes the proof of the Lemma.\end{proof}

We move on to the 

\begin{proof}[Proof of Proposition \ref{Carlesonerrorest.prop}]
	To see that $\||\Theta_t^{(a)}\||_{op} \lesssim 1 + \||\Theta_t^0\||_{op}$, and that
	$\|\Theta_t^{(a)} 1\|_{\C} \lesssim 1 + \|\Theta_t^0\|_{\C}$, 
	we  simply notice  that $\Theta_t^{(a)} = (\Theta_t^0 , t^m \partial_t^m (\s_t \nabla_\parallel))$
	so that the desired bounds follow directly from the previous lemma.
	
	We are left with showing the bounds in Proposition \ref{Carlesonerrorest.prop} for $T_t = \Theta'_t$. We note immediately that \eqref{Carerrest2boundleq.eq} will follow from \eqref{Carerrest1boundleq.eq} and Lemma \ref{FSAAAHKL3.2.lem}. Therefore, it is enough to show \eqref{Carerrest1boundleq.eq}. In fact, by Lemma \ref{tTboundlem.lem}, it suffices to show that $
	\||\Theta'_t\||_{op} \lesssim \|| \tT_t \||_{op} + \||\Theta_t^0\||_{op}$. 	For $g \in L^2(\rn, \mathbb{C}^n)$, we have that
	\begin{multline*}
	\Theta'_t g  = t^m \partial_t^m \s_t ({B_2}_\| g) + t^m \partial_t^m (\s_t \nabla)\cdot \tilde{A} g
	\\ =  t^m \partial_t^m \s_t ({B_2}_\| g) + t^m \partial_t^m (\s_t \nabla_\parallel)\cdot A_\parallel g - t^m \partial_t^{m +1} \s_t \vec{b} g,
	\end{multline*}
	where $\vec{b} = (A_{n+1, j})_{1 \le j \le n}$. The ellipticity of $A$ gives  immediately that \\
	$\|| t^m \partial_t^m(\s_t \nabla_\parallel) A_\parallel \||_{op} \lesssim \|| \tT\||_{op}$, and $\|| t^m \partial_t^{m +1} \s_t \vec{b} \||_{op} \lesssim \|| \tT\||_{op}$. It remains to handle the first term. Observe that ${B_2}_\| g = \div_\parallel \nabla_\parallel I_2{B_2}_\| g = \div_\parallel \vec{R} I_1{B_2}_\| g ,$ where   $\vec{R}$ is the vector-valued Riesz tranform. It follows that ${B_2}_\| g= \div_\| \vec{G}$ with $\| \vec{G} \|_2 \lesssim \|B_2 \|_n \|g \|_2$, and hence
	\begin{equation*}
	\|| t^m\partial_t^m S_t B_2 \||_{op} \lesssim \||\tT_t\||_{op} \|B_2 \|_n,
	\end{equation*}
	which yields the desired bound.\end{proof}

\subsection{Proof of the square function bound} We finally turn to the proof of Theorem \ref{Tbargument.thrm} (and hence, by our reduction, the proof of Theorem \ref{perturbfromlargem.thrm}). Our method  follows  the lines of \cite{GH}, circumventing some difficulties by introducing $\Theta^{(a)}_t$ and $b_Q^{(a)}$.
\begin{proof}[Proof of Theorem \ref{Tbargument.thrm}] Let $C_1$ be a constant, depending on $m$, $n$, $\lambda$ and $\Lambda$, for which the inequalities (\ref{Carerrest1boundleq.eq}) and (\ref{Carerrest2boundleq.eq}) hold. We   choose $\eta$ in Proposition \ref{allbQproperties.lem} as $\eta := 1/(2C_1 + 4)$. By the generalized Christ-Journ\'e $T1$ theorem for square functions, (see \cite[Theorem 4.3]{GH}) to prove the theorem it is enough\footnote{The careful reader will notice that we have verified the hypotheses of  \cite[Theorem 4.3]{GH} above aside from the quasiorthogonality estimate \cite[equation (4.4)]{GH}. This estimate is slightly misstated in \cite{GH}, where $h$ should be replaced by $\mathcal{Q}_s h$ and we verify this below when dealing with the term labeled $J_1$. } to show that
\begin{equation}\label{Tbgoaleq.eq}
\| \Theta^0_t 1 \|_{\C} \le C.
\end{equation}
As in \cite{GH}, we want to reduce the above estimate to one of the form
$$\dint_{R_Q} \Big| (\Theta_t 1)A_t^{\mu_Q} b_Q \Big|^2 \, \frac{dx \, dt}{t} \le C|Q|,$$
where $A_t^{\mu_Q}$ is an averaging operator adapted to $\mu_Q$ (and hence $Q$) we will introduce later and $R_Q$ is the Carleson region $Q \times (0, \ell(Q))$. The argument up until this reduction, namely \eqref{Tbgoal3eq.eq}, is almost exactly as in \cite{GH}. Define $\zeta(x,t) := \Theta_t 1(x)$, $\zeta^0(x,t) := \Theta_t^0 1(x)$, and $\zeta'(x,t) := \Theta_t' 1(x)$, where these objects make sense as elements of $L^2_{\loc}(\bb R^{n+1}_+)$ by Lemma \ref{AAAHKL3.11.lem} and Proposition \ref{ODforsevop.prop}. Consider the cut-off surfaces
$$F_1: = \big\{(x,t) \in \ree_+ : |\zeta^0(x,t)| \le \sqrt{\eta} |\zeta'(x,t)|\big\},$$
$$F_2: = \big\{(x,t) \in \ree_+ : |\zeta^0(x,t)| > \sqrt{\eta} |\zeta'(x,t)|\big\}.$$
We easily have that $\| \zeta^0 \|_{\C} \le \| \zeta^0\mathbbm{1}_{F_1} \|_{\C}  + \| \zeta^0 \mathbbm{1}_{F_2} \|_{\C}$. By definition of $F_1$, Proposition \ref{Carlesonerrorest.prop}, and the fact that $\eta < 1/(2C_1)$, we realize that
$$\| \zeta^0 \mathbbm{1}_{F_1} \|_{\C} \le \eta \| \zeta^1 \|_{\C} \le C_1\eta(1 + \|\zeta^0 \|_{\C}) \le \frac{1}{2}(1 + \|\zeta^0 \|_{\C}).$$
Consequently, $\|\zeta^0\|_{\C} \le 1 + 2 \| \zeta^0 \mathbbm{1}_{F_2} \|_{\C}$, and recall that we may work with truncated versions of each of $\zeta, \zeta^0, \zeta'$ so that all quantities are finite. Accordingly, we have reduced the proof of \eqref{Tbgoaleq.eq} to showing that
\begin{equation}\label{Tbgoal2eq.eq}
\| \zeta^0 \mathbbm{1}_{F_2} \|_{\C} \le C.
\end{equation}

By  \eqref{destbQ3.eq} and \eqref{destbQ4.eq} we have that
\begin{multline*}
\frac{1}{2} |\zeta^0| \le \Big| \zeta^0 \cdot\frac1{\mu_Q(Q)}\int_Q b^0_Q \, d\mu_Q \Big| \le  \Big| \zeta \cdot\frac1{\mu_Q(Q)}\int_Q b_Q \, d\mu_Q \Big| +  \Big| \zeta' \cdot\frac1{\mu_Q(Q)}\int_Q b'_Q \, d\mu_Q \Big| 
\\  \le \Big| \zeta \cdot\frac1{\mu_Q(Q)}\int_Q b_Q \, d\mu_Q \Big| + \frac{\eta}{2}|\zeta'|,
\end{multline*}
for every dyadic cube $Q \subset \rn$. Therefore, for every such $Q\subset \rn$, the estimates
$$\frac{1}{2} |\zeta^0|  \le \Big| \zeta \cdot\frac1{\mu_Q(Q)}\int_Q b_Q \, d\mu_Q \Big| + \frac{\sqrt{\eta}}{2}|\zeta^0|,\qquad\text{and}$$
$$|\zeta| \le |\zeta^0| + |\zeta'| \le (1 + \eta^{-\frac{1}{2}})|\zeta^0| \le 2 \eta^{-1/2}|\zeta^0|$$
hold in $F_2$. Combining the previous three estimates, we have that for $(x,t) \in F_2$ and every dyadic cube $Q$
\begin{equation}\label{tbdubtildeq.eq}
\frac{\sqrt\eta}{2}(1-\sqrt\eta) \frac{1}{2}|\zeta(x,t)| \le (1-\sqrt{\eta}) \frac{1}{2} |\zeta^0(x,t)| \le  \Big| \zeta \cdot\frac1{\mu_Q(Q)}\int_Q b_Q \, d\mu_Q \Big|.
\end{equation}

At this juncture, we make the observation that, in order to obtain (\ref{Tbgoal2eq.eq}), it suffices to show that for some $\alpha > 0$ chosen small enough, we have that 
\begin{equation}\label{tbdagprimeeq.eq}
\| \zeta^0\mathbbm{1}_{F_2}\mathbbm{1}_{\Gamma^\alpha_\nu}(\zeta)\|_{\C} \le C,
\end{equation}
with $C$ independent of $\nu$, where $\Gamma^\alpha_\nu$ is an arbitrary cone of aperture $\alpha$; that is, 
$$\Gamma^\alpha_\nu := \{z \in \mathbb{C}^{2}: |(z/|z|) - \nu)| < \alpha\},$$
for $\nu \in \mathbb{C}^{2}$ a unit vector. It is clear that if we establish \eqref{tbdagprimeeq.eq}, then \eqref{Tbgoal2eq.eq} follows by summing over a collection of cones covering $\mathbb{C}^{2}$.
In light of this, we fix such a cone $\Gamma^\alpha_\nu$ with $\alpha$ to be chosen. By \eqref{tbdubtildeq.eq} and the fact that $\eta < 1/4$ we have that  for each $(x,t) \in F_2$ with $\zeta(x,t) \in \Gamma^\alpha_\nu$ and every dyadic cube $Q \subset \rn$,
\begin{multline*}
\frac{\sqrt\eta}{8}\le \Big|\frac{\zeta(x,t)}{|\zeta(x,t)|} \cdot\frac1{\mu_Q(Q)}\int_{Q} b_Q \, d\mu \Big|\\ \le  \Big|\Big(\frac{\zeta(x,t)}{|\zeta(x,t)|} - \nu\Big) \cdot \frac1{\mu_Q(Q)}\int_{Q} b_Q \, d\mu \Big| +  \Big|\nu \cdot\frac1{\mu_Q(Q)}\int_{Q} b_Q \, d\mu \Big|
\\ \le C_0 \alpha + \Big|\nu \cdot \frac1{\mu_Q(Q)}\int_{Q} b_Q \, d\mu \Big|,
\end{multline*}
where in the last step, we used Schwarz's inequality, the fact that 
$$1/C_0 \le d\mu/dx = \phi_Q \le 1 \text{ on } Q,$$
and \eqref{destbQ1.eq}.
Since $\alpha$ is at our disposal, we may choose $\alpha < \tfrac{\sqrt{\eta}}{16C_0}$, so that 
\begin{equation}\label{tbhearteq.eq}
\frac{\sqrt{\eta}}{16} = : \theta  \le \Big|\nu \cdot \frac1{\mu_Q(Q)}\int_{Q} b_Q \, d\mu \Big|.
\end{equation}
Next, we observe that in order to obtain \eqref{tbhearteq.eq} we needed $(x,t) \in F_2$ with $\zeta(x,t) \in \Gamma^\alpha_\nu$. This means that \eqref{tbhearteq.eq} holds whenever
\begin{equation*} 
\dint_{R_Q}|\zeta^0(x,t)|^2 \mathbbm{1}_{F_2}(x,t) \mathbbm{1}_{\Gamma^\alpha_\nu}(\zeta(x,t)) \, \frac{dx \, dt}{t} \neq 0.
\end{equation*}
Consequently, when proving \eqref{tbdagprimeeq.eq} we can always assume that \eqref{tbhearteq.eq} holds.

Now, fix any dyadic cube $Q$ such that \eqref{tbhearteq.eq} holds and, following \cite{GH}, use a stopping time procedure to extract a family $\F=\{Q_j\}$ of non-overlapping dyadic subcubes of $Q$ which are maximal with respect to the property that at least one of the following conditions holds:
\begin{gather*}
\frac1{\mu_Q(Q_j)}\int_{Q_j} |b_Q|\, d\mu_Q  > \frac{\theta}{4\alpha}  \qquad \text{(type $I$)}
\\ \Big|\nu \cdot \frac1{\mu_Q(Q_j)}\int_{Q_j} b_Q \, d\mu_Q \Big|  \le \frac{\theta}{2}  \qquad\text{(type $II$)}.
\end{gather*}
If some $Q_j$ happens to satisfy both the type $I$ and type $II$ conditions we (arbitrarily) assign it to be of type $II$. We will write $Q_j \in \F_I$ or $Q_j \in \F_{II}$ to mean that a cube is of type $I$ or of type $II$ respectively.
This stopping time argument produces an `ample sawtooth' with desirable bounds in the following sense.
\begin{claim}[Ample sawtooth]\label{TbSawtooth.cl}
There exists $\beta > 0$, uniform in $Q$, such that
\begin{equation}\label{TbAmpleeq.eq}
\sum_{Q_j \in \F} |Q_j| \le (1-\beta)|Q|,
\end{equation}
provided that $\alpha > 0$ is  small enough (depending on allowable constants). Moreover,
\begin{equation}\label{TbBoundeq.eq}
|\zeta(x,t)|^2\mathbbm{1}_{\Gamma^\alpha}(\zeta(x,t)) \le C_\theta|\zeta(x,t)A_t^{\mu_Q}b_Q(x)|^2, \quad \text{for } (x,t) \in E_Q^*,
\end{equation}
where $E_Q^* : = R_Q \setminus(\cup_{Q_j \in \F} R_{Q_j})$. Here $A_t^{\mu_Q}$ is the `dyadic averaging operator adapted to the measure $\mu_Q$', that is, $A_t^\mu f(x) = \frac1{\mu_Q(Q(x,t))}\int_{Q(x,t)} f \, d\mu_Q,$ where $Q(x,t)$ denotes the smallest dyadic cube, of side length at least $t$, that contains $x$. 
\end{claim}

We postpone the proof of the claim for a bit. The ampleness condition   \eqref{TbAmpleeq.eq}   allows us to use the  ``John-Nirenberg lemma for Carleson measures'' to replace $R_Q$ in the definition of $\| \cdot \|_{\C}$ by $E^*_Q$. This is done via an induction argument; see for instance, \cite[Lemma 1.37]{H1}. Thus, we have by \eqref{TbBoundeq.eq} that
\begin{multline*}
\| \zeta^0\mathbbm{1}_{F_2}\mathbbm{1}_{\Gamma^\alpha_\nu}(\zeta) \|_{\C}  \lesssim_{\beta} \sup_{Q} \frac{1}{|Q|} \dint_{E_Q^*}|\zeta^0(x,t)|^2 \mathbbm{1}_{F_2}(x,t)\mathbbm{1}_{\Gamma^\alpha_\nu}(\zeta(x,t))\, \frac{dx\,dt}{t}
\\  \lesssim  \sup_{Q} \frac{1}{|Q|} \dint_{R_Q}|\zeta(x,t)A_t^{\mu_Q}b_Q(x)|^2 \, \frac{dx \, dt}{t},
\end{multline*}
where we used that $|\zeta^0| \le |\zeta|$ in the first line and replaced $E_Q^*$ by the larger set $R_Q$ after using \eqref{TbBoundeq.eq} in the second line. As we had reduced the proof of the theorem to showing the estimate \eqref{tbdagprimeeq.eq}, it is enough to show that
\begin{equation}\label{Tbgoal3eq.eq} 
\sup_{Q} \frac{1}{|Q|} \dint_{R_Q}|\zeta(x,t)A_t^{\mu_Q}b_Q(x)|^2 \, \frac{dx \, dt}{t} \le C.
\end{equation}
To this end, we fix a dyadic cube $Q$ and write
\begin{equation*}
\zeta A_t^{\mu_Q}b_Q = [(\Theta_t 1)A_t^\mu - \Theta_t]b_Q + \Theta_t b_Q  =:R_t b_Q + \Theta_tb_Q = I + II.
\end{equation*}
First we handle term $II$, which is (almost) good by design. We write
$$II = \Theta_t b_Q = \hat \Theta_t \hat b_Q - \Theta_t^{(a)}b_Q^{(a)} =: II_1 + II_2.$$
By \eqref{destbQ2.eq}, the contribution from the term $II_1$ in \eqref{Tbgoal3eq.eq}  is controlled by $C_0$. Moreover, by Proposition \ref{Carlesonerrorest.prop} we have  that
\begin{multline*}
\dint_{R_Q} |\Theta_t^{(a)}b_Q^{(a)}|^2 \, \frac{dx \, dt}{t}  \le C_1\lVert b_Q^{(a)}\rVert^2_{L^2(\rn)}(1 +\| \Theta^0_t 1 \|_{\C})
\\ \le C_1C_0|Q|\, \|B_1\|_{n}^2(1 +\| \Theta^0_t 1 \|_{\C}),
\end{multline*}
so that the contribution of $II_2$ can be hidden in \eqref{Tbgoaleq.eq}, provided that $\|B_1\|_{n}$ is sufficiently small (depending on $\eta$, $\alpha$). Here, we used that $b_Q^{(a)}(y) =|Q| B_1 F_Q(y,0)$, so that
$$
\lVert b_Q^{(a)}\rVert^2_{L^2(\rn)} = \int_{\rn} |Q|^2|B_1F_Q(\cdot,0)|^2 \le \|B_1\|_{n}^2 |Q|^2 \int_{\rn} |\nabla F_Q(\cdot,0)|^2 \le C_0 \|B_1\|_{n}^2\,|Q|.
$$

It remains to obtain a desirable bound for $I$. Let $\{\m Q_s\}_{s>0}$ be a CLP family (see Definition \ref{CLP.def}). By a standard orthogonality argument and \eqref{destbQ1.eq}, it is enough to show that for some $\beta_0 > 0$ and all $t \in (0,\ell(Q))$, the estimate  
\begin{equation}\label{Redeq.eq}
\int_{Q} |R_t\m Q_s^2 h|^2\lesssim \min\Big(\frac{s}{t}, \frac{t}{s}\Big)^{\beta_0} \int _{\rn}|\m Q_s h|^2
\end{equation}
holds for all $h \in H \times L^2(\rn)$. 

We remind the reader that $H:= \{h': h' = \nabla F, F \in Y^{1,2}(\rn)\}$ and that $b_Q \in H \times L^2(\rn)$. Before proving \eqref{Redeq.eq}, we make a small technical point. Having fixed $Q$, we let $\tilde \mu_Q$ be a measure on $\rn$ defined by $\tilde\mu_Q := \mu_Q|_{Q} + \frac{1}{\tilde{C_0}}dx|_{\rn \setminus Q}$, and set $E_t = A_t^{\tilde \mu_Q}$. Notice that for $(x,t) \in Q \times (0,\ell(Q))$, $A_t^{\tilde \mu_Q}$ acts exactly as $A_t^{\mu_Q}$. Thus, in order to prove \eqref{Redeq.eq}, we may replace $R_t$ by $\widetilde R_t$, where $\widetilde{R}_t : = [(\Theta_t 1)E_t-\Theta_t]$.   Notice that we may apply Lemma \ref{AAAHKL3.11.lem} to $\Theta_t$, since $\Theta_t$ has good off-diagonal decay (see Proposition \ref{ODforsevop.prop}) and satisfies uniform $L^2$ bounds on slices (see Proposition \ref{L2avgandsliceboundsperturb.prop}). Thus, $(\Theta_t 1)$ is well defined as an element of $L^2_{\loc}$ and, since $E_t$ is a  self-adjoint averaging operator, 
we have that
\begin{equation}\label{TbDiamondeq.eq}
\sup_{t > 0} \| (\Theta_t 1)E_t\|_{L^2 \to L^2} \le C. 
\end{equation}
We break \eqref{Redeq.eq} into cases.

{\bf Case 1: $t \le s$.} In this case, we see by \eqref{TbDiamondeq.eq} and properties of $\Theta_t$ that $\widetilde{R}_t 1 = 0$, $\|\widetilde R_t\|_{L^2 \to L^2} \le C$ and $\widetilde R_t$ has good off-diagonal decay. Hence, it follows from   Lemma \ref{AAAHKL3.5.lem} that
$$\|\widetilde R_t\m Q_s^2 h\|_{L^2(\rn)} \lesssim t \| \nabla\m Q_s^2 h\|_{L^2(\rn)} \lesssim \tfrac{t}{s}\|s\nabla\m Q_s\m Q_s h\|_{L^2(\rn)} \lesssim  \tfrac{t}{s}\|\m Q_s h\|_{L^2(\rn)},$$
which shows \eqref{Redeq.eq} with $\beta_0 = 2$ in this case.

{\bf Case 2: $t > s$.} In this case, we break $\widetilde R_t$ into its two separate operators. One can verify that $\|E_t\m Q_s\|_{L^2 \to L^2} \lesssim (\tfrac{s}{t})^{\gamma}$ for some $\gamma > 0$.  Since $E_t$ is a projection operator, we have that $E_t = E_t^2$ and hence by \eqref{TbDiamondeq.eq}, we see that
\begin{equation*}
\| (\Theta_t 1)E_t\m Q_s^2h\|_2 = \| (\Theta_t 1)E_t [E_t\m Q_s^2h]\|_{2} \lesssim \|E_t\m Q_s^2h\|_{2} \lesssim (\tfrac{s}{t})^{\gamma}\|\m Q_sh\|_{2},
\end{equation*}
which shows that the contribution of $(\Theta_t 1)E_t\m Q_s^2$ to \eqref{Redeq.eq} when $t > s$ is as desired with $\beta_0 = 2\gamma$.

We are left with handling $\Theta_t\m Q_s^2 h$. Since $h = (h', h^0) \in H \times L^2(\rn)$, we write $h = (\nabla_\parallel F, h^0)$, with $F \in Y^{1,2}(\rn)$ (note $\nabla_\parallel = \nabla$ here). Then we may write 
\begin{multline*} 
\Theta_t\m Q_s h  = \Theta_t^0\m Q_s^2h^0 + \Theta'_t\m Q_s^2 \nabla_\parallel F
\\  = \Theta_t^0\m Q_s^2h^0 + [\Theta'_t\m Q_s^2 \nabla_\parallel F + \Theta_t^{(a)}B_1\m Q_s^2F] -  \Theta_t^{(a)}B_1\m Q_s^2F
\\  = J_1 + J_2  + J_3.
\end{multline*}
To handle $J_1$, we write $\m Q_s = s \div_\parallel s\nabla_\parallel e^{s^2 \Delta}$, so that 
\begin{equation*} 
J_1 = \Theta_t^0\m Q_s^2h^0  = t^{m}(\partial_t)^{m+1}\s^{\cL}_t\m Q_s\m Q_s h^0=  \tfrac{s}{t}t^{m+1}(\partial_t)^{m+1}\s^{\cL}_t \div_\parallel s\nabla_\parallel e^{s^2 \Delta}\m Q_s h^0.
\end{equation*}
Note that  by \eqref{eq.sliceest} we have that $t^{m+1}(\partial_t)^{m+1}\s^{\cL}_t \div_\parallel$ and $s\nabla_\parallel e^{s^2 \Delta}$ are bounded operators on $L^2(\bb R^n)$. Therefore, we have that $\| \Theta_t^0\m Q_s^2h^0\|_{2} \lesssim \frac{s}{t}\|\m Q_sh^0\|_{2}$, and the contribution of $J_1$ to \eqref{Redeq.eq} when $t > s$ is as desired with $\beta_0 = 2$.

For the term $J_2$, first we use Proposition \ref{CLPW12.prop} to justify that there exists $g\in L^2(\bb R^n)$ such that $\m Q_s F = I_1 g$, where $I_1 = (-\Delta)^{-1/2}$ is the Riesz potential of order $1$, and satisfying $\|g\|_2 \approx \|\nabla_\parallel\m Q_s F \|_2 = \|\m Q_s \nabla_\parallel F \|_2 = \|\m Q_s h' \|_2$ (every $F \in Y^{1,2}(\rn)$ arises as the Riesz potential of a function $g$ in $L^2(\rn)$). Then, we may use integration by parts on slices (Proposition \ref{IBPonslicesprop.prop}) to compute that
\begin{multline*}
J_2 = t^{m}(\partial_t)^m \s^{\cL}_t\big({B_2}_\parallel \nabla_\parallel\m Q_s^2 F\big) + t^{m}(\partial_t)^m (\s^{\cL}_t \nabla)\tilde{A}\nabla_\parallel\m Q_s^2 F  +  t^{m}(\partial_t)^m (\s^{\cL}_t \nabla)B_1\m Q_s^2F
\\  = - t^{m}(\partial_t)^{m +1}(\s^{\cL}_t \nabla)\vec A_{\cdot,n+1}\m Q_sI_1 g + t^m(\partial_t)^{m+1}\s^{\cL}_t{B_2}_\perp \m Q_sI_1 g
  = J_{2,1} + J_{2,2}.
\end{multline*}
Since $\|s^{-1}\m Q_sI_1 \|_{L^2 \to L^2} \le C$ and $t^{m + 1}(\partial_t)^{m +1}(\s^{\cL}_t \nabla):L^2 \to L^2$, we obtain that the contribution of $J_{2,1}$ to \eqref{Redeq.eq} when $t > s$ is as desired with $\beta_0 = 2$. 
Similarly, $t^m(\partial_t)^{m+1}\s^{\cL}_t {B_2}_\perp : L^2 \to L^2$, so that the contribution of $J_{2,2}$ to \eqref{Redeq.eq} when $t > s$ is as desired with $\beta_0 = 2$. 

We are left with controlling the contribution of
\begin{equation*}
J_3 = \Theta_t^{(a)}B_1\m Q_s^2F = t^m\partial_t^m (\s^{\cL}_t\nabla)B_1\m Q_s F = \Theta_{t,m} B_1 I_1 g,
\end{equation*}
where $F = I_1 g$, $F \in Y^{1,2}$ and $g \in L^2$ with $\| g\|_2 \approx \|\nabla_\parallel F\|_2$ . By Proposition \ref{QOforSnabBI.lem}, for all $s < t$ we have that
$$\|\Theta_t^{(a)}B_1\m Q_s^2F\|_{L^2(\rn)} \lesssim \Big(\frac{s}{t} \Big)^\gamma\|\m Q_s g\|_{L^2(\rn)}.$$
Then we may control this term in \eqref{Redeq.eq} with $g$ in place of $h = \nabla_\parallel F$, which is sufficient as $\| g\|_2 \lesssim \|\nabla_\parallel F\|_2$.

The proof of the theorem is finished modulo the  

\begin{proof}[Proof of Claim \ref{TbSawtooth.cl}]
	We first verify \eqref{TbBoundeq.eq}. Observe, by the maximality of the family $Q_j$, that for any dyadic subcube $Q'$ of $Q$  which is not contained in any $Q_j$, we have the inequalities opposite to the type $I$ and type $II$ inequalities, with $Q'$ in place of $Q_j$. Thus,
	$$\frac{\theta}{2} \le \Big|\nu \cdot A_t^{\mu_Q} b_Q(x)\Big|\quad
	\text{ and }\quad
	|A_t^{\mu_Q} b_Q(x)| \le \frac{\theta}{4\alpha}$$
	for all $(x,t) \in E_Q^*$. It follows that if $z \in \Gamma^\alpha_{\nu}$ and $(x,t) \in E_Q^*$, we have the bound
	\begin{multline*}
	\frac{\theta}{2} \le \Big|\nu \cdot A_t^{\mu_Q} b_Q(x)\Big| 
	\le  \Big|(z/|z|) \cdot A_t^{\mu_Q} b_Q(x)\Big|  +  \Big|(z/|z| -\nu) \cdot A_t^{\mu_Q} b_Q(x)\Big| 
	\\  \le  \Big|(z/|z|) \cdot A_t^{\mu_Q} b_Q(x)\Big| + \frac{\theta}{4},
	\end{multline*}
	where we used the definition of $\Gamma^\alpha_{\nu}$ in the last line. The above estimate yields \eqref{TbBoundeq.eq} with $C_\theta = (\tfrac{4}{\theta})^2$ by setting $z = \zeta(x,t)$.
	
	Now we establish \eqref{TbAmpleeq.eq}. Set $E: = Q \setminus( \cup_{Q_j \in \F} Q_j)$ and $B_I := \cup_{Q_j \in \F_I} Q_j$. By definition of $\F_I$ and the fact that $1/C_0 \le d\mu/dx \le 1$ on $Q$, we have that $B_I\subset \big\{\n M (b_Q) >\frac{\theta}{4C_0\alpha}\big\}$ where $\n M$ is the uncentered Hardy-Littlewood maximal function on $\rn$ (taken over cubes). The weak-type $(2,2)$ inequality for the Hardy-Littlewood maximal function and \eqref{destbQ1.eq} yield the estimate
	\begin{equation*} 
	|B_I| \le CC_0^2\Big(\frac{\alpha}{\theta}\Big)^2 \int_{\rn}|b_Q|^2 \le CC_0^3\Big(\frac{\alpha}{\theta}\Big)^2|Q|.
	\end{equation*}
	From this estimate, \eqref{destbQ1.eq},  \eqref{tbhearteq.eq}, the definition of type $II$ cubes, and H\"older's inequality we obtain
	\begin{multline*}
	\theta \mu_Q(Q) \le \Big| \nu \cdot \int_Q b_Q \, d\mu_Q\Big|
	\\ \le \Big| \nu \cdot \int_E b_Q \, d\mu_Q\Big| + \int_{B_I} |b_Q| \, d\mu + \sum_{Q_j \in \F_{II}} \Big| \nu \cdot \int_{Q_j} b_Q \, d\mu_Q\Big|
	\\  \le |E|^{1/2} \| b_Q\|_{L^2(\rn)} + |B_I|^{1/2} \| b_Q\|_{L^2(\rn)} + \frac{\theta}{2}\sum_{Q_j \in \F_{II}} \mu_Q(Q_j)
	\\  \le C|E|^{1/2} |Q|^{1/2} + C_{\theta}\alpha |Q| + \frac{\theta}{2}\mu_Q(Q).
	\end{multline*}
	Choosing $\alpha > 0$ small enough and using the fact that $(1/2)^n |Q| \le \mu_Q(Q) \le |Q|$, the above estimate implies that $|Q| \le C_\theta |E|$, which yields the claim with $\beta = 1/C_\theta$.\end{proof}

Thus we conclude the proof of Theorem \ref{Tbargument.thrm}.\end{proof}

\section{Control of slices via square function estimates}

We are able to use the square function estimate obtained in the previous section to immediately improve our $L^2\ra L^2$ boundedness results of $t-$derivatives of the single layer potential. More precisely, in the following lemma, we extend 
estimate \eqref{eq.sliceest} (previously valid for $m\geq 2$), to the case $m=1$, given sufficient smallness of
$\max\{\|B_1\|_n, \|B_2\|_n\}$.
\begin{lemma}[Stronger $L^2\ra L^2$ estimate]\label{lm.k=1} The estimate
\begin{equation*}
		\Vert t\nabla\partial_t \s_t^{\cL}f\Vert_{L^2(\bb R^n)}\lesssim\Vert f\Vert_{L^2(\bb R^n)},
\end{equation*}
holds, provided that $\max\{\|B_1\|_n, \|B_2\|_n < \varepsilon_0$ and $\varepsilon_0>0$ is small enough so that \eqref{eq.sqfnest} holds for $m = n + 10$.
\end{lemma}

We may use Lemma \ref{lm.k=1} to obtain the ``travel down'' procedure for $\nabla\m S^\cL\nabla$.

\begin{lemma}[$L^2\ra L^2$ estimates for $S_t\nabla$]\label{lm.l2slicesgrad} The following statements are true.
\begin{enumerate}[i)]
\item\label{item.traveldownsgrad2} For each ${\bf f}\in L^2(\bb R^n,\bb C^{n+1})$ and each $t\neq0$ we have that
\begin{equation}\label{eq.l2slicesgrad}
\Vert t^k\partial_t^k(\m S^\cL_t\nabla){\bf f}\Vert_2\lesssim\Vert{\bf f}\Vert_2,\qquad k\geq1,
\end{equation}
\begin{equation}\label{eq.l2slicegradsgrad}
\Vert t^k\partial_t^{k-1}\nabla(\m S^\cL_t\nabla){\bf f}\Vert_2\lesssim\Vert{\bf f}\Vert_2,\qquad k\geq2,
\end{equation}
provided that $\max\{\|B_1\|_n, \|B_2\|_n < \varepsilon_0$ is small. Therefore, for each $m>k\geq2$,
\begin{equation}\label{eq.traveldownsgrad2}
\vertiii{t^k\partial_t^{k-1}\nabla(\m S^\cL_t\nabla){\bf f}}\lesssim_m\vertiii{t^m\partial_t^m(\m S^\cL_t\nabla){\bf f}}+\Vert{\bf f}\Vert_2,
\end{equation}
provided that $\max\{\|B_1\|_{L^n(\rn)}, \|B_2\|_{L^n(\rn)}\}$ is small.
\item\label{item.traveldownsgrad1} The estimate (\ref{eq.traveldownsgrad2}) holds for $k=1$ if the operator $\nabla$ acting on $(\m S_t^\cL\nabla)$ is replaced by $\partial_t$.
\end{enumerate} 
\end{lemma}

We proceed with the   

\noindent\emph{Proof of Theorem \ref{thm.suponslices}.} Let ${\bf h}\in C_c^{\infty}(\bb R^n)^{n+1}$ and fix $\tau>0$. Notice that by Lemma \ref{ContSlices.lem}, the pairing $({\bf h},\Tr_t\nabla u)_{2,2}$ is meaningful. Let $R\gg\tau$, $\psi\in C_c^{\infty}(\bb R)$ satisfy $\psi\equiv1$ on $[\tau,R]$, $\psi\equiv0$ on $[2R,\infty)$, $|\psi|\leq1$ and $|\psi'|\leq\frac2R$. We have the following estimates:
\begin{multline}\label{eq.rtoinfty}
|I|:=\Big|\int_{\bb R^n}{\bf h}\cdot\int_R^{2R}\psi'\nabla u\Big|\leq\int_{\bb R^n}\int_R^{2R}|{\bf h}||\psi'||\nabla u|\\ \leq\frac2{\sqrt R}\Vert{\bf h}\Vert_{L^2(\bb R^n)}\Vert\nabla u\Vert_{L^2(\bb R^{n+1}_+)}\longrightarrow0\quad\text{as }R\ra\infty,
\end{multline}
\begin{multline}\label{eq.tcloseto0}
\Big|\int_{\bb R^n}{\bf h}\cdot t\Tr_{t+\tau}\nabla\partial_t u\Big|\leq\Vert{\bf h}\Vert_{L^2(\bb R^n)}\frac t{t+\tau}\Vert(t+\tau)\Tr_{t+\tau}\nabla\partial_tu\Vert_{L^2(\bb R^n)}\\ \leq\frac t{\tau}\Vert{\bf h}\Vert_{L^2(\bb R^n)}\Vert\nabla u\Vert_{L^2(\bb R^{n+1}_+)}\longrightarrow0\quad\text{as }t\searrow0,
\end{multline}
\begin{multline}\label{eq.rtoinfty2}
|II|:=\Big|\int_{\bb R^n}\int_{R-\tau}^{2R-\tau}{\bf h}\cdot t\psi'(t+\tau)\Tr_{t+\tau}\partial_t\nabla u\,dt\Big|\leq2\dashint_R^{2R}\int_{\bb R^n}t|{\bf h}||\partial_t\nabla u|\,dt\\ \leq2\Vert{\bf h}\Vert_{L^2(\bb R^n)}\sup_{t\in(R,2R)}\Vert t\Tr_t\partial_t\nabla u\Vert_{L^2(\bb R^n)} \lesssim\Vert{\bf h}\Vert_{L^2(\bb R^n)}\Vert\nabla u\Vert_{L^2(\bb R^{n+1}_{R/2})}\longrightarrow0\quad\text{as }R\ra\infty,
\end{multline}
where in (\ref{eq.tcloseto0}) we used (\ref{eq.supl2gradshift}), and in (\ref{eq.rtoinfty2}) we used (\ref{eq.l2gradslice}) and the absolute continuity of the integral. We now perform two integration by parts in the following calculation, recalling that $\psi(2R)=0$ so that the arising boundary terms  vanish.
\begin{multline*}
\int_{\bb R^n}{\bf h}\cdot\Tr_{\tau}\nabla u=\int_{\bb R^n}{\bf h}\cdot\psi(\tau)\Tr_{\tau}\nabla u-\int_{\bb R^n}{\bf h}\cdot\psi(2R)\Tr_{2R}\nabla u\nonumber
\\ =-\int_{\bb R^n}{\bf h}\cdot\int_{\tau}^{2R}\psi\partial_t\nabla u~-~\int_{\bb R^n}{\bf h} \cdot\int_R^{2R}\psi' \nabla u \nonumber
\\ =\int_{\bb R^n}\int_{0}^{2R-\tau}{\bf h}\cdot t\Tr_t\n T^{\tau}\psi\partial_t^2\nabla\n T^{\tau}u\,dt+\int_{\bb R^n}\int_{R-\tau}^{2R-\tau}{\bf h}\cdot t\Tr_t\n T^{\tau}\psi'\partial_t\nabla\n T^{\tau}u\,dt-I\nonumber
\\ =\int_{\bb R^n}\int_{0}^{2R-\tau}{\bf h}\cdot t\Tr_t\n T^{\tau}\psi\partial_t^2\nabla\n T^{\tau}u\,dt+II-I 
\end{multline*}
where in the third equality we used (\ref{eq.tcloseto0}) already. Note that the terms $I, II$ drop to $0$ as $R\ra\infty$ by the estimates (\ref{eq.rtoinfty}) and (\ref{eq.rtoinfty2}).  For technical reasons, let us integrate by parts one more time.  The boundary term that is introduced is again controlled as in (\ref{eq.tcloseto0}) and (\ref{eq.rtoinfty2}) because we may apply the results of Proposition \ref{prop.tder} to $\partial_t^2\n T^{\tau}u$. Hence we have that
\begin{multline}
\int_{\bb R^n}\int_{0}^{2R-\tau}{\bf h}\cdot t\Tr_t\n T^{\tau}\psi\partial_t^2\nabla\n T^{\tau}u\,dt\\ =-\frac12\int_{\bb R^n}\int_{0}^{2R-\tau}{\bf h}\cdot t^2\Tr_t\n T^{\tau}\psi\partial_t^3\nabla\n T^{\tau}u\,dt+III,
\end{multline}
where $|III|\ra0$ as $R\ra\infty$. Intuitively, we would like to introduce Green's formula at this point, but we want the ``input" in the layer potentials to still depend on $t$ for when we later dualize to control our integral by square function estimates.  Let us now do a change of variables $t\mapsto2t$, and carefully track the use of the chain rule:
\begin{multline*}
\frac12\int_{\bb R^n}\int_{0}^{2R-\tau}{\bf h}\cdot t^2\Tr_t\n T^{\tau}\psi\partial_t^3\nabla\n T^{\tau}u\,dt
\\= 4\int_{\bb R^n}\int_{0}^{R-\frac{\tau}2}{\bf h}\cdot t^2\n T^{\tau}\psi(2t)\partial_{2t}^3\nabla_{x,2t}\n T^{\tau}u(\cdot,2t)\,dt\\ =\frac12\int_{\bb R^n}\Big[\int_0^{R-\frac{\tau}2}\vec{h_{\|}}\cdot t^2\n T^{\tau}\psi(2t)\partial_t^3\nabla_{\|}\n T^{\tau}u(\cdot,2t)\,dt\\+\frac12h_{\perp}\int_0^{R-\frac{\tau}2}t^2\n T^{\tau}\psi(2t)\partial_t^4\n T^{\tau}u(x,2t)\,dt\Big].
\end{multline*} 
We now consider  $s\in\bb R$ and write $2t=t+s|_{s=t}$. If $F$ is a differentiable function in $t$, the chain rule tells us that 
$\partial_t F(t+s)=\partial_s F(t+s)$.  By this change of variables, and the above identity, we compute that
\begin{multline*}
\frac12\int_{\bb R^n}\int_{0}^{2R-\tau}{\bf h}\cdot t^2\Tr_t\n T^{\tau}\psi\partial_t^3\nabla\n T^{\tau}u\,dt\nonumber
\\ =4\int_0^{R-\frac{\tau}2}t^2\n T^{t+\tau}\psi(t)\Big[\int_{\bb R^n}{\bf h}\cdot\Tr_t\nabla_{x,t}D_{n+1}^3\n T^s\n T^{\tau}u(x,t)\Big]_{s=t}\,dt.\nonumber
\\ =4\int_0^{R-\frac{\tau}2}t^2\n T^{t+\tau}\psi(t)\Big[\int_{\bb R^n}{\bf h}\cdot\Tr_t\nabla D_{n+1}\n T^{\tau}\Big(D_{n+1}^2\n T^su\Big)\Big]_{s=t}\,dt. 
\end{multline*}

We now apply Green's formula, Theorem \ref{thm.greenformula} \ref{item.greenformula}). The function $v:=D_{n+1}^2\n T^su$ belongs to $W^{1,2}(\bb R^{n+1}_+)\subset\Y$ and solves $\m L v=0$ in $\bb R^{n+1}_+$ in the weak sense. Therefore the identity $v=-\m D^{\cL,+}(\Tr_0v)+\m S^{\cL}(\partial_{\nu}^{\cL,+}v)$ holds in $\Y$, for any $s>0$. But by the results of Proposition \ref{prop.tder}, for each $t>0$ we have the identity
\[
\Tr_t\nabla D_{n+1}\n T^{\tau}v=\Tr_t\nabla D_{n+1}\n T^{\tau}\big(-\m D^{\cL,+}(\Tr_0v)+\m S^{\cL}(\partial_{\nu}^{\cL,+}v)\big)
\]
in $L^2(\bb R^n)$, for any $s>0$ and $t>0$. As such, per our calculations we have the identity
\begin{multline*}
\frac12\int_{\bb R^n}\int_{0}^{2R-\tau}{\bf h}\cdot t^2\Tr_t\n T^{\tau}\psi\partial_t^3\nabla\n T^{\tau}u\,dt\\ =-4\int_0^{R-\frac{\tau}2}t^2\n T^{t+\tau}\psi(t)\Big[\int_{\bb R^n}{\bf h}\cdot\Tr_t\nabla D_{n+1}\n T^{\tau}\m D^{\cL,+}(\Tr_0v)\Big]_{s=t}\,dt\\ + 4\int_0^{R-\frac{\tau}2}t^2\n T^{t+\tau}\psi(t)\Big[\int_{\bb R^n}{\bf h}\cdot\Tr_t\nabla D_{n+1}\n T^{\tau}\m S^{\cL}(\partial_{\nu}^{\cL,+}v)\Big]_{s=t}\,dt=IV+V. 
\end{multline*}

Now we make use of the adjoint relations (\ref{eq.slgradadj}), (\ref{eq.dlgradadj}) and (\ref{eq.tdermovetder}) to dualize $IV$ and $V$. Indeed, we see that
\begin{multline*} 
\int_{\bb R^n}{\bf h}\cdot\Tr_t\nabla D_{n+1}\n T^{\tau}\m D^{\cL,+}(\Tr_0v)=\overline{\Big(D_{n+1}\partial_{\nu,-t-\tau}^{\cL^*,-}(\m S^{\cL^*}\nabla){\overline{\bf h}} ,\Tr_0v\Big)_{2,2}}\\ =\overline{\Big(D_{n+1}e_{n+1}\cdot\Tr_{-t-\tau}\Big[A^*\nabla+\overline{B_2}\Big](\m S^{\cL^*}\nabla){\overline{\bf h}},\Tr_sD_{n+1}^2u\Big)_{2,2}}\,,
\end{multline*}
\begin{multline*} 
\int_{\bb R^n}{\bf h}\cdot\Tr_t\nabla D_{n+1}\n T^{\tau}\m S^{\cL}(\partial_{\nu}^{\cL,+}v)=\overline{\Big(\Tr_{-t-\tau}D_{n+1}(\m S^{\cL^*}\nabla){\overline{\bf h}},\partial_{\nu}^{\cL,+}v\Big)_{2,2}} \\ =\overline{\Big(\Tr_{-t-\tau}D_{n+1}(\m S^{\cL^*}\nabla){\overline{\bf h}},-e_{n+1}\cdot\Tr_s\Big[A\nabla +B_1 \Big]D_{n+1}^2u\Big)_{2,2}}.
\end{multline*}
Therefore, using the Cauchy-Schwartz inequality, we estimate that
\begin{multline}\label{eq.suponslicescalcsl2}
|IV|\leq4\int_0^{R-\frac{\tau}2}t^2\int_{\bb R^n}\Big|\Tr_{-t-\tau}D_{n+1}\Big[A^*\nabla+\overline{B_2}\Big](\m S^{\cL^*}\nabla){\overline{\bf h}}\Big|\,\Big|\Tr_tD_{n+1}^2u\Big|\,dt\\ \lesssim\||t^2\partial_t\nabla\m (S^{\cL^*}\nabla)\overline{{\bf h}}\||_-\||t\partial_t^2u\|| \lesssim\Vert{\bf h}\Vert_2\||t\partial_t^2u\||,
\end{multline}
\begin{multline}\label{eq.suponslicescalcdl2}
|V|\leq4\int_0^{R-\frac{\tau}2}t^2\int_{\bb R^n}\Big|\Tr_{-t-\tau}D_{n+1}(\m S^{\cL^*}\nabla){\overline{\bf h}}\Big|\,\Big|\Tr_t\Big[A\nabla +B_1\Big]D_{n+1}^2u\Big|\,dt\\\lesssim\||t\partial_t\m (S^{\cL^*}\nabla)\overline{{\bf h}}\||_-\||t^2\partial_t^2\nabla u\|| \lesssim\Vert{\bf h}\Vert_2\||t^2\partial_t^2\nabla u\||,
\end{multline}
where we used the square function estimate (\ref{eq.sqfnest}) and the ``travel-down'' procedure (\ref{eq.traveldownsgrad2}). Now send $R\ra\infty$, which sends $|I|,|II|,|III|\ra0$. By the bounds (\ref{eq.suponslicescalcsl2}), (\ref{eq.suponslicescalcdl2}), and Lemma \ref{lm.travelup}, the desired bound for the gradient follows.

To obtain the bound for the $L^{\frac{2n}{n-2}}(\rn)$ norm, we use   Lemma \ref{ContSlices.lem} to ensure that at each horizontal slice,  the $L^{\frac{2n}{n-2}}(\rn)$ norm of a  $\yotp$ solution is finite. Then  we may apply the Sobolev embedding, whence the desired result follows.\hfill{$\square$}

The method of proof of Theorem \ref{thm.suponslices} is robust, in the sense that we may loosen the condition that $u\in Y^{1,2}(\bb R^{n+1}_+)$, provided that $u$ is such that the square function  in the right-hand side of (\ref{eq.suponslices}) is finite, and that the gradient of $u$ decays to $0$ in the sense of distributions for large $t$. More precisely, we have

\begin{theorem}[A more general $\Tr<S$ result]\label{thm.suponslices2a} Suppose that $u\in W^{1,2}_{\loc}(\bb R^{n+1}_+)$, $\cL u=0$ in $\bb R^{n+1}_+$ in the weak sense, and $\nabla u(\cdot,t)$ converges to $0$ in the sense of distributions as $t\ra\infty$ (we refer to this last condition as the decaying condition). Furthermore, assume that $\||t\nabla D_{n+1}u\||<\infty$. Then, for every $\tau>0$, the following statements are true.
	\begin{enumerate}[i)]
	\item If $\cL 1\neq 0$ in $\reu$, then
	\begin{gather}\label{eq.suponslices1a}
	\| \tr_\tau u\|_{L^{\frac{2n}{n-2}}(\rn)}+\Vert\Tr_{\tau}\nabla u\Vert_{L^2(\bb R^n)}\lesssim \int_\tau^\infty \int_\rn t|\dno^2 u|^2\, dxdt \lesssim \||tD_{n+1}^2 u\||.
	\end{gather}
	\item If $\cL 1=0$ in $\reu$, then there exists a constant $c\in \CC$ such that $v:=u-c$ (which is again a solution) satisfies estimate (\ref{eq.suponslices1a}).
	\end{enumerate}
\end{theorem}

The proof of this theorem is omitted as it is very similar to the proof of Theorem \ref{thm.suponslices} as soon as we have the following technical result.

\begin{proposition}[Solutions with gradient decay]\label{y12slices.prop}
Suppose that $u\in W^{1,2}_{\loc}(\reu)$ is a solution of $\cL u=0$ in $\reu$ and that $\cL 1 \neq 0$ on some box $I= Q\times (t_1,t_2)\subset \reu$. Further, assume that $\sup_{t>0}\| \nabla u(t)\|_{\ltrn} <\infty$, and that $\lim_{t\to\infty} \| \nabla u(t)\|_{\ltrn} =0$ (see (\ref{eq.ut})). Then $u(t)\in \yotn$ for every $t>0$.
\end{proposition}

\begin{proof} 
{\bf Step 1.} There exists a constant $c\in\bb C$ such that for all $t>0$, $u(\cdot,t)-c\in Y^{1,2}(\bb R^n)$.

To see this, first note that by the Sobolev embedding, there exists a function $f:(0,+\infty)\ra\bb C$ such that for each $t>0$, $u(\cdot,t)-f(t)\in Y^{1,2}(\bb R^n)$. We must show that $f$ is identically a constant. Since (see the proof of Theorem 1.78 in \cite{MZ}) for each $t>0$ we have that $f(t)=\lim_{R\ra\infty}\dashint_{B(0,R)}u(\cdot,t)$, it can be shown by the Sobolev embedding and considering the difference quotient $\frac{u(\cdot,t+h)-u(\cdot,t)}h$ that $f$ is differentiable and that $f'(t)\equiv0$ for all $t>0$. It follows that $f$ is a constant, as desired.

{\bf Step 2.} For the box $I\subset \reu$ as in the hypotheses, it holds that
\begin{equation*}
\dint_I |u^R |^{2^*}  \to 0 \qquad \textup{ as } R\to \infty,
\end{equation*}
where $u^R(\cdot, \cdot)= u(\cdot, \cdot+R)$. 

This is the crucial step. We set $p=2^*$ and  $u_I^R= |I|^{-1}\dint_{I} u^R$ for ease of notation. By the Poincar\'e-Sobolev inequality, we see that
\begin{equation}\label{sec6eq1.eq}
\| u^R-u^R_I\|_{L^p(I)} \lesssim \|\nabla u^R\|_{L^2(I)} \to 0, \quad \textup{ as } R\to \infty,
\end{equation}
where we used the definition of $u^R$ and the decaying condition of the gradient. In particular, we have that $u^R-u^R_I\to 0$ in $Y^{1,2}(I)$, so that $\cL(u^R-u^R_I)\to 0$ in $I$, which implies that for every  $\varphi\in C_c^\infty(I)$, the limit
\begin{equation*}
-u_I^R \dint_I B_1\cdot \overline{\nabla\varphi}  = \dint_I\big[ (A\nabla(u^R-u^R_I) + B_1(u^R-u_I^R))\cdot \overline{\nabla \varphi} + B_2 \cdot \nabla (u^R-u^R_I) \overline{\varphi}\big]   \to 0
\end{equation*}
holds. Since $\cL 1 \neq 0$ in $I$,  for some $\varphi_0\in C_c^\infty(I)$ we have that $
\dint_I B_1  \cdot \overline{\nabla \varphi_0} \neq 0$, whence $u_I^R\to 0$ as $R\to \infty$. The claim now follows by using this result in \eqref{sec6eq1.eq}. Notice that this argument holds just as well for any box $J$ containing $I$, in particular it holds for $\frac32I$.

{\bf Step 3.} For $Q\subset \rn$, $t\in (t_1,t_2)$ as in the hypotheses, we have that
\begin{equation*}
\int_Q |\trt u^R |^{p}  \to 0, \qquad \textup{ as } R\to \infty.
\end{equation*}

This is a consequence of Step 2 and the definition of the trace: For any $\phi\in C_c^\infty(Q)$ and $\eta\in C_c^\infty(t_1,t_2)$ with $\eta (s)=1$ near $t$, we set $\Phi:=\phi\eta\in C_c^\infty(I)$ and estimate
\begin{multline*}
|(\trt u^R, \phi)|=\Big| \dint_{\ree_+}( \dno u^R \Phi + u^R\dno \Phi) \Big|\\
 \leq \| \dno u^R\|_{Y^{1,2}(I)}\| \Phi\|_{L^{p'}(I)} + \| u^R\|_{Y^{1,2}(I)} \| \dno \Phi\|_{L^{p'}(I)}\\
 \lesssim_{\eta,\eta'} (\| \dno u^R\|_{Y^{1,2}(I)} + \| u^R\|_{Y^{1,2}(I)} ) \| \phi\|_{L^{p'}(I)}. 
\end{multline*}
The claim now follows  by the Caccioppoli  inequality; to wit,
\begin{equation*}
\| \dno u^R\|_{L^p(I)} + \| \nabla \dno u^R\|_{L^2(I)} \lesssim_{|I|} \sup_{s> t_2+R} \| \nabla u(s)\|_{\ltrn}\to 0 \quad \textup{ as } R\to \infty,
\end{equation*}
using the fact that $p<\frac{2n}{n-2}$.

We now conclude the proof: By Step 1, we can place $\tr_s (u-c)\in \yotn$ for all $s>0$. By Sobolev's inequality and the hypotheses, $\| \tr_s u-c\|_{\yotn} \to 0$ as $s\to \infty$. On the other hand, by Step 3, we have that $\tr_s u \to 0$ in $L^{p}(Q)$, so that $c=0$ and the desired result follows.\end{proof}

A quick  application of Theorem \ref{thm.suponslices2a} to the improvement of (\ref{eq.l2slicegradsgrad})  will be useful for the Dirichlet problem:

\begin{corollary}[Improvement to slice estimate]\label{cor.l2slicegradsgradk=1} The estimate (\ref{eq.l2slicegradsgrad}) holds true for $k=1$. In particular, (\ref{eq.traveldownsgrad2}) holds true for $k=1$ as well.
\end{corollary} 

We can also, very similarly, prove

\begin{theorem}[$L^2-$sup on slices]\label{thm.suponslicesd} Suppose that $u\in W^{1,2}_{\loc}(\bb R^{n+1}_+)$, $\cL u=0$ in $\bb R^{n+1}_+$, and that $u$ converges to $0$ in the sense of distributions. Furthermore, assume that $\||t\nabla u\||<\infty$. Then, for every $\tau>0$,
\begin{equation}\label{eq.suponslicesd}\nonumber
\Vert\Tr_{\tau} u\Vert_{L^2(\bb R^n)}\lesssim \int_\tau^\infty \int_\rn t|\nabla u|^2\, dxdt\lesssim\||t\nabla u\||
\end{equation}
where the implicit constant is independent of $\tau$ and $u$.
\end{theorem}

In the second paper, we will establish uniqueness under some weak background hypotheses. For this reason, we give two definitions and make an observation.
\begin{definition}[Good $\mathcal N / \mathcal R$ solutions]
We say that $u \in W^{1,2}_{\loc}(\ree_+)$ is a \emph{good $\mathcal N/ \mathcal R$ solution} if 
$\cL u = 0$ in $\ree_+$ in the weak sense, $u \in \sltp$ (see Definition \ref{def.slicespace} for the   slice spaces $S^2_+$ and $D^2_+$), and
$\partial_tu_\tau \in Y^{1,2}(\ree_+)$ for all $\tau > 0$, where $u_\tau(\cdot, \cdot) := u(\cdot, \cdot + \tau)$.
\end{definition}

\begin{definition}[Good $\mathcal D$ solutions]
We say that $u \in W^{1,2}_{\loc}(\ree_+)$ is a \emph{good $\mathcal D$ solution} if 
$\cL u = 0$ in $\ree_+$ in the weak sense, $u \in \dltp$ and
$u_\tau \in Y^{1,2}(\ree_+)$ for all $\tau > 0$, where $u_\tau(\cdot, \cdot) := u(\cdot, \cdot + \tau)$.
\end{definition}

As an immediate consequence of Theorems \ref{thm.suponslices2a} and \ref{thm.suponslicesd} we exhibit
\begin{corollary}\label{alwaysmakesgoodsoln.cor}
Let $u\in W^{1,2}_{\loc}(\reu)$ satisfy $\cL u=0$ in $\reu$.
\begin{enumerate}[i)]
\item If $\|| t\nabla\partial_t u\||<\infty$ and  $\lim_{t\to\infty} \nabla u(t)=0$ in the sense of distributions (see (\ref{eq.ut})), then either $u$ is a good $\mathcal N$/$\mathcal R$ solution (in the case that $\cL 1\neq0$ in $\reu$), or $u-c$ is a good $\mathcal N$/$\mathcal{R}$ solution for some constant $c$ (in the case that $\cL =0$ in $\reu$).
\item If $\||t\nabla u\||<\infty$ and $\lim_{t\to \infty} u(t) =0$ in the sense of distributions, then $u$ is a good $\mathcal D$ solution.
\end{enumerate}
\end{corollary}

\bibliography{bhlmprefs}
\bibliographystyle{alpha-sort-max}
 
\end{document}